\documentclass{article}
  \usepackage{amsmath}
  \usepackage{amssymb}
  \usepackage{harpoon}
  \usepackage{latexsym,color}
  \usepackage{ntheorem}
  \usepackage{mathrsfs}
  \usepackage[numbers,sort&compress]{natbib}

  \allowdisplaybreaks[1]
  \newcommand{\beqa}{\begin{eqnarray}}
  \newcommand{\eeqa}{\end{eqnarray}}
  
  \newcommand{\R}{\mathbb{R}}
  \newcommand{\N}{\mathbb{N}}
  \newcommand{\SP}{\mathbb{S}}
  \newcommand{\si}{\sigma}
  \newcommand{\dd}{\mathrm{d}}

  \newcommand{\no}{\nonumber}

  \newcommand{\inttts}{\iiint\limits_{\R^3\times\R^3\times\SP^2}}

  \newcommand{\beas}{\begin{eqnarray*}}
  \newcommand{\eeas}{\end{eqnarray*}}
  \newcommand{\p}{\partial}

  \theoremseparator{.}
\newcommand{\ld}{\lambda}

\newcommand{\vp}{\varphi}
\newcommand{\vh}{\varrho}
\newcommand{\vep}{\varepsilon}
\newcommand{\og}{\omega}
\newcommand{\Og}{\Omega}
\newcommand{\sg}{\sigma}

\newcommand{\dt}{\delta}

\newcommand{\fr}{\frac}
\newcommand{\wt}{\widetilde}
\newcommand{\intt}{\int\!\!\!\!\int}
\newcommand{\inttt}{\int\!\!\!\!\int\!\!\!\!\int}

\newcommand{\mR}{{\mathbb R}}
\newcommand{\mN}{{\mathbb N}}
\newcommand{\mQ}{{\mathbb Q}}

\newcommand{\mS}{{\mathbb S}}
\newcommand{\bR}{{\mathbb R}^3 }
\newcommand{\bS}{{\mathbb S}^2 }

\newcommand{\bRR}{{\bR}\times{\bR}}

\newcommand{\bRS}{{\bR}\times {\mathbb S}^2 }
\newcommand{\bRRS}{{\bRR}\times{\mathbb S}^2 }

\newcommand{\la}{\langle}
\newcommand{\ra}{\rangle}

\newcommand{\bea}{\begin{eqnarray}}
\newcommand{\eea}{\end{eqnarray}}

\newcounter{thm}
\newtheorem{theorem}{Theorem}[section]
\newtheorem{proposition}[theorem]{Proposition }
\newtheorem{definition}[theorem]{Definition }
\newtheorem{lemma}[theorem]{Lemma}
\newtheorem{corollary}[theorem]{Corollary}
\newtheorem{remark}[theorem]{Remark }

\newenvironment{proof}{{\bf Proof. }}{$\hfill\Box$}
\newcounter{myequation}[section]
\newenvironment{myequation}{\stepcounter{myequation}\begin{equation}}{\end{equation}}
\newenvironment{myeqnarray}{\stepcounter{myequation}\begin{eqnarray}}{\end{eqnarray}}
\newcommand{\be}{\begin{myequation}}
\newcommand{\ee}{\end{myequation}}
\newcommand{\bes}{\begin{myeqnarray}}
\newcommand{\ees}{\end{myeqnarray}}
\newcommand{\bepf}{\begin{proof}}
\newcommand{\eepf}{\end{proof}}
\DeclareMathOperator*{\esssup}{ess\,sup}

\makeatletter 
\@addtoreset{equation}{section}
\makeatother  

  \author{Wenyi Li\, and\, Xuguang Lu}
  \title{Global existence of solutions of
the Boltzmann equation
for Bose-Einstein particles
with anisotropic initial data}
\begin{document}
\bibliographystyle{plain}
  \date{}
  \maketitle

{\bf Abstract} In this paper we prove the global in time existence and uniqueness of solutions of
the spatially homogeneous Boltzmann equation for Bose-Einstein particles for the hard sphere model for bounded
anisotropic initial data. The main idea of our proof is as follows: we first establish
an intermediate equation which is closely related to the original equation
and is relatively easily proven to have global in time and unique solutions, then we use
the multi-step iterations of the collision gain operator  to obtain
a desired uniform $L^{\infty}$-bound for the solution of the intermediate equation so that
it becomes the solution of the original equation.

{\bf Keywords} Boltzmann equation;
Bose-Einstein particles;
Global existence;
Iteration of collision operator;
Hard sphere model
\tableofcontents

\baselineskip 18.5pt
  \section{Introduction}

The Boltzmann equation for Bose-Einstein particles describes time evolution of dilute
Bose gases. Derivations of such quantum
Boltzmann equations can be found for instance in  \cite{Nordheim1928On}, \cite{Uehling1933Transport},
\cite{Chapman1939The}, \cite{d2005on}, \cite{erd2004on}.
 In this paper we consider the case of spatially homogeneous equation
and study the global in time existence and uniqueness of solutions for anisotropic initial data.
The spatially homogeneous Boltzmann equation for Bose-Einstein particles under consideration is given in $\si$-representation by
  \beqa\label{Equation1}
  \frac{\partial f}{\partial t}= \iint_{\R^3\times\SP^2}B(v-v_*,\si)\big[f'f'_{*}(1+\hbar^3 f+\hbar^3 f_{*})-ff_{*}(1+\hbar^3 f'+\hbar^3 f'_{*})\big]\dd \si\dd v_{*}
  \eeqa
  with $(t,v)\in [0,\infty)\times \R^3,$  where $\hbar>0$ is the Planck's constant, $f=f(t,v)$ is the density of the number of particles at time $t\in [0,\infty)$ with the velocity $v\in \R^3$,
  $f'=f(t,v')$, $f_{*}=f(t,v_*)$, $f'_{*}=f(t,v_{*}')$,  and $v,v_*$  and  $v',v_*'$  are velocities of two particles before and after their collision, where
  \beqa\label{sigma representation}
  v'=\fr{v+v_*}{2}+\fr{|v-v_*|}{2}\si,\qquad
  v'_{*}=\fr{v+v_*}{2}-\fr{|v-v_*|}{2}\si.
  \eeqa
 Recall that the equation (\ref{Equation1}) can also be written in $\og$-representation by
  \beqa
  \frac{\partial f}{\partial t}= \iint_{\R^3\times\SP^2}\wt{B}(v-v_*,\og)\big[f'f'_{*}(1+\hbar^3 f+\hbar^3 f_{*})-ff_{*}(1+\hbar^3 f'+\hbar^3 f'_{*})\big]\dd \og\dd v_{*}
  \eeqa where
 \beqa \label{omega representation1} && v'=v-\la v-v_*,\og\ra\og,\qquad
  v'_{*}=v+\la v-v_*,\og\ra \og \\
\label{omega representation2} &&\si={\bf n}-2\la{\bf n}, \og\ra\og,\quad
 \wt{B}(v-v_*,\og)=2|\la {\bf n}, \og\ra| B(v-v_*,\si) \eeqa
 ${\bf n}=\fr{v-v_*}{|v-v_*|}$, and $B(v-v_*,\si)$ is given in the equation (\ref{Equation1}) (see e.g.
Chapter 2.4 in \cite{Villani2002Chapter}).
In the rest part of this paper we will always use the $\si$-representation of the equation (\ref{Equation1}).

The function $B(v-v_*,\si)$ in the equation (\ref{Equation1}) is called the collision kernel which,
according to \cite{Benedetto2007classical},\cite{d2005on},\cite{erd2004on} for quantum Boltzmann equation, takes the form
  \beqa\label{collision kernel 1}
B(v-v_*,\sg)=\fr{|v-v_*|}{\hbar^4}\Big[\widehat{\Phi}\Big(\fr{|v-v'|}{\hbar}\Big)+
\widehat{\Phi}\Big(\fr{|v-v_*'|}{\hbar }\Big)\Big]^2
\eeqa
where the radially symmetric function $\widehat{\Phi}(|\xi|):=\widehat{\Phi}(\xi)=
\int_{\R^3}\Phi(x)e^{-{\rm i}\la \xi, x\ra}\dd x$ is the Fourier transform of an interacting potential ${\Phi}(x)=\Phi(|x|)$ which is assumed (as in many cases) to be radially symmetric.
From (\ref{collision kernel 1}) and (\ref{sigma representation}) one sees that $B(v-v_*,\si)$ is a function of $(|v-v_*|, |\la v-v_*,\si\ra|),$ i.e.
\beqa\label{about v-v*}
B(v-v_*,\si)=B_1(|v-v_*|,|\la v-v_*,\si\ra|),\quad (v,v_*,\si) \in{\mR}^3\times{\mR}^3\times\SP^2
\eeqa
with a nonnegative measurable function $B_1$ on $\R_{\ge 0}\times {\mR}_{\ge 0}$ determined by
$\widehat{\Phi}(|\cdot|)$.

Investigation experience shows that the most possible case for obtaining a global
solution for anisotropic initial data is the hard sphere model or the
asymptotically hard sphere model:
\beqa\label{BB} B(v-v_*,\si)\sim |v-v_*|\qquad {\rm when}\quad  |v-v_*| >> 1.\eeqa
According to (\ref{collision kernel 1}), the precise meaning of (\ref{BB}) is that the function $\widehat{\Phi}(|\xi|)$
satisfies the following condition:
\beqa\label{condition of Phi}
a_0 \fr{|v-v_*|^{\beta}}{1+|v-v_*|^{\beta}}
 \le
 \big(\widehat{\Phi}(|v-v'|)+
\widehat{\Phi}(|v-v_*'|)\big)^2
 \le b_0
\eeqa
where $\beta \ge 3$ and  $0<a_0\le b_0<\infty$ are constants independent of $\hbar$.
This includes a physical case where the interaction potential $\Phi$
is equal to a Dirac $\dt$-function $\dt(x)$ plus a small attractive
force's potential $-U(x)$:
\beqa \label{potential}\Phi(x)=\delta(x)-U(x), \qquad x\in \R^3. \eeqa For instance one may take
$$ U(x)=U(|x|)=\fr{1}{4\pi|x|}\exp(-|x|),\qquad x\in {\mathbb R}^3$$
which is Yukawa potential in $\R^3$; its Fourier transform is
 \beas
 \widehat{U}(\xi)=\widehat{U}(|\xi|)=\fr{1}{1+ |\xi|^2},\qquad \xi\in {\mathbb R}^3.
 \eeas
From (\ref{potential}) we have $\widehat{\Phi}(|\xi|)=1-\widehat{U}(\xi)=1-\fr{1}{1+ |\xi|^2}$
so that
\beas
 \fr{1}{8}
 \fr{|v-v_*|^4}{1+|v-v_*|^4}
 \le
 \big(\widehat{\Phi}(|v-v'|)+
\widehat{\Phi}(|v-v_*'|)\big)^2
 \le 4,
 \eeas
 which satisfies (\ref{condition of Phi}) with $\beta=4$, $a_0=\fr{1}{8}$ and $b_0=4$.
Of course if $U(x)=0$, i.e. if $\Phi(x)=\dt(x)$,  then $B(v-v_*,\sg)=\fr{1}{{\hbar}^4}|v-v_*|$
which corresponds to the hard sphere model.

To simplify notations we may set $\hbar =1$. That is, in the rest of this paper we need only to consider the
equation (\ref{Equation1}) for the scale normalized case $\hbar=1$.  In fact, by the scaling
transforms
\beqa\label{scaling transform}
\wt{f}(t,v)=\hbar^3 f(\hbar^3 t, v),\qquad f(t,v)=\hbar^{-3}\wt{f}(\hbar^{-3} t, v),\qquad
\Psi(|x|)=\hbar{\Phi}(|\hbar x|) \eeqa
we have $\widehat{\Psi}(|\xi|)=\widehat{\Psi}(\xi)=\fr{1}{{\hbar}^2}\widehat{\Phi}(\fr{\xi}{\hbar})=
\fr{1}{{\hbar}^2}\widehat{\Phi}(\fr{|\xi|}{\hbar})$ so that it is easily checked that $f$ is a solution of the equation (\ref{Equation1}) if and only if $\wt{f}$ is a
solution of the scale normalized equation:
\beqa\label{Equation}
\frac{\partial f}{\partial t}
= \iint_{\R^3\times\SP^2}B(v-v_*,\si)\big[f'f'_{*}(1+ f+ f_{*})-ff_{*}(1+ f'+ f'_{*})\big]\dd \si\dd v_{*}
\eeqa
with
\beqa\label{collision kernel}
B(v-v_*,\sg)=|v-v_*|\big(\widehat{\Psi}(|v-v'|)+
\widehat{\Psi}(|v-v_*'|)\big)^2.
\eeqa
Correspondingly,  the condition ({\ref{condition of Phi}}) is rewritten in terms of
$\widehat{\Psi}$ as
\beqa\label{lpbd}
\fr{a_0}{{\hbar}^4}\cdot \fr{|v-v_*|^{\beta}}{{\hbar}^{\beta}+|v-v_*|^{\beta}}
 \le
 \big(\widehat{\Psi}(|v-v'|)+
\widehat{\Psi}(|v-v_*'|)\big)^2
 \le \fr{b_0}{{\hbar}^4}.
\eeqa
Let us denote
\beqa\label{abh} a=\fr{a_0}{{\hbar}^4},\quad  b=\fr{b_0}{{\hbar}^4}, \quad {\rm and\,\, assume\,\,
that}\quad  0<\hbar\le 1.\eeqa
Then from (\ref{lpbd}) and (\ref{collision kernel}) we have
\beqa&&
a\fr{|v-v_*|^{\beta}}{1+|v-v_*|^{\beta}}
 \le
 \big(\widehat{\Psi}(|v-v'|)+
\widehat{\Psi}(|v-v_*'|)\big)^2
 \le b,\no\\
 &&\label{condition of kernel}
a \fr{|v-v_*|^{\beta}}{1+|v-v_*|^{\beta}}|v-v_*|
 \le
 B(v-v_*,\si)
 \le b|v-v_*|.
\eeqa

 One of the main difficulties in proving the global in time existence of solutions of
Eq.(\ref{Equation}) with general $L^1$ initial data is the divergence of the cubic term (see e.g. \cite{Lu2004On}):
\beqa\label{divergence}\sup_{f\ge 0,\, \|f\|_{L^1_2}\le 1}\iiint_{\R^3\times R^3\times\SP^2}B(v-v_*,\si)f(v)f(v_{*})f(v') \dd \si\dd v_{*}\dd v=\infty \eeqa
(here $\|\cdot\|_{L^p_s}$ is defined below).
Another main difficulty comes from the low temperature effect
which yields the Bose-Einstein condensation and that the regular part of the
 equilibrium (i.e. the Bose-Einstein distribution) is unbounded near the origin; these
 together with the result of convergence to equilibrium
(see \cite{Lu2005On}) imply that for low temperature there must be no such a global solution that is bounded on $[0, \infty)\times{\mR}^3$.

So far, basic results for Eq.(\ref{Equation}) on global in time existences, singular behavior, long time behavior, kinetics of Bose-Einstein condensation, etc. are all concerned with solutions for
isotropic initial data (which implies that the solutions are also isotropic), see for instance
\cite{Escobedo2003Homogeneous},\cite{M2008Singular},\cite{Escobedo2015Finite},\cite{Lu2005On},\cite{Lu2014The},\cite{Spohn2010Kinetics}.
In a recent work  \cite{Briant2016On}, Briant and Einav considered a wide class of bounded solutions of the equation (\ref{Equation}) and proved the local in time existence, uniqueness, stability  and
 moment production of solutions of Eq.(\ref{Equation}) without the isotropic assumption on the initial data.

 Our aim of this paper is to prove the global in time existence and uniqueness of
 solutions of Eq.(\ref{Equation}) for bounded anisotropic initial data $f_0$ satisfying  (\ref{condition}) (see Theorem \ref{main results} below). As mentioned above, this
 must belong to the case of high temperature.

In order to state our main result we need to introduce the definition of solutions of
Eq.(\ref{Equation}) and some related notations.

\subsection{Definition of solutions}
Our working function spaces
are the usual weighted Lebesgue $L^p$ spaces
$L^p_s(\R^3)$ defined by
\beas f\in L^p_s(\R^3)\,\Longleftrightarrow\,
\|f\|_{L^p_s}:=\Big(\int_{\R^3}\la v\ra^{sp} |f(v)|^{p} \dd v\Big)^{\fr{1}{p}}<\infty
\eeas
where $1\le p< \infty$ , $0\le s <\infty$ and $\la v\ra^{s}=(1+|v|^2)^{s/2}.$
For $p=\infty$, $0\le s<\infty$ we define
\beas
&&\|f\|_{L^{\infty}_s}=\esssup_{v\in{\bR}}\la v\ra^s |f(v)|\quad
{\rm and}\quad
f\in L^{\infty}_s(\R^3)\,\Longleftrightarrow\,
\|f\|_{L^{\infty}_s}<\infty.
\eeas
As usual let $Q(f)$ denote the collision integral in the right hand side of Eq.(\ref{Equation}):
\beqa\label{Q(f)}
&&Q(f)(v)=
\intt_{{\bRS}}B(v-v_*,\sg)\big[f'f_*'
(1+ f+ f_*)-ff_*
(1+ f'+f_*')\big]{\rm d}\sg{\rm d}v_*,
\quad  v\in {\mR}^3.
\eeqa
And we denote
$Q(f)(t,v):=Q(f(t,\cdot))(v)$  for the case $f=f(t,v)$.

\begin{definition}\label{definition of solution}
Given $0\le f_0\in L^1_2({\mR}^3)\cap C(\R^3)$.
We say that a nonnegative function $f\in C([0,\infty)\times{\R}^3)\cap
L^{\infty}([0,\infty); L^1_2(\R^3))$
is a solution to Eq.$(\ref{Equation})$
with the initial datum $f(0,\cdot)=f_0$,
if $f$ satisfies the following {\rm (i),(ii):}

{\rm (i)}
for any $v\in \R^3$, the function $t\mapsto f(t,v)$ belongs to $C^1([0,\infty))$,

{\rm (ii)} $f$ satisfies Eq.$(\ref{Equation})$ on $[0,\infty)\times \R^3$, i.e.
\beqa\label{diffequation}\fr{\p }{\p t}f(t,v)=Q(f)(t,v)\qquad \forall\, (t,v)\in [0,\infty)\times \R^3.\eeqa
Furthermore, if $f$ conserves the mass, momentum and energy, i.e. if
\beqa\label{MEconservation} \int_{\R^3}(1,
v,|v|^2)f(t,v){\rm d}v= \int_{\R^3}(1, v,|v|^2)
f_0(v){\rm d}v\qquad \forall\, t\in [0,\infty)\eeqa
then  $f$ is called a conservative solution.
\end{definition}

\subsection{Moments and kinetic temperature}

Moments $M_k(f)$ of order $k\in [0, s]$ for $0\le f\in L^1_s(\R^3)$
are defined by
\beas
M_k(f)=\int_{\R^3}|v|^k f(v) \dd v.
\eeas
For $0\le f\in L^1_2(\R^3)$,  $m M_0(f)$ and $\fr{m}{2}M_2(f)$ are the mass and
kinetic energy of a particle system per unit space volume, where $m$ is the mass of one particle.
Without confusion we also call $M_0(f)$ and $M_2(f)$ the mass and energy.
In this paper we always assume that initial data $0\le f_0\in L^1_2(\R^3)$ satisfy $M_0(f_0)>0$.
Denote $M_0=M_0(f_0),\, M_2=M_2(f_0)$. By the conservation of mass and energy, the kinetic temperature $\overline{T}$ of the particle system is defined by (see e.g.  Chapter 2 in \cite{Chapman1939The})
\beas
\overline{T}=\fr{1}{3k_{B}}\cdot \fr{m M_2}{M_0},
\eeas
and the critical temperature $\overline{T}_c$ corresponding to Eq.(\ref{Equation}) is given by (see e.g. \cite{Lu2005On} and references therein)
\beas
\overline{T}_{c}=\fr{m \zeta(5/2)}{2\pi k_{B}\zeta(3/2)}\Big(\fr{M_0}{\zeta(3/2)}\Big)^{2/3}
\eeas
where $k_B$ is the Boltzmann constant,
$\zeta(s)=\sum_{n=1}^{\infty}\fr{1}{n^s}\,(s>1)$ is the Riemann-Zeta function.
By calculation we have
\beqa\label{rate of Kinetic temperature}
\fr{\overline{T}}{\overline{T}_{c}}
=\fr{2\pi[\zeta(3/2)]^{5/3}}{3\zeta(5/2)}\cdot\fr{M_2}{{M_0}^{5/3}}.
\eeqa
As mentioned above, if $\overline{T}/\overline{T}_{c}\le 1$
(i.e. the case of low temperature), there is no bounded solution on $[0,\infty)\times{\mR}^3$.
The case of $\overline{T}/\overline{T}_{c}> 1$ but not too large is difficult.
While the case
 $\overline{T}/\overline{T}_{c}>> 1$ is easily proved to be necessary for having
 global bounded solutions for small initial data as concerned in this paper, see  {\bf Remark \ref{Remark 1.4}} below.

\subsection{Main result}
 Our main result of the paper is the following
 \begin{theorem}\label{main results} Suppose the collision kernel $B$ in $(\ref{collision kernel})$ satisfies $(\ref{condition of kernel})$ with $\beta\ge 3$.  Let $0\le f_0\in L^1_3({\bR})\cap L_{3}^{\infty}({\bR}) \cap C(\R^3)$ satisfies
 $\int_{\R^3}vf_{0}(v)\dd v=0, M_2(f_0)=\int_{\R^3}|v|^2f_{0}(v)\dd v>0,$
and
\beqa\label{condition}
     (\|f_0\|_{L^1}+ \|f_0\|_{L^{\infty}})
  \Big(\fr{\|f_0\|_{L_3^1}}{\min\{\|f_0\|_{L^1}, M_2(f_0)\}}\Big)^{4\beta}
  \le \fr{1}
  {2^{35\beta-11}}
  \fr{(4\beta+2)^{2\beta+1}}
  {(4\beta+4)^{2\beta+2}}\Big(\fr{a}{b}\Big)^{2(\beta+1)}
     \eeqa
$($which implies  $\overline{T}/\overline{T}_{c}>> 1)$.
     Then there exists a unique solution
     $f \in L^{\infty}([0,\infty); L^1_3({\bR})\cap L^{\infty}({\bR}) \cap C(\R^3))$ of
     Eq.$(\ref{Equation})$ with the initial datum $f(0,\cdot)=f_0$ and $f$ conserves  the mass, momentum and energy.

Moreover we have
\beqa\label{estimate infty}
\sup_{t\ge 0}{\|f(t)\|_{L^{\infty}}}
\le
    2^{35\beta}
    \Big(\fr{b}{a}\Big)^{2(\beta+1)}
    \fr{(2\beta+2)^{2\beta-2}}{(2\beta+3)^{2\beta-2}}
    \Big(\fr{\|f_0\|_{L_3^1}}{\min\{\|f_0\|_{L^1}, M_2(f_0)\}}\Big)^{4\beta}
   (\|f_0\|_{L^1}+ \|f_0\|_{L^{\infty}}).
\eeqa

\end{theorem}

 \begin{remark}
{\rm It is easily seen that for any $0\le g_0\in L^1_3({\bR})\cap L_{3}^{\infty}({\bR}) \cap C(\R^3)$ and any constant $\ld>0$, the initial datum $f_0:=\ld g_0$ satisfies the condition (\ref{condition})
 when $\ld>0$ is small enough. Thus the initial data satisfying the condition (\ref{condition}) exist extensively.
Besides, the assumption of zero mean-velocity $\int_{\R^3}vf_{0}(v)\dd v=0$ does not
lose generality. In fact, for the case $v_0:=\fr{1}{M_0}\int_{\R^3}vf_{0}(v)\dd v\neq 0$\,($M_0=\int_{\R^3}f_{0}(v)\dd v$), it is easily seen that
for the solution $\wt{f}(t,v)$ of Eq.(\ref{Equation}) with the initial datum $\wt{f}_0(v):=f_0(v+v_0)$ (which has zero mean-velocity), the $v$-translation
$(t,v)\mapsto f(t,v):=\wt{f}(t, v-v_0)$ is the solution of Eq.(\ref{Equation}) with the initial datum $f_0(v)$.}
\end{remark}

\begin{remark}\label{Remark 1.4}{\rm From the condition (\ref{condition}) we see that $\|f_0\|_{L^{\infty}}$ is less than
the right hand side of (\ref{condition}).
In the proof of Theorem \ref{main results} we will prove that this smallness of $\|f_0\|_{L^{\infty}}$ implies the very high temperature condition:
\beqa\label{Kinetic high change}
\fr{\overline{T}}{\overline{T}_{c}}
=\fr{2\pi[\zeta(3/2)]^{\fr{5}{3}}}{3\zeta(5/2)}\cdot\fr{M_2}{{M_0}^{\fr{5}{3}}}
\ge
2^{\fr{70\beta}{3}-\fr{19}{3}}
\fr{(4\beta+4)^{\fr{4\beta+4}{3}}}{(4\beta+2)^{\fr{4\beta+2}{3}}}
\Big(\fr{b}{a}\Big)^{\fr{4\beta+4}{3}}>>1.
\eeqa
In other words, the very high temperature condition is necessary for the global in time existence of such small bounded solutions of Eq.(\ref{Equation}).}
\end{remark}

\begin{remark}
{\rm
 From our proof of Theorem \ref{main results} it is easily seen that if we do not assume that the initial data $f_0$ are continuous,
i.e. if we only assume that $0\le f_0\in L^1_3(\R^3)\cap L^{\infty}_{3}({\R}^3)$, then
the above main result also holds true for mild solutions, where the definition of mild solutions is given in Definition \ref{mild solutions} below.
}
\end{remark}

\begin{remark}
  {\rm
From (\ref{condition of Phi}) and (\ref{abh}) (i.e. $a/b=a_0/b_0$) we see that our initial assumption (\ref{condition}) in Theorem \ref{main results} for the original equation (\ref{Equation1}) is equivalent to
$$
(\|f_0\|_{L^1}+ \|f_0\|_{L^{\infty}})
  \Big(\fr{\|f_0\|_{L_3^1}}{\min\{\|f_0\|_{L^1}, M_2(f_0)\}}\Big)^{4\beta}
  \le \fr{1}
  {2^{35\beta-11}}
  \fr{(4\beta+2)^{2\beta+1}}
  {(4\beta+4)^{2\beta+2}}\Big(\fr{a}{b}\Big)^{2(\beta+1)}
  \fr{1}{\hbar^3}$$
which
is easily satisfied for any given initial datum $f_0$ if $\hbar $ is small enough.
}
\end{remark}

\subsection{Strategy and organization of the paper}

In order to prove the global in time existence of solutions of Eq.(\ref{Equation}), we consider two types of approximate equations: cutoff equations and the intermediate equation. In cutoff equations the collision kernel $B$ is cut off as $B_n=B\wedge n$ and the
solutions $f$ in the collision integrals are cut off as $f\wedge n,\, f \wedge K$ (with constants $n>0,\, K>0$). Here
$$x\wedge y=\min\{x, y\},\quad x,y\in {\R}.$$
Then with $K$ fixed we use the $L^1$ relative compactness of  mild solutions $\{f^n\}_{n=1}^{\infty}$ of cutoff equations to obtain a mild solution $f$ of the intermediate equation.
In the intermediate equation the collision kernel $B$ is the original kernel, but the mild solution $f$ in the collision integrals are partly modified as $f \wedge K$. We finally use carefully multi-step iterations of the collision gain operator $Q^+(\cdot,\cdot)$ with a suitable choice of $K$ to obtain the uniform $L^{\infty}$ estimate $f\le K$
for the solution $f$ of the intermediate equation so that $f$ is the solution of Eq.(\ref{Equation}).

The paper is organized as follows:
In Section 2 we give some basic definitions and then present some technical lemmas and propositions which will be used in the subsequent sections.
Section 3 is the proof of existence and $L^1$ relative compactness of mild solutions $f^n$ of cutoff equations.
In Section 4 we prove the existence of a bounded mild solution $f$ of  the intermediate equation. The proof of Theorem \ref{main results} is concluded in Section 5.
\vskip3mm
 \section{Some properties of collision operations}
This section is a preparation for proving our main result.
We begin by recalling a few elementary properties of collision integrals
which are used in deriving basic equalities and estimates.
From the $\og$-representation (\ref{omega representation1}) and the
identity $|v'|^2+|v_*'|^2=|v|^2+|v_*|^2$
one sees that for any
$\og\in {\mS}^2$, $(v,v_*)\mapsto (v',v_*')$ is an orthogonal linear transformation
on ${\mR}^6$. It is this property that makes the proof of some elementary properties of collision integrals relatively easy. On the other hand, the $\sg$-representation (\ref{sigma representation})
has the advantage that gives a nice structure for the collision integrals.
The two representations however are equivalent in representing collision integrals. In fact
it is not difficult to prove the following identity (for all $F$ which are nonnegative measurable or satisfy
required integrability):
\beqa\label{sigma omega}\int_{{\mS}^2} B(v-v_*,\sg)F(v',v_*'){\rm d}\sg=
\int_{{\mS}^2}\wt{B}(v-v_*,\og)F(v',v_*'){\rm d}\og,\quad v,v_*\in \R^3 \eeqa
where $(v',v_*')$ in the left hand side and in the right hand side are
given by $\sg$-representation (\ref{sigma representation})  and $\og$-representation
(\ref{omega representation1}) respectively. This property allows us to translate
some elementary properties of collision integrals with the $\sg$-representation  into
those with the $\og$-representation so that they can be proven rigorously (see also
Chapter 2.4 in \cite{Villani2002Chapter}). For instance, applying (\ref{sigma omega}) and the $\og$-representation one deduces
the following general identity with the $\sg$-representation (which is often used in
deriving fundamental properties of collision integrals):
\beqa\label{identity}
&& \inttt_{{\mR}^3\times{\mR}^3\times {\mS}^2}B(v-v_*,\sg)F(v',v_*',v,v_*)\dd\sg \dd v_* \dd v\no\\
&&=\inttt_{{\mR}^3\times{\mR}^3\times {\mS}^2} B(v-v_*,\sg)F(v,v_*,v',v_*')\dd \sg \dd v_*\dd v\qquad ({\rm exchange\,\, prime})\eeqa
where $(v',v_*')$ is given by the $\sg$-representation (\ref{sigma representation}) and
$F$ are nonnegative Lebesgue measurable functions on
${\mR}^3\times{\mR}^3\times {\mR}^3 \times {\mR}^3$ or satisfy required integrability.

\subsection{Collision integral operators and approximate equations}
In this subsection, we introduce definitions of collision integral operators, two types of approximate equations and mild solutions.

\begin{definition}\label{def 2.1}
 Fix $K>0, n\in{\mN}$. Suppose the collision kernel $B$ in $(\ref{collision kernel})$ satisfies $(\ref{condition of kernel})$.
For any nonnegative measurable functions $f, g$ on $\R^3$, we define for any $v\in \R^3$
\beqa\label{Q+ bilinear} Q^{+}(f,g)(v)=\iint\limits_{\R^3\times\SP^2}B(v-v_*,\si)f'{g}'_{*}\dd \si \dd v_{*}\eeqa
 which is called the (bilinear) collision gain operator, and
\beqa&&\label{Q+}
Q^{+}(f)(v)=
\intt_{{\bRS}}
B(v-v_*,\sg)f'f_*'
(1+ f+ f_*){\rm d}\sg{\rm d}v_*,\\
&&\label{Q-}
Q^{-}(f)(v)=
\intt_{{\bRS}}
B(v-v_*,\sg)
ff_*
(1+ f'+ f_*')
{\rm d}\sg{\rm d}v_*,
\\
&&\label{QK+}
Q_{K}^{+}(f)(v)=
\intt_{{\bRS}}
B(v-v_*,\sg)f'f_*'
(1+ f\wedge K+ f_*\wedge K){\rm d}\sg{\rm d}v_*,\\
&&\label{Lft}
  L_{K}(f)(v)=\iint\limits_{\R^3\times\SP^2}B(v-v_*,\si)f_{*}(1+ f'\land K+ f'_{*}\land K)\dd \si\dd v_{*},
\\
&&\label{QK-}
Q_{K}^{-}(f)(v)=f(v)L_K(f)(v)=
\intt_{{\bRS}}
B(v-v_*,\sg)
ff_*
(1+ f'\wedge K+ f_*'\wedge K)
{\rm d}\sg{\rm d}v_*,
\\
&&Q_{K}(f)(v)=Q_{K}^{+}(f)(v)-Q_{K}^{-}(f)(v)\quad {\rm for}\,\, f\in  L^1_1(\R^3),\label{QK}
\\
&&\label{Qnk+}
Q^{+}_{n,K}(f)(v)=
\intt_{{\bRS}}
B_n(v-v_*,\sg)(f'\wedge n)(f_*'\wedge n)
(1+ f\wedge K+ f_*\wedge K){\rm d}\sg{\rm d}v_*,
\\
&&\label{Qnk-}
Q^{-}_{n,K}(f)(v)=
\intt_{{\bRS}}
B_n(v-v_*,\sg)
(f\wedge n)(f_*\wedge n)
(1+ f'\wedge K+ f_*'\wedge K)
{\rm d}\sg{\rm d}v_*,\\
&&Q_{n,K}(f)(v)=Q^{+}_{n,K}(f)(v)-Q^{-}_{n,K}(f)(v)\label{Qnk}
\eeqa
where $B_n(v-v_*,\sg)=B(v-v_*,\sg)\wedge n$.
 \end{definition}

\begin{definition}
 Given any $K>0, n\in \N$. Suppose the collision kernel $B$ in $(\ref{collision kernel})$ satisfies $(\ref{condition of kernel})$.
Our cutoff equation of Eq.$(\ref{Equation})$ mentioned above is defined by
\begin{equation}\label{approximate equation 1}
\fr{\p }{\p t}f(t,v)
=Q_{n,K}(f)(t,v),\quad (t,v)\in[0,\infty)\times {\R}^3,
\end{equation}
and the intermediate equation of Eq.$(\ref{Equation})$ is defined by
\beqa\label{approximate equation 2}
\fr{\p }{\p t}f(t,v)
=Q_{K}(f)(t,v),\quad (t,v)\in[0,\infty)\times {\R}^3.
\eeqa
Here $Q_{n,K}(f)(t,v)=Q_{n,K}(f(t,\cdot))(v),\, Q_{K}(f)(t,v)=Q_{K}(f(t,\cdot))(v)$.
\end{definition}

 \begin{definition}\label{mild solutions}
 Let $n\in \mathbb{N}$, $B_n=B\wedge n$ with $B$ the collision kernel satisfying $(\ref{condition of kernel})$.  Let  $Q(\cdot), Q^{\pm}(\cdot)$, $Q_{K}(\cdot), Q_{K}^{\pm}(\cdot)$,
 $Q_{n,K}(\cdot)$, $Q_{n,K}^{\pm}(\cdot)$ be the collision operators defined in  $(\ref{Q(f)})$, $(\ref{Q+})-(\ref{Qnk})$ respectively. Let $Q_*(\cdot),Q_*^{\pm}(\cdot)$ be one of the three couples
 $Q(\cdot), Q^{\pm}(\cdot)$; $Q_{K}(\cdot), Q_{K}^{\pm}(\cdot)$;
 $Q_{n,K}(\cdot)$, $Q_{n,K}^{\pm}(\cdot)$.
 Given any $0\le f_0\in L^1_2({\mR}^3)$,
we say that a nonnegative measurable function $(t, v)\mapsto f(t,v)$ on $[0,\infty)\times \R^3$
is a mild solution of the equation
\beas&&
\fr{\p }{\p t}f(t,v)
=Q_{*}(f)(t,v),\quad (t,v)\in [0,\infty)\times {\R}^3
\eeas
with the initial datum $f(0,\cdot)=f_0$,
if $f$ satisfies
$f\in L^{\infty}([0,\infty); L^1_2(\R^3))$ and there is a null set $Z\subset {\mR}^3$,
which is independent of $t$, such that for all $t\in [0,\infty)$ and all $v\in {\mR}^3\setminus Z$,
\beas&&
\int_{0}^{t}Q_{*}^{\pm}(f)(\tau,v){\rm d}\tau<\infty,
\\
&&
f(t,v)=f_0(v)+\int_{0}^{t}Q_{*}(f)(\tau,v){\rm d}\tau.\eeas
Furthermore if $f$ also conserves the mass, momentum and energy, i.e. if
$f$ satisfies  $(\ref{MEconservation})$, then
$f$ is called a conservative mild solution.
For any given $0<T<\infty$, by replacing $[0,\infty)$ with $[0, T]$ we also define
$($conservative$)$
mild solutions on $[0, T]\times {\mR}^3.$
\end{definition}

\begin{remark} \label{remark of mild solutions}
{\rm
For the cutoff case $Q(\cdot)=Q_{n,K}, Q^{\pm}(\cdot)=Q_{n,K}^{\pm}(\cdot)$,
the null sets $Z$ in the definition of mild solutions can be chosen as an empty set. See Proposition \ref{fn} below.
}
\end{remark}

\subsection{Some lemmas and propositions}
This subsection is a collection of  technical lemmas and propositions.
We begin with the proofs of two important lemmas. These two lemmas together with their corollary will help us to obtain the moment estimates of mild solutions of approximate equations (\ref{approximate equation 1}) and (\ref{approximate equation 2}).
  \begin{lemma} \label{LemmaA}
  Let $k>1, x\ge 0, y\ge 0$. Then we have
\beqa\label{1<k<2}
(k-1)\min\{ x^k, y^k\}\le (x+y)^k-x^k-y^k\le (2^k-2)\max\{ x^{k-\ld}y^{\ld},\, y^{k-\ld}x^{\ld}\}
\eeqa
where $0\le \ld\le \min\{1, k/2\}. $
  \end{lemma}

\bepf The second inequality in ({\ref{1<k<2}}) relies on the following elementary
inequality:
\beqa \label{4.8}
(1+X)^k\leq 1+X^k+(2^k-2)X^{\ld}, \ \ \ \ X\in[0,1].
\eeqa
To prove (\ref{4.8}) we first assume that $k\ge 2$. Let
$$\vp(X)=(1+X)^k-1-X^k-(2^k-2)X,\quad X\in [0,1].$$
Since $\vp$ is convex on $[0,1]$,
it follows that $\vp(X)\le \max\{ \vp(0), \vp(1)\}=0$ for all
$X\in [0,1].$ Thus
\beas
(1+X)^k\le 1+X^k+(2^k-2)X
\le  1+X^k+(2^k-2)X^{\ld}\qquad \forall\, X\in [0,1].
\eeas
Here in the last inequality we used the assumption $0\le \ld\le \min\{1, k/2\}=1.$
Next assume that $1<k\le 2.$ We consider
\beas\psi(X)=(1+X)^kX^{-\ld}-X^{-\ld}-X^{k-\ld},\quad X\in(0,1].\eeas
Using convexity of the derivative $\psi'(X)$ it is not difficult to prove that the function $\psi(X)$ is increasing on $(0,1]$.
Then we have
$$(1+X)^kX^{-\ld}-X^{-\ld}-X^{k-\ld}=\psi(X)\le \psi(1)=2^k-2\qquad \forall\, X\in (0,1].$$
This gives
\beas
(1+X)^k-1-X^{k}\le (2^k-2)X^{\ld}\qquad \forall\, X\in [0,1].
\eeas
combining the above we have (\ref{4.8}).

Now let $x\ge 0, y\ge 0$. Without loss of generality we may assume $x\le y$ and $y>0$.
It follows from (\ref{4.8}) that
\beas
&&(x+y)^k-x^k-y^k
=y^k\Big( \Big(1+\fr{x}{y}\Big)^k-1-\Big(\fr{x}{y}\Big)^k\Big)
\no\\
&&\le y^k (2^k-2)\Big(\fr{x}{y}\Big)^{\ld}
=(2^k-2)\max\{ x^{k-\ld}y^{\ld},\, y^{k-\ld}x^{\ld}\},\quad k> 1.
\eeas
This proves the second inequality in (\ref{1<k<2}).
Next  using Bernoulli's inequality we have
$$
(x+y)^k
=y^k \Big(1+\fr{x}{y}\Big)^k
\ge y^k\Big( 1+k\fr{x}{y}\Big)
=y^k+k x y^{k-1}\ge y^k+k x^k, \quad k>1.$$
Then
$$(x+y)^k-x^k-y^k\ge (k-1)x^k=(k-1)\min\{ x^k, y^k\},\quad k>1.$$
This ends up the proof.
\end{proof}

\vskip3mm

\begin{lemma}\label{LemmaD}  Let  $ c\ge 0,  \vh(v)=c+|v|^2 , v\in{\R^3}$, and let
$k> 1, 0\le \ld\le \min\{1, k/2\}$.  Then  for  all  $(v,v_*,\sg)
\in{\R^3\times \R^3\times \bS}$
\beqa\label{(1)}
&&[\vh(v')]^k+[\vh(v_*')]^k - [\vh(v)]^k-[\vh(v_*)]^k
\no\\
&&\le 2(2^k-2)\{[\vh(v)]^{k-\ld}[\vh(v_*)]^{\ld}+
\vh(v)^{\ld}[\vh(v_*)]^{k-\ld}\}
-4^{-k}(k-1)[\kappa(\theta)]^{k}[\vh(v)]^k
\eeqa
where $v'$,$v_*'$ are given by $($\ref{sigma representation}$)$, and
\beqa\label{(2)}
&&\kappa(\theta)=\min\{(1-\sin(\theta/2))^2\,, (1-\cos(\theta/2))^2\},
\quad \theta=\arccos(\la {\bf n},\sg\ra)\in[0,\pi]\\
\label{(3)}
&&{\bf n}=\fr{v-v_*}{|v-v_*|}\quad {\rm if}\quad v\neq v_*;\quad
{\bf n}={\bf e}_1\quad {\rm if}\quad v= v_*.
\eeqa
\end{lemma}

\begin{proof}Given any  $(v,v_*,\sg)
\in{\R^3\times \R^3\times \bS},$ denote
$$\vh=\vh(v), \quad \vh_*=\vh(v_*), \quad \vh'=\vh(v'),\quad
\vh_*'=\vh(v_*'),$$
$$D_k=(\vh')^k+(\vh_*')^k-(\vh)^k-(\vh_*)^k.$$
Then it follows from the identity  $\vh'+\vh_*'=\vh+\vh_*$ that
\beas
D_k=(\vh+\vh_*)^k-(\vh)^k-(\vh_*)^k-
[(\vh'+\vh_*')^k-(\vh')^k-(\vh_*')^k].
\eeas
By Lemma \ref{LemmaA} we conclude
\beas&& (\vh+\vh_*)^k-(\vh)^k-(\vh_*)^k
\le(2^k-2)\max\{(\vh)^{k-\ld}\vh_*^{\ld},\,
\vh^{\ld}(\vh_*)^{k-\ld}\}\eeas
and
$$(\vh'+\vh_*')^k-(\vh')^k-(\vh_*')^k\geq (k-1)\min\{(\vh')^k,\, (\vh_*')^k\}.$$
Thus
\beqa\label{(4)}
D_k\le (2^k-2)\max\{(\vh)^{k-\ld}\vh_*^{\ld},\,
\vh^{\ld}(\vh_*)^{k-\ld}\}- (k-1)\min\{(\vh')^k,\, (\vh_*')^k\}.
\eeqa
This implies first that (\ref{(1)}) holds for $\kappa(\theta)=0$.
Now suppose that $\kappa(\theta)>0$, i.e. $0<\cos(\theta/2)<1$.
By definition of $v',v_*'$ we have
\beas
|v'-v|=|v-v_*|\sin(\theta/2),\quad |v_*'-v|=|v-v_*|\cos(\theta/2).
\eeas
From this we have
$$|v'|=|v'-v+v|\ge |v|-|v'-v|
=|v|-|v-v_*|\sin(\theta/2),$$
$$|v_*'|=|v_*'-v+v|\ge |v|-|v_*'-v|
=|v|-|v-v_*|\cos(\theta/2).$$
Let
$$M(\theta)=\max\bigg\{\frac{\sin(\theta/2)}{1-\sin(\theta/2)}\,,\,
\frac{\cos(\theta/2)}
{1-\cos(\theta/2)}\bigg\}.$$
Then it follows from $\max\{ \sin(\theta/2),\cos(\theta/2)\}\ge 1/2$  that $M(\theta)\ge 1$.
At this stage we will
look at two possibilities.

{\bf Case 1}: $|v|\geq2M(\theta)|v_*|$. By definitions of $v', v_*'$ and $\cos(\theta)$, we have
\beas&& |v'|
\ge |v|-|v-v_*|\sin(\theta/2)
\ge |v|-|v|\sin(\theta/2)-|v_*|\sin(\theta/2)
\\
&&\quad\,\,\,=|v|(1-\sin(\theta/2))-|v_*|\sin(\theta/2)\ge \fr{1-\sin(\theta/2)}{2}|v|,\\
&&|v_*'|
\ge |v|-|v-v_*|\cos(\theta/2)
\ge |v|-|v|\cos(\theta/2)-|v_*|\cos(\theta/2)
\\
&&\quad\,\,\,=|v|(1-\cos(\theta/2))-|v_*|\cos(\theta/2)\ge \fr{1-\cos(\theta/2)}{2}|v|.
\eeas
These imply
$$\vh'=c+|v'|^2\ge c+\frac{1}{4}(1-\sin(\theta/2))^2|v|^2
\ge \frac{1}{4}(1-\sin(\theta/2))^2(c+|v|^2)
=\frac{1}{4}(1-\sin(\theta/2))^2\vh\ge \fr{1}{4}\kappa(\theta) \vh,$$
$$\vh_*'=c+|v_*'|^2\ge c+\frac{1}{4}(1-\cos(\theta/2))^2|v|^2
\ge \frac{1}{4}(1-\cos(\theta/2))^2(c+|v|^2)
=\frac{1}{4}(1-\cos(\theta/2))^2\vh\ge \fr{1}{4}\kappa(\theta) \vh,$$
and so
$$(k-1)\min\{(\vh')^k,\, (\vh_*')^k\}\ge
(k-1)\fr{1}{4^k}[\kappa(\theta)]^k (\vh) ^k,$$
hence, by (\ref{(4)}), we have
\beqa\label{Lemma B Case1}
D_k\le (2^k-2)\max\{(\vh)^{k-\ld}\vh_*^{\ld},\,
\vh^{\ld}(\vh_*)^{k-\ld}\}- 4^{-k}(k-1)[\kappa(\theta)]^k(\vh)^k.
\eeqa

{\bf Case 2}: $|v|\leq2M(\theta)|v_*|$. Then (notice that $M(\theta)\ge 1$)
$$\vh=c+|v|^2\le c+[2M(\theta)]^2|v_*|^2\le [2M(\theta)]^2(c+|v_*|^2)
=[2M(\theta)]^2 \vh_*.$$
From the inequality
\beas
\min\{ (1-x)^2, (1-y)^2\}\max\big\{\fr{x}{1-x},\, \fr{y}{1-y}\big\}\le \fr{1}{4}\qquad \forall\, 0<x, y<1
\eeas
we have
$$\kappa(\theta)M(\theta)=\min\{(1-\sin(\theta/2))^2\,, (1-\cos(\theta/2))^2\}\max\bigg\{\frac{\sin(\theta/2)}{1-\sin(\theta/2)}\,,\,
\frac{\cos(\theta/2)}
{1-\cos(\theta/2)}\bigg\}\le \fr{1}{4}. $$
This together with $\ld\le \fr{k}{2}$ and $k-1\le 2^k-2\, (k>1)$ gives
\beas 4^{-k}(k-1)[\kappa(\theta)]^k(\vh)^k
\le 4^{-k}(2^k-2)(1/2)^k(\vh)^{k-\ld}(\vh_*)^{\ld}.\eeas
Therefore, using (\ref{(4)}) (omitting the negative term) we have
\beqa\label{Lemma B Case2}
D_k\le (2^k-2)(1+(1/2)^k )
\max\{(\vh)^{k-\ld}(\vh_*)^{\ld},\,
(\vh)^{\ld}(\vh_*)^{k-\ld}\}- 4^{-k}(k-1)[\kappa(\theta)]^k(\vh)^k.\eeqa
Combining (\ref{Lemma B Case1}) and (\ref{Lemma B Case2}) gives (\ref{(1)}).
\end{proof}
\vskip3mm
\begin{corollary}\label{LemmaE}
Let  $s> 2$.  Then
for all $ (v,v_*,\sg)\in {\R^3\times \R^3\times \bS}$ and $0\le \gamma\le \min\{2, s/2\}$ we have
\beas\la v'\ra^{s}+\la v_*'\ra^s - \la v\ra ^s-\la v_*\ra^s
&\le& 2(2^{s/2}-2)\big(\la v\ra^{s-\gamma}\la v_*\ra^{\gamma}+
\la v\ra^{\gamma}\la v_*\ra^{s-\gamma}\big)-2^{-s}(\fr{s}{2}-1)[\kappa(\theta)]^{s/2}\la v\ra ^s,\\
|v'|^{s}+|v_*'|^s - |v|^s-|v_*|^s
&\le& 2(2^{s/2}-2)\big(|v|^{s-\gamma}|v_*|^{\gamma}+
|v|^{\gamma}|v_*|^{s-\gamma}\big)-2^{-s}(\fr{s}{2}-1)[\kappa(\theta)]^{s/2}| v| ^s
\eeas
where $v'$, $v_*'$ are given in $(\ref{sigma representation})$ and $\kappa(\theta), \theta$ are given in $(\ref{(2)})$ and $(\ref{(3)})$ respectively.
\end{corollary}

\begin{proof}In Lemma \ref{LemmaD},  we take $k=s/2> 1, \gamma =2\ld $ and take
$c=1, 0$ respectively, and then we deduce the above inequalities.
\end{proof}
\vskip3mm
The lemma below will be frequently used when we deal with the estimates of some kinds of collision integral operators with cutoff such as $Q_{K}$ and $Q_{n,K}$.
\begin{lemma}\label{LemmaF}
Let  $x,y,z\ge 0$.
Then
$$|x\wedge z-y\wedge z|\le |x-y|,\qquad x\le y\,\, \Longrightarrow\,\, x\wedge z\le y\wedge z,\qquad(x+y)\wedge z \le x\wedge z+y\wedge z.$$
\end{lemma}

\bepf Fix $z\ge 0$ and let
$f(X)=X\wedge z=\fr{1}{2}(X+z-|X-z|),\,\, X\ge 0.$ The properties in this lemma follow easily from
the fact that $f$ is concave and non-decreasing on $[0,\infty)$.
\eepf
\vskip3mm

The next lemma deals with the completeness of
some function spaces (e.g. $L^{\infty}([0,\infty); L^1({\bR}))$).
This completeness is of course important when using for instance the
fixed point theorem of contractive mappings to prove
the existence of solutions of some integral equations. A proof of such a lemma should be able to find from some textbooks. For convenience of the reader we would like to present here a proof.

\begin{lemma}\label{lemma of L(infty)([0,T];L1)}
Let $\Og\subset {\mR}^N$ $($with $N\in{\mN})$ be a Lebesgue measurable set,
let $I=[0,\infty)$ or $I=[0, T]$ with $0<T<\infty$.  Define
\beqa\label{definition of L(infty)([0,T];L1)}
&&L^{\infty}(I; L^1(\Og))=\Big\{ f:  I\times \Og\to [-\infty, \infty]\,\,
\ {\rm is\,\, measurable}\ {\rm on}\ I\times\Og\,\,\Big|\,\,
\sup_{t\in I}\int_{\Og}|f(t,v)|\dd v<\infty\Big\}\qquad\qquad
\eeqa
with the norm
\beqa\label{NORM} \|f\|:=\sup_{t\in I}\|f(t)\|_{L^1(\Og)}=\sup_{t\in I}\|f(t,\cdot)\|_{L^1(\Og)}=\sup_{t\in I}\int_{\Og}|f(t,v)|\dd v,\quad f\in L^{\infty}(I; L^1(\Og)).\eeqa
Then $L^{\infty}(I; L^1(\Og))$ with the norm $\|\cdot\|$ is a Banach space.

Furthermore, let $\{f^n\}_{n=1}^{\infty}$ be a bounded sequence in $L^{\infty}(I; L^1(\Og))$
satisfying that for any $t\in I,$ $\{f^n(t,\cdot)\}_{n=1}^{\infty}$
is a Cauchy sequence in $L^1(\Og)$.  Then there is a function $f\in L^{\infty}(I; L^1(\Og))$
such that
$\|f^n(t,\cdot)-f(t,\cdot)\|_{L^1(\Og)}\,(n\to\infty)$ for all $t\in I$.
Besides if  assume in addition that all $f^n$ are nonnegative on $I\times \Og$, then $f$ is also nonnegative on $I\times \Og$.
\end{lemma}

\bepf It is obvious that $L^{\infty}(I; L^1(\Og))$ is a real normed linear
space.
We first prove the second part in {\bf Step 1} since the first part
(to be proved in {\bf Step 2}) is relatively easy.

{\bf Step 1.} By the assumption and the compactness of $L^1(\Og)$ we know that for any $t\in I$
there is a function $g(t,\cdot)\in L^1(\Og)$ such that
\beqa\label{partialconvergence}
\lim\limits_{n\to \infty}\|f^n(t)-g(t)\|_{L^1(\Og)}=0,\quad \sup_{t\in I}\|g(t)\|_{L^1(\Og)}\le \sup_{n\ge 1}\|f^n\|<\infty.
\eeqa
In order to prove measurability in full variables we consider an integrable weight function:
$$\rho(t)=e^{-t}\quad {\rm if}\quad I=[0,\infty);\quad  \rho(t)=1
\quad {\rm if}\quad  I=[0,T].$$
We need to show that $\{f^n\}_{n=1}^{\infty}$ is a Cauchy sequence in
$L^1(I\times\Og, \rho(t){\rm d}t{\rm d}v)$. In fact we have
\beas&&
\sup_{m\ge n}\intt_{I\times\Og}\rho(t)|f^{m}(t,v)-f^{n}(t,v)|\dd v\dd t
=\sup_{m\ge n}\int_{I}\rho(t)\|f^{m}(t)-f^{n}(t)\|_{L^1(\Og)}\dd t
\\
&&\le \int_{I}\rho(t)\sup_{m\ge n}\|f^{m}(t)-f^{n}(t)\|_{L^1(\Og)}\dd t =
\int_{I}\rho(t) \og_n(t)\dd t
\eeas
where $\og_n(t)=\sup\limits_{m\ge n}\|f^{m}(t)-f^{n}(t)\|_{L^1(\Og)}$. Next, from (\ref{partialconvergence}) we have
$$\lim_{n\to\infty}\og_n(t)=0\quad \forall\, t\in I;\quad
\sup_{n\in {\mN}, t\in I}\og_n(t)\le 2 \sup_{n\ge 1}\|f^n\|<\infty.$$
This together with Lebesgue's dominated convergence implies that
$\int_{I}\rho(t)\og_n(t)\dd t\to 0\,(n\to\infty)$. Thus  $\{f^n\}_{n=1}^{\infty}$ is a Cauchy sequence in
$L^1(I\times\Og, \rho(t){\rm d}t{\rm d}v)$. Since $L^1(I\times\Og, \rho(t){\rm d}t{\rm d}v)$ is complete, there is a function
 $h\in L^1(I\times\Og, \rho(t){\rm d}t{\rm d}v)$ such that
\beqa\label{h}
\int_{I}\rho(t)\|f^{n}(t)-h(t)\|_{L^1(\Og)}{\rm d}t=\intt_{I\times\Og} \rho(t)|f^{n}(t,v)-h(t,v)|\dd v\dd t
 \to 0\quad( n\to \infty).
 \eeqa
Now consider two sets
\beas&& Z_{1}=\{t\in I\,\,|\,\, h(t,\cdot)\not\in L^1(\Og)\,\},\\
&&
Z_{2}=\{t\in I\setminus Z_{1}\,\,|\,\,
 \|f^{n}(t)-h(t)\|_{L^1(\Og)}\not\to 0\,\,(n\to\infty)\,\}.\eeas
We prove that $Z_1, Z_2$ are null sets. From Fubini's theorem we know that ${\rm mes} (Z_{1})=0.$
To prove that ${\rm mes}(Z_{2})=0$, we first use Fatou's lemma to obtain
$$\int_{I}\rho(t)(\liminf\limits_{n\to \infty}\|f^{n}(t)-h(t)\|_{L^1(\Og)}){\rm d}t\le
\liminf\limits_{n\to \infty}\int_{I}\rho(t)\|f^{n}(t)-h(t)\|_{L^1(\Og)}{\rm d}t
=0$$
which implies that
$$Z_{3}:=\{t\in I\setminus Z_{1}\,\,|\,\, \liminf\limits_{n\to \infty}
 \|f^{n_k}(t)-h(t)\|_{L^1(\Og)}\neq 0\}$$
has measure zero. Next take any $t\in I\setminus (Z_{1}\cup Z_{3})$. We have
$h(t)\in L^1(\Og)$ and
$$\liminf_{n\to\infty}\|f^{n}(t)-h(t)\|_{L^1(\Og)}=0 $$
which implies that there exists a subsequence  $\{n_k\}_{k=1}^{\infty}\subset {\mN}$ (depending on t)
such that
$$\lim_{k\to\infty}\|f^{n_{k}}(t)-h(t)\|_{L^1(\Og)}=0.$$
Comparing this with (\ref{partialconvergence}) we conclude that $h(t,v)=g(t,v)$ a.e. $v\in \R^3$. Thus using (\ref{partialconvergence}) again we obtain
$$\|f^{n}(t)-h(t)\|_{L^1(\Og)}=\|f^{n}(t)-g(t)\|_{L^1(\Og)}\to 0\quad(n\to \infty)$$
which implies that $t\in I\setminus Z_{2}$. This proves that $I\setminus (Z_{1}\cup Z_{3})\subset
 I\setminus Z_{2}$, i.e.
$Z_{2}\subset Z_{1}\cup Z_{3}$ and so
${\rm mes}(Z_{2})=0$.
Let $Z=Z_1\cup Z_{2}$ and define
\beqa\label{f}
f(t,v)=
  \left\{
  \begin{aligned}
  &g(t,v)\,,&  (t, v)\in Z\times \Og;
  \\
  &h(t,v)\,,\,\,& t\in (I\setminus Z)\times \Og.
  \\
  \end{aligned}
  \right.
\eeqa
Then  $f$ is measurable on $I\times \Og$ (since $mes(Z)=0$), and  for any $t\in I$,
the function $v\mapsto f(t,v)$ belongs to $L^1(\Og)$. Also it is easily
checked that
\beqa\label{convergenct in every t}
\lim_{n\to\infty}\|f^{n}(t)-f(t)\|_{L^1(\Og)}=0\quad {\rm for \,\,every}\,\, t\in I.\eeqa
From this and the boundedness of $\{f^n\}_{n=1}^{\infty}$ in $L^{\infty}(I; L^1(\Og))$ we have
$ \sup\limits_{t\in I}\|f(t)\|_{L^1(\Og)}\le \sup\limits_{n\ge 1}\|f^n\|<\infty.$
Thus $f\in L^{\infty}(I; L^1(\Og))$.

 Finally suppose in addition that all $f^n$ are nonnegative on $I\times \Og$, then from (\ref{partialconvergence}) and (\ref{h}) we see that $g$, $h$ can be chosen as nonnegative functions on  $I\times \Og$. It follows from the definition (\ref{f}) that $f$ is also nonnegative on $I\times \Og$.
Thus we have finished the proof of the second part of this lemma.

{\bf Step 2.} We now prove the first part of this lemma. Let $\{f^n\}_{n=1}^{\infty}$
be a Cauchy sequence in $L^{\infty}(I; L^1(\Og))$.
Then $\{f^n\}_{n=1}^{\infty}$ is bounded in $L^{\infty}(I; L^1(\Og))$ and recalling the definition of $\|\cdot\|$ we see from {\bf Step 1} that there exists a function $f\in L^{\infty}(I; L^1(\Og))$ such that (\ref{convergenct in every t}) holds true.
From this we have
$\|f^{n}(t)-f(t)\|_{L^1(\Og)}=
\lim\limits_{m\to\infty}\|f^{n}(t)-f^m(t)\|_{L^1(\Og)}$
and so
$$\|f^n-f\|=\sup_{t\in I}\|f^{n}(t)-f(t)\|_{L^1(\Og)}\le
\sup_{m\ge n}\|f^n-f^m\|\to 0\quad (n\to\infty).$$
This proves that $\{f^n\}_{n=1}^{\infty}$ converges
in $L^{\infty}(I; L^1(\Og))$ and so $L^{\infty}(I; L^1(\Og))$ is
a Banach space.
\eepf
\vskip3mm

The lemma below will help us to prove a property that if $f$ is a mild solution of Eq.(\ref{approximate equation 2}) and satisfies
$f(t,v)\le K$ for all $(t,v)\in [0,\infty)\times (\R^3\setminus Z)$ with a null set $Z$
independent of $t$, then $f$ is a mild solution of Eq.(\ref{Equation}).

\begin{lemma}\label{Lemma H} Let $I\subset {\mR}$ be an interval,
$\Og\subset {\mR}^N$ an Lebesgue measurable set.
Let $f, g: I\times \Og\to [-\infty, \infty]$ be Lebesgue measurable
functions satisfying

{\rm (i)}\, for almost every  $v\in \Og,\,
t\mapsto f(t,v),\, t\mapsto g(t,v)$ are  continuous on $I$,

{\rm (ii)}\, $f(t,v)=g(t,v)$  for almost every $(t,v)\in I\times \Og.$

\noindent Then there is a common null set $Z\subset \Og$, such that
$f(t,v)=g(t,v)$ for all $(t,v)\in I\times (\Og\setminus Z)$.
\end{lemma}

\bepf By assumption (i), there is a null set $Z_1\subset \Og$ such that
for every $v\in \Og\setminus Z_1$,
$t\mapsto f(t,v),\, t\mapsto g(t,v)$ are continuous on $I$.  Let
$S=\{(t,v)\in I\times \Og\,|\, f(t,v)\neq g(t,v)\}$, $S_v=\{t\in I\,|\, (t,v)\in S\}, v\in \Og$.
By assumption (ii) we have $mes(S)=0$ and so by Fubini's theorem,
the set $Z_2=\{v\in \Og\,|\, S_v\,\, {\rm is \,\, not\,\, a\,\, null \,\, set}\}$
has measure zero. Let $Z=Z_1\cup Z_2$. Then $mes(Z)=0$ and for any $(t,v)\in I\times(\Og\setminus Z)$,
if $t\in I\setminus S_v$, then $f(t,v)=g(t,v)$; if $t\in S_v$, then since
$S_v$ is a null set which implies that $I\setminus S_v$ is dense in $I$,
there is a sequence $\{t_n\}_{n=1}^{\infty}\subset  I\setminus S_v$ such that
$t_n\to t\,(n\to\infty)$ and so, by continuity,
$f(t,v)=\lim\limits_{n\to\infty}f(t_n,v)=\lim\limits_{n\to\infty}g(t_n,v)=g(t,v)$.
\eepf
\vskip3mm

\begin{lemma}\label{Lemma7}
If $f(t)$ is absolutely continuous on $[0,T]$,
$G(y) $ is Lipschitz continuous on $[a,b]$  and $f([0,T])\subset
[a,b].$ Then
$$ G(f(t))=G(f(0))+\int_{0}^{t}G_{1}(f(\tau))f'(\tau)d\tau
,\ \ \ \ t\in[0,T]$$ where $G_{1}(y)=\frac{d}{dy}G(y)$  a.e. on
$[a,b].$
\end{lemma}

\bepf See \cite{mukherjea1984real}, p.223 Theorem 4.3, p.263 Theorem 4.9,
and note that
 by assumption of the lemma the function $t\mapsto G(f(t))$ is also absolutely
continuous on $[0,T]$.
\eepf
\vskip3mm

Our next lemma will be used to prove that condition (\ref{condition}) can imply the very high temperature condition (\ref{Kinetic high change}).
\begin{lemma}\label{lemma9}
Given constants $0<p<q<\infty$. Let $\phi$ be measurable on $[0,\infty)$ with
$0\le \phi \le 1$ and $0<\int_0^{\infty}r^{q-1}\phi(r)\dd r<\infty.$ Then
\beqa\label{Tc estimate}
\Big(p\int_0^{\infty}r^{p-1}\phi(r)\dd r\Big)^{\fr{1}{p}}
\le
\Big(q\int_0^{\infty}r^{q-1}\phi(r)\dd r\Big)^{\fr{1}{q}},
\eeqa
and the equality sign holds if and only if there is a constant $0<R<\infty$ such that
$\phi=1_{[0,R]}$ a.e. on $[0,\infty).$
\end{lemma}

\bepf See \cite{Lu2001On}, p.382 Lemma 4.
\eepf
\vskip3mm

The next two propositions are concerned with computation formula and basic estimates for
certain general collision integrals. They are very useful in obtaining the pointwise estimates of some kinds of collision integral operators.

\begin{proposition}\label{proposition infty}
Let $W(z,\sg, v)=W_1(|z|,
\la \fr{z}{|z|},\sg\ra \,,v)$ where $W_1(r,s,v)$ is a nonnegative
Lebesgue measurable function on
$[0,\infty)\times[-1,1]\times{\R^3}$. Then \beas &&\intt_{\bRS}
W(v-v_*,\sg,v')\dd \sg
\dd v_*=\intt_{\bRS}\fr{1}{\sin^{3}(\theta/2)}W\Big(\fr{v-v_*}{\sin(\theta/2)},\sg,
v_*\Big)\dd\sg \dd v_*\,, \\ \\
 &&\intt_{\bRS}
W(v-v_*,\sg, v_*')\dd \sg
\dd v_*=\intt_{\bRS}\fr{1}{\cos^{3}(\theta/2)}W\Big(\fr{v-v_*}{\cos(\theta/2)},\sg,
v_*\Big) \dd\sg \dd v_*\,\eeas
where $v'$, $v_*'$ are given in $(\ref{sigma representation})$, $\theta$ is given in $(\ref{(3)})$,
$\sin(\theta/2)=\sqrt{\fr{1-\la {\bf n},\sg\ra}{2}},\,
\cos(\theta/2)=\sqrt{\fr{1+\la {\bf n},\sg\ra}{2}}. $
\end{proposition}

\bepf See \cite{Lu2006On}, p.1715 Proposition 2.1.
\eepf
\vskip3mm

\begin{proposition}\label{proposition2}
 Let $p,q,\gamma\ge 0, 0\le g\in L^1_{p+\gamma}({\bR})\cap L^{\infty}_{p}({\bR}), 0\le f\in L^1_{q+\gamma}({\bR})\cap L^{\infty}_{q}({\bR})$.
Then
$$\intt_{{\bRS}}|v-v_*|^{\gamma}\la v'\ra^{p}\la v_*'\ra ^q
g(v')f(v_*'){\rm d}\sg
{\rm d}v_*
\le 2^{(3+\gamma)/2}|{\bS}|\big(\|f\|_{L^{\infty}_q}\|g\|_{L^{1}_{p+\gamma}}
+|\|g\|_{L^{\infty}_p}\|f\|_{L^{1}_{q+\gamma}}
\big)\la v\ra ^{\gamma}
$$
for all $v\in \R^3$, where $v'$, $v_*'$ are given in $(\ref{sigma representation})$.
\end{proposition}

\bepf From Proposition \ref{proposition infty} and the inequality $|v-v_*|\le \la v\ra \la v_*\ra$ we have
\beas&&\intt_{{\bRS}}|v-v_*|^{\gamma}\la v'\ra^{p}\la v_*'\ra ^q
g(v')f(v_*'){\rm d}\sg
{\rm d}v_*
\\
&&=\intt_{{\bRS}}|v-v_*|^{\gamma}\la v'\ra^{p}\la v_*'\ra ^q
g(v')f(v_*')1_{ \{\la
\fr{v-v_*}{|v-v_*|}, \sg\ra\le 0\}}{\rm d}\sg
{\rm d}v_*\\
&&\quad +\intt_{{\bRS}}|v-v_*|^{\gamma}\la v'\ra^{p}\la v_*'\ra ^q
g(v')f(v_*')1_{\{
\la \fr{v-v_*}{|v-v_*|}, \sg\ra> 0\}}{\rm d}\sg
{\rm d}v_*
\\
&&\le \|f\|_{L^{\infty}_q}
\intt_{{\bRS}}\fr{|v-v_*|^{\gamma}}{\sin^{\gamma}(\theta/2)}\la  v_*\ra^{p}
g(v_*) \fr{1}{\sin^3(\theta/2)}1_{\{\sin(\theta/2)\ge \sqrt{\fr{1}{2}}\}}{\rm d}\sg
{\rm d}v_*\\
&&\quad +\|g\|_{L^{\infty}_p}\intt_{{\bRS}}\fr{|v-v_*|^{\gamma}}{\cos^{\gamma}(\theta/2)}\la v_*\ra ^q
f(v_*)\fr{1}{\cos^3(\theta/2)}1_{\{\cos(\theta/2)\ge \sqrt{\fr{1}{2}}\}}
{\rm d}\sg{\rm d}v_*
\\
&&\le 2^{(3+\gamma)/2}\|f\|_{L^{\infty}_q}
\intt_{{\bRS}}|v-v_*|^{\gamma}\la v_*\ra^{p}
g(v_*) {\rm d}\sg
{\rm d}v_*+ 2^{(3+\gamma)/2}\|g\|_{L^{\infty}_p}\intt_{{\bRS}}|v-v_*|^{\gamma}\la v_*\ra ^q
f(v_*){\rm d}\sg{\rm d}v_*
\\
&&\le 2^{(3+\gamma)/2}|{\bS}|\big(\|f\|_{L^{\infty}_q}\|g\|_{L^{1}_{p+\gamma}}
+\|g\|_{L^{\infty}_p}\|f\|_{L^{1}_{q+\gamma}}
\big)\la v\ra ^{\gamma},\qquad v\in {\bR}.\eeas
\eepf
\vskip3mm

 The proposition below will be used to prove the existence of mild solutions of Eq.(\ref{approximate equation 1}) in Section 3.
\begin{proposition}\label{I(f,g)}
For any $n\in \mathbb{N}$, let
$B_n=B\wedge n$ with $B$ the collision kernel satisfying $(\ref{condition of kernel})$.
For any
 $0\le f,g\in L^1({\bR})$ define
\beas
&&{\cal I}_{n}(f,g)=
\inttt_{{\bRRS}}
B_n(v-v_*,\sg)\big|(f'\wedge n)(f_*'\wedge n)
(1+ f\wedge K+ f_*\wedge K)\\
&&\,\,\,\,\,\quad\qquad\qquad \qquad \qquad  \qquad-(g'\wedge n)(g_*'\wedge n)
(1+ g\wedge K+ g_*\wedge K)\big|
{\rm d}\sg{\rm d}v_*{\rm d}v.
\eeas
Then
\beqa\label{In(f,g)}&&
{\cal I}_{n}(f,g)
\le (1+2K)|{\bS}|n (\|f\|_{L^1}+\|g\|_{L^1})\|f-g\|_{L^1}
+2^{7/2}|{\bS}|n\|g\wedge n\|_{L^{\infty}}\|g\|_{L^1}\|f-g\|_{L^1}.\qquad\eeqa
\end{proposition}

\bepf Compute
\beas&&
{\cal I}_{n}(f,g)
=\inttt_{{\bRRS}}
B_n(v-v_*,\sg)\big|(f'\wedge n)(f_*'\wedge n)
(1+ f\wedge K+ f_*\wedge K)\\
&&\quad -(g'\wedge n)(g_*'\wedge n)
(1+ g\wedge K+ g_*\wedge K)\big|
{\rm d}\sg{\rm d}v_*{\rm d}v
\\
&&\le
(1+2K)\inttt_{{\bRRS}}
B_n(v-v_*,\sg)\big((|f' -g'|(f_*'\wedge n)
\\
&&\quad +(g'\wedge n)|f_*'-g_*'|)+ (g'\wedge n)(g_*'\wedge n)(
|f-g|+|f_*-g_*|)\big)
{\rm d}\sg{\rm d}v_*{\rm d}v
\\
&&\le (1+2K)|{\bS}|n\intt_{{\bRR}}|f -g|f_*{\rm d}v_*{\rm d}v
+(1+2K)|{\bS}|n
\intt_{{\bRR}}g|f_*-g_*|{\rm d}v_*{\rm d}v\\
&&\quad +2 n\int_{{\bR}}|f-g|\intt_{{\bRS}}(g'\wedge n)(g_*'\wedge n)
{\rm d}\sg{\rm d}v_*{\rm d}v
\eeas
where we used (\ref{identity}).
According to Proposition \ref{proposition2} (with $p=q=\gamma=0$)  we have for any $v\in {\bR}$
$$
\intt_{{\bRS}}(g'\wedge n)(g_*'\wedge n)
{\rm d}\sg{\rm d}v_*
\le 2^{5/2}|{\bS}|\|g\wedge n\|_{L^{\infty}}\|g\wedge n\|_{L^1}. $$
Combining the above gives (\ref{In(f,g)}).
\eepf
\vskip3mm
The following proposition is concerned with estimates of collision gain operator $Q^{+}(\cdot,\cdot)$.
  \begin{proposition}\label{proposition 2.3}
   Let $Q^+(\cdot,\cdot)$ be defined by $(\ref{Q+ bilinear})$ with the collision kernel $B$ satisfying $(\ref{condition of kernel})$ and let $0\le f,g\in L_1^1(\R^3)$. Then
  \beqa\label{exchange}
  Q^{+}(f,g)(v)=Q^{+}(g,f)(v) \quad \forall v\in \R^3.
  \eeqa
  Furthermore if $0\le f\in L_2^1(\R^3)\cap L^{\infty}(\R^3)$, then
   \beqa \label{Q1}
   Q^+(f,f)(v)
   &\le &
   2^5\pi b\|f\|_{L^{\infty}}\|f\|_{L_1^1}\la v\ra  \quad \forall\, v\in {\R}^3
   \\
   \label{Q+(f0,ft)L1}
   \|Q^+(f,f)\|_{L^1}
   &\le &4\pi b \|f\|_{L_1^1}^2
 \\
  \label{Q+f0f0}
   \|Q^+(f,f)\|_{L^2}
   &\le & 2^{3+\fr{1}{4}}\pi b\|f\|_{L^{\infty}}^{\fr{1}{2}}\|f\|_{L^1}^{\fr{1}{2}}\|f\|_{L_2^1}.
   \eeqa
   \end{proposition}

   \bepf Recalling (\ref{sigma representation}) and (\ref{about v-v*}),  (\ref{exchange}) can be easily proved by using the change of variables $(v_*,\si) \to (v_*,-\si)$.
      Next, it follows from (\ref{condition of kernel}) and Proposition \ref{proposition2} (with $p=q=0$, $\gamma=1$) that
   \beas
   &&Q^+(f,f)(v)
   \le
   b\iint\limits_{\R^3\times\SP^2}|v-v_*|f'{f_*}'\dd \si\dd v_{*}
   \le 8b|\bS|\|f\|_{L^{\infty}}\|f\|_{L_1^1}\la v\ra  \quad \forall\, v\in {\R}^3
   \eeas
   which proves (\ref{Q1}).
   To prove ($\ref{Q+(f0,ft)L1}$), we have
   \beas
   &&\|Q^+(f,f)\|_{L^1}
   \le
   b\inttt_{\bRRS}\la v\ra \la v_*\ra f f_{*}\dd \si \dd v_* \dd v
   \le b|\bS|\|f\|_{L_1^1}^2 \quad \forall t\ge 0.
   \eeas
    Finally, using H\"{o}lder's inequality, Proposition \ref{proposition2} (with $p=q=\gamma=0$) and the fact that $|v-v_*|\le \la v\ra \la v_*\ra$ we have
 \beas
\|Q^+(f,f)\|_{L^2}^2
   &\le &
   \int_{\bR}\Big(b\intt_{\bRS}|v-v_*|f'f_{*}'\dd \si \dd v_*\Big)^2\dd v
   \\
   &\le&
   b^2\int_{\bR}
   \Big(\intt_{\bRS}|v- v_*|^2 f'f_{*}'\dd \si \dd v_*\Big)
   \Big(\intt_{\bRS} f'f_{*}'\dd \si \dd v_*\Big)
   \dd v
   \\
   &\le &
    4\sqrt{2}b^2|\bS|\|f\|_{L^{\infty}}\|f\|_{L^1}
   \inttt_{\bRRS}\la v\ra^2\la v_*\ra ^2 ff_{*}\dd \si \dd v_*\dd v
   \\
   &= &
   4\sqrt{2}b^2|\bS|^2\|f\|_{L^{\infty}}\|f\|_{L^1}\|f\|_{L_2^1}^2 .
   \eeas
This gives (\ref{Q+f0f0}).
   \eepf
   \vskip3mm

 The last proposition below gives us very important tools which make the multi-step iterations of the collision gain operator $Q^{+}(\cdot,\cdot)$ work well. It should be noted that
in this proposition  the constants given in the explicit version are helpful in applications as will be seen
 in the proof of Theorem {\ref{main results}.

  \begin{proposition}\label{most important}
  Let $Q^+(\cdot,\cdot)$ be defined by $(\ref{Q+ bilinear})$ with the collision kernel $B$ satisfying $(\ref{condition of kernel})$.  Then

  (a)  If $0\le f\in L^1(\R^3)\cap L^2(\R^3)$, $0\le g,h\in L^1(\R^3)$,
  then
  \beqa \label{Q+(ft,Q+(fs,fs))}
  Q^+\big(f,Q^+(g,h)\big)(v)
  \le
  2^{5+\fr{2}{3}}{\pi}^{\fr{4}{3}}b^2\|f\|^{\fr{1}{3}}_{L^1}\|f\|^{\fr{2}{3}}_{L^2}
  \|g\|_{L^1}\|h\|_{L^1} \quad for \,\, all\,\, v\in {\R}^3 .
  \eeqa

  (b)  If $1\le p\le 2$ and $0\le f\in L_{\fr{2-p}{p}}^1(\R^3)$, $0\le g,h\in L_{\fr{3-p}{p}}^1(\R^3)$, then
  \beqa\label{Q+(f,Q+(g,h)) Lp}
  \big\|Q^+\big(f,Q^+(g,h)\big)\big\|_{L^p}
  \le  2^{4+\fr{2}{p}}{\pi}^{1+\fr{1}{p}}b^2
  \|f\|_{L_{\fr{2-p}{p}}^1}\|g\|_{L_{\fr{3-p}{p}}^1}\|h\|_{L_{\fr{3-p}{p}}^1}
  \eeqa
 and consequently $($with $p=1, 2$ $)$
 \beqa\label{Q+(Q+(ft,Q+(fr,fr)),Q+(fs,fs))}
  Q^+\big(Q^+(f,Q^+(g,g)),Q^+(h,h)\big)(v)
  \le
  2^{11} \pi^{3}b^4
  \|f\|_{L^1}^{\fr{2}{3}}\|f\|_{L^1_1}^{\fr{1}{3}}
  \|g\|_{L_{\fr{1}{2}}^1}^{\fr{4}{3}}
  \|g\|_{L_2^1}^{\fr{2}{3}}
  \|h\|_{L^1}^2
  \eeqa
 for all $v\in \R^3.$
  \end{proposition}

  \bepf From Lemma 2.1 in \cite{xuguang2013measure}
 we have
  \beqa\label{Q(f,Q(g,h))}
  &&Q^+\big(f,Q^+(g,h)\big)(v)\label{Lu Mouhot}
  =\inttt_{ {\R}^3\times{\R}^3\times{\R}^3 }
  K_{B}(v,v_*,w,w_*)f(v_*)g(w)h(w_*)\dd v_* \dd w \dd w_*
  \eeqa
  where $K_B:{\R^3}\times{\R}^3\times{\R}^3\times{\R}^3\to [0,\infty)$ is defined by
   \beqa
   K_{B}(v,v_*,w,w_*)=
  \left\{
  \begin{aligned}
  \fr{8}{|v-v_*||w-w_*|}\zeta\Big(\Big\la{\bf n}, \fr{2v-(w+w_*)}{|w-w_*|}\Big\ra\Big)
\int_{{\mathbb S}^1({\bf n})}\fr{B(w-w_*,\si)B(w'-v_*,\si')}
  {|w'-v_*|}{\dd}^{\perp}\omega
  \\
  {\rm if}\quad |v-v_*||w-w_*|\neq 0\qquad
  \\
  0 \quad\quad\quad\quad\quad \quad\quad\quad\quad\quad \quad\quad\quad\quad\quad
  {\rm if}\quad |v-v_*||w-w_*|= 0\qquad
  \end{aligned}
  \right.
  \eeqa
  where
  \beas
  &&{\bf n}=\fr{v-v_*}{|v-v_*|},\quad
  w'=\fr{w+w_*}{2}+\fr{w+w_*}{2}\si,\quad
  \si'=\fr{2v-v_*-w'}{|2v-v_*-w'|},\\
  &&\zeta(t):=(1-t^2)^{\fr{3}{2}}{\bf 1}_{(-1,1)}(t), \quad  t\in \R,
  \quad {\mathbb S}^1({\bf n})=\{\omega\in {\mathbb S}^2|\omega\perp{\bf n} \}
  \eeas
  and ${\dd}^{\perp}\omega$ denotes the sphere measure element of ${\mathbb S}^1({\bf n})$.
  If $|v-v_*||w-w_*|\neq 0,$ then we compute
  \beqa\label{KB}
  &&K_{B}(v,v_*,w,w_*)=
  \fr{8}{|v-v_*||w-w_*|}\zeta\Big(\Big\la{\bf n}, \fr{2v-(w+w_*)}{|w-w_*|}\Big\ra\Big)
  \int_{{\mathbb S}^1({\bf n})}\fr{B(w-w_*,\si)B(w'-v_*,\si)}
  {|w'-v_*|}{\dd}^{\perp}\omega
  \no\\
  &&\le
  b^2 \fr{8}{|v-v_*||w-w_*|}\zeta\Big(\Big\la{\bf n}, \fr{2v-(w+w_*)}{|w-w_*|}\Big\ra\Big)
  \int_{{\mathbb S}^1({\bf n})}
 \fr{|w-w_*||w'-v_*|}
  {|w'-v_*|}
  {\dd}^{\perp}\omega
  \no\\
  &&
  = 8b^2 |\mathbb S^1|\fr{1}{|v-v_*|}\zeta\Big(\Big\la{\bf n}, \fr{2v-(w+w_*)}{|w-w_*|}\Big\ra\Big)
  \le 8b^2 |\mathbb S^1|\fr{1}{|v-v_*|}.
  \eeqa
  Thus from (\ref{Lu Mouhot}), (\ref{KB}), and Lemma 2.5 in \cite{xuguang2013measure} we have for any $v\in \R^3$
  \beas
  &&Q^+\big(f,Q^+(g,h)\big)(v)
  \le
  \inttt_{{\R}^3\times{\R}^3\times{\R}^3}8b^2|\mathbb S^1|\fr{1}{|v-v_*|}f(v_*)g(w)h(w_*)\dd v_* \dd w \dd w_*
  \\
  &&\le 8b^2|\mathbb S^1|\|g\|_{L^1}\|h\|_{L^1}\int_{\R^3}\fr{1}{|v-v_*|}f(v_*)\dd v_*
  \le
  16b^2|\mathbb S^1||\bS|^{\fr{1}{3}}\|f\|^{\fr{1}{3}}_{L^1}
  \|f\|^{\fr{2}{3}}_{L^2}\|g\|_{L^1}\|h\|_{L^1}
  \eeas
  which gives (\ref{Q+(ft,Q+(fs,fs))}).

  Define
  \beas
  T_{w,w_*}(f)(v):=\int_{\bR}K_{B}(v,v_*,w,w_*)f(v_*)\dd v_*\quad {\rm for\,\, all}\,\, v,w,w_*\in \R^3.
  \eeas
  Then from (\ref{Q(f,Q(g,h))})  we have for any $1\le p\le 2$
  \beqa\label{Propositon 2.4 L1}
  \big\|Q^+\big(f,Q^+(g,h)\big)\big\|_{L^p}\le \intt_{\bRR}\|T_{w,w_*}(f)\|_{L^p}g(w)h(w_*)\dd w\dd w_*.
  \eeqa
  Next we define for any $1\le \alpha\le 2$
  \beas
  J_{\alpha}(w,w_*):=\int_{\bR}\fr{1}{|v-v_*|^{\alpha}}{\bf 1}_{(-1,1)}\Big(\Big\la{\bf n}, \fr{2v-(w+w_*)}{|w-w_*|}\Big\ra\Big)\dd v,\quad w,w_*\in \R^3.
  \eeas
Making the change of variable
  \beas
  v=v_*+\fr{|w-w_*|}{2}r\si , \quad \dd v=\Big|\fr{w-w_*}{2}\Big|^3 r^2\dd r\dd \si
  \eeas
we have
  \beas
  J_{\alpha}(w,w_*)=\Big|\fr{w-w_*}{2}\Big|^{3-\alpha}\int_{\bS}I(\la u, \si\ra)\dd \si
  \eeas
  where
  \beas
  I(\la u, \si\ra)=\int_0^{\infty}r^{2-\alpha}{\bf 1}_{\{|r+\la u, \si\ra|<1\}}\dd r,\quad u=\fr{2v_*-(w+w_*)}{|w-w_*|}.
  \eeas
   If $\la u, \si\ra\ge 1$, then $I(\la u, \si\ra)=0$; if $\la u, \si\ra< 1$, then
  \beas
  I(\la u, \si\ra)\le 2(1+|u|)^{2-\alpha}\le 2\Big(\fr{|\fr{w-w_*}{2}|+|\fr{w+w_*}{2}|+|v_*|}{|\fr{w-w_*}{2}|}\Big)^{2-\alpha}
  \le \fr{2(\la v_*\ra \la w\ra \la w_*\ra)^{2-\alpha}}{|\fr{w-w_*}{2}|^{2-\alpha}}
  \eeas
  where we used
  \beas
  &&|w-w_*|\le \langle w\rangle \langle w_*\rangle,
  \no\\
  &&|w'-v_*|=|w'-v+v-v_*|=\sqrt{|v-v_*|^2+|v-w'|^2},
  \no\\
  &&|v-v_*|\le|w'-v_*|\le \fr{|w+w_*|}{2}+\fr{|w-w_*|}{2}+|v_*|
  \le \langle w\rangle \langle w_*\rangle\langle v_*\rangle.
  \eeas
  It follows that for any $w,w_*\in \R^3$
  \beqa\label{J1}
  J_{\alpha}(w,w_*)=\Big|\fr{w-w_*}{2}\Big|^{3-\alpha}\int_{\bS}I(\la u, \si\ra)\dd \si
  \le |{\bS}||w-w_*|(\la v_*\ra \la w\ra \la w_*\ra)^{2-\alpha}.
  \eeqa
  Combining this with (\ref{KB}) we deduce
  \beas
  &&\Big(\int_{\bR}[K_B(v,v_*,w,w_*)]^p\dd v\Big)^{\fr{1}{p}}
  \le 8b^2|\mathbb S^1|
  (J_{p}(w,w_*))^{\fr{1}{p}}
\\
&&  \le 8b^2|\mathbb S^1|
  |{\bS}|^{\fr{1}{p}}|w-w_*|^{\fr{1}{p}}(\la v_*\ra \la w\ra \la w_*\ra)^{\fr{2-p}{p}}
  \\
  &&\le
  8b^2|\mathbb S^1||{\bS}|^{\fr{1}{p}}
  (\la w\ra \la w_*\ra)^{\fr{1}{p}}
  (\la v_*\ra)^{\fr{2-p}{p}}
  (\la w\ra \la w_*\ra)^{\fr{3-p}{p}},\quad 1\le p\le 2.
  \eeas
Therefore using Lemma 2.4 in \cite{xuguang2013measure} and the assumption on $f$ we conclude that for any $1\le p\le 2$
  \beqa
  \|T_{w,w_*}(f)\|_{L^p}
   \le 8b^2|\mathbb S^1||\bS|^{\fr{1}{p}}
  (\la w\ra \la w_*\ra)^{\fr{3-p}{p}} \|f\|_{L_{\fr{2-p}{p}}^1}\quad \forall\, w,w_*\in \R^3  \eeqa
  which together with (\ref{Propositon 2.4 L1})  proves (\ref{Q+(f,Q+(g,h)) Lp}).

 Finally we prove (\ref{Q+(Q+(ft,Q+(fr,fr)),Q+(fs,fs))}). According to (\ref{Q+(ft,Q+(fs,fs))}) and (\ref{Q+(f,Q+(g,h)) Lp}) (with $p=1,2$) we have for all $v\in \R^3$
  \beas
  &&Q^+\big(Q^+(f,Q^+(g,g)),Q^+(h,h)\big)(v)
  \no\\
  &&\le
  16b^2|\mathbb S^1||\bS|^{\fr{1}{3}}\|h\|_{L^1}^2
  \|Q^+(f,Q^+(g,g))\|_{L^1}^{\fr{1}{3}}
  \|Q^+(f,Q^+(g,g))\|_{L^2}^{\fr{2}{3}}
  \no\\
  &&\le
  16b^2|\mathbb S^1||\bS|^{\fr{1}{3}}\|h\|_{L^1}^2
  \big(8b^2|\mathbb S^1||\bS|\|f\|_{L_1^1}\|g\|_{L_2^1}^2\big)^{\fr{1}{3}}
  \big(8b^2|\mathbb S^1||\bS|^{\fr{1}{2}}
  \|f\|_{L^1}\|g\|^2_{L_{\fr{1}{2}}^1}\big)^{\fr{2}{3}}
  \no\\
  &&=
  2^{11} \pi^{3}b^4
  \|f\|_{L^1}^{\fr{2}{3}}\|f\|_{L^1_1}^{\fr{1}{3}}
  \|g\|_{L_{\fr{1}{2}}^1}^{\fr{4}{3}}
  \|g\|_{L_2^1}^{\fr{2}{3}}
  \|h\|_{L^1}^2.
  \eeas
  \eepf
\vskip3mm

   \section{Mild solutions of cutoff equations}
   This section is a further preparation for obtaining mild solutions of Eq.(\ref{approximate equation 2}). We will first prove the existence of mild solutions $\{f^n\}_{n=1}^{\infty}$ of Eq.(\ref{approximate equation 1}). Then we will show that $\{f^n(t)\}_{n=1}^{\infty}$ satisfy the $L^1$ relative compactness conditions, etc.\, Throughout this section, the constant $K>0$ in Eq.$(\ref{approximate equation 1})$ is fixed.

\subsection{Existence of mild solutions of cutoff equations}
In order to prove the global in time existence of mild solutions of Eq.(\ref{approximate equation 1}), we first use fixed point theorem of contractive self-mappings to
 prove the existence of mild solutions of Eq.(\ref{approximate equation 1}) on
 $[0,T]\times{\R}^3$ for some $0<T<\infty$, then we will prove the conservation of the mass of $f$ and extend $f$ to $[0,\infty)\times {\R}^3$ so that the extended $f$ is a mild solution of Eq.(\ref{approximate equation 1}) on $[0,\infty)\times{\R}^3$.

\begin{proposition}\label{fn}
For any $0\le f_0\in L_2^1({\bR})$ and any $n\in \mathbb{N}$, there exists a conservative mild solution $f^n$
of Eq.$(\ref{approximate equation 1})$ on $[0,\infty)\times{\R}^3$
with the initial datum $f_0$ such that $($see  Remark \ref{remark of mild solutions}$)$
\beas
f^n(t,v)=f_0(v)+\int_{0}^{t}Q_{n,K}(f^n)(\tau,v){\rm d}\tau\qquad \forall\, (t,v)\in [0,\infty)\times\R^3
\eeas
where $Q_{n,K}(\cdot)$ is defined in $(\ref{Qnk})$ with collision kernel $B$ satisfying $(\ref{condition of kernel})$.
\end{proposition}

\bepf  {\bf Step 1.}  For any $T\in (0,\infty)$ recall the definition of  $L^{\infty}([0, T]; L^1({\bR}))$ in (\ref{definition of L(infty)([0,T];L1)}) and define
$${\cal A}_T=\{ f\in L^{\infty}([0, T]; L^1({\bR})) \,|\,\,
\|f\|:=\sup_{t\in [0, T]}\|f(t)\|_{L^1}
\le 2\|f_0\|_{L^1}\}$$
with the distance
$$\|f-g\|=\sup_{t\in [0, T]}\|f(t)-g(t)\|_{L^1},\quad f,g\in {\cal A}_T.$$
From Lemma \ref{lemma of L(infty)([0,T];L1)}, it is obvious that $ {\cal A}_T $ is a closed subset of the Banach space
$(L^{\infty}([0, T]; L^1({\bR}))$ and so $ {\cal A}_T $ is complete. Fix $n\in {\mN}$.
Let
$$ J_n(f)(t,v)=f_0(v)+\int_{0}^{t}Q_{n,K}(|f|)(\tau,v){\rm d}\tau,\quad (t,v)\in [0,\infty)\times\R^3,\quad f\in
{\cal A}_T.
$$
To obtain contractiveness of $J_n$  we choose
$$T=T_n=\fr{1}{16(1+2K+2^{3/2})|{\bS}|n\|f_0\|_{L^1}}.$$
Now we need to prove that
$J_n:{\cal A}_{T_n}\to {\cal A}_{T_n}$ and
\beqa\label{contractiven condition}
\sup_{t\in [0, T_n]}\|J_n(f)(t)-J_n(g)(t)\|_{L^1}
\le \fr{1}{2}\sup_{t\in [0, T_n]}\|f(t)-g(t)\|_{L^1},\quad f,g\in {\cal A}_{T_n}.
\eeqa
In fact for any $f\in {\cal A}_{T_n}$  and any $t\in [0,T_n]$ we have
\beas&& \int_{{\bR}}\big|Q_{n,K}(|f|)(t,v)\big|{\rm d}v
\le
\inttt_{{\bRRS}}
B_n(v-v_*,\sg)(|f'|\wedge n)(|f_*'|\wedge n)
(1+| f|\wedge K+| f_*|\wedge K)
{\rm d}\sg{\rm d}v_*{\rm d}v
\\
&&\quad +
\inttt_{{\bRRS}}
B_n(v-v_*,\sg)(|f|\wedge n)(|f_*|\wedge n)
(1+| f'|\wedge K+| f_*'|\wedge K)
{\rm d}\sg{\rm d}v_*{\rm d}v
\\
&&=2{\cal I}_n(|f(t)|,0)\le 2(1+2K)|{\bS}|n (\|f(t)\|_{L^1})^2
\le 8(1+2K)|{\bS}|n (\|f_0\|_{L^1})^2<\infty
\eeas
where we used (\ref{identity}) and Proposition \ref{I(f,g)}.
This implies that $(t,v)\mapsto J_n(f)(t,v)$ is
measurable on $[0,T_n]\times\bR$ and
\beqa\label{|fn|}
&&\int_{{\bR}}|
J_n(f)(t,v)|{\rm d}v
\le \|f_0\|_{L^1}
+\int_{0}^{t}{\rm d}\tau\int_{{\bR}}\big|Q_{n,K}(|f|)(\tau,v)\big|{\rm d}v
\no\\
&&\le
\|f_0\|_{L^1}
+8(1+2K)|{\bS}|n\|f_0\|_{L^1}T_n \|f_0\|_{L^1}
\le 2\|f_0\|_{L^1}\quad \forall\,t\in [0, T_n].\eeqa
Thus
$\sup\limits_{t\in [0, T_n]}\|J_n(f)(t)\|_{L^1}
\le 2\|f_0\|_{L^1}$ and so $J_n(f)\in {\cal A}_{T_n}$.
Next for any $f, g\in {\cal A}_{T_n}$ and any $t\in [0,T_n]$ we have (by using Proposition \ref{I(f,g)})
\beqa\label{J(f)-J(g)}
&&\|J_n(f)(t)-J_n(g)(t)\|_{L^1}
\le \int_{0}^{t}{\rm d}\tau\int_{{\bR}}\big|Q_{n,K}(|f|)(\tau,v)-Q_{n,K}(|g|)(\tau,v)
\big|{\rm d}v
\no\\
&&=2\int_{0}^{t}{\cal I}_{n,k}(|f(\tau)|, |g(\tau)|){\rm d}\tau
\le
2\int_{0}^{t}(1+2K)|{\bS}|n (\|f(\tau)\|_{L^1}+\|g(\tau)\|_{L^1})\|f(\tau)-g(\tau)\|_{L^1}
\no\\
&&\quad +2 n 2^{5/2}|{\bS}|\|g(\tau)\wedge n\|_{L^{\infty}}\|g(\tau)\|_{L^1}\|f(\tau)-g(\tau)\|_{L^1}
{\rm d}\tau
\no\\
&&\le
2\int_{0}^{t}(1+2K)|{\bS}|4n\|f_0\|_{L^1}\|f(\tau)-g(\tau)\|_{L^1}
+2 n 2^{5/2}|{\bS}|\|f_0\|_{L^1}\|f(\tau)-g(\tau)\|_{L^1}
{\rm d}\tau
\no\\
&&\le 8(1+2K+2^{3/2} )|{\bS}|n\|f_0\|_{L^1} T_n\sup_{t\in [0, T_n]}
\|f(t)-g(t)\|_{L^1}
\le \fr{1}{2}\sup_{t\in [0, T_n]}
\|f(t)-g(t)\|_{L^1}.
\eeqa
This proves (\ref{contractiven condition}). Thus  there exists a unique $f\in {\cal A}_{T_n}$ such that $J_n(f)=f.$

 {\bf Step 2.} Given any  $n\in \mathbb{N}$. Let $T=T_n>0$ be defined in {\bf Step 1} and let $f^n:=f$  obtained in {\bf Step 1} be the unique fixed point of $J_n:{\cal A}_{T_n}\to {\cal A}_{T_n}$. From Proposition \ref{proposition2} and (\ref{|fn|}) we have for any $(t,v)\in [0,T_n]\times \R^3$
 \beas
 &&\big|Q_{n,K}^{+}(|f^n|)(t,v)\big|
 \le \intt_{{\bRS}}
 B_n(v-v_*,\sg)(|{f^n}'|\wedge n)(|{f_*^n}'|\wedge n)
(1+ |f^n|\wedge K+ |{f_*^n}|\wedge K){\rm d}\sg{\rm d}v_*
\\
&&\le
  (1+2K)n\intt_{{\bRS}}
 (|{f^n}'|\wedge n)(|{f_*^n}'|\wedge n){\rm d}\sg{\rm d}v_*
  \le 2^{\fr{5}{2}}(1+2K)n^2|\bS|\|f(t)\|_{L^1}
  \le 2^{\fr{7}{2}}(1+2K)n^2|\bS|\|f_0\|_{L^1},\\
 &&|Q_{n,K}^{-}(|f^n|)(t,v)|
 \le
n^2(1+2K)\intt_{{\bRS}}|f_*^n|{\rm d}\sg{\rm d}v_*
\le
2(1+2K)n^2|\bS|\|f_0\|_{L^1}.
 \eeas
     Fix any $v\in \R^3$. The function $t\mapsto f^n(t,v)$  is  absolutely continues on $[0,T_n]$. Since $(\cdot)^+$ is Lipschitz continues, it follows from Lemma \ref{Lemma7} that $t\mapsto (-f^n(t,v))^+$ is absolutely continues about $t\in [0,T_n]$. Since $f^n(0,v)=f_0(v)\ge 0$, it follows that
    for all $(t,v)\in[0,T_n]\times \R^3$
 \beas&&
 (-f^n(t,v))^+=\int_0^t\Big(-\fr{\dd}{\dd \tau}{f^n}(\tau,v){\bf 1}_{\{{f^n}(\tau,v)\le 0\}}\Big)\dd\tau
 =\int_0^t\Big(-Q_{n,k}(|{f^n}|)(\tau, v){\bf 1}_{\{f(\tau, v)\le 0\}}\Big)
 {\rm d}\tau\\
 &&=\int_0^t\Big (Q_{n,k}^{-}(|{f^n}|)(\tau, v){\bf 1}_{\{f(\tau, v)\le 0\}}
 -Q_{n,k}^{+}(|{f^n}|)(\tau, v){\bf 1}_{\{f(\tau, v)\le 0\}}\Big){\rm d}\tau
\\
&&\le \int_0^tQ_{n,k}^{-}(|{f^n}|)(\tau, v){\bf 1}_{\{f(\tau, v)\le 0\}}
 {\rm d}\tau
\\
&&\le (1+2K)n|{\bS}|\int_0^t\Big(\int_{\bR}
|f^n(\tau, v_*)|{\rm d}v_*\Big) |f^n(\tau, v)|{\bf 1}_{\{{f^n}(\tau, v)\le 0\}}{\rm d}\tau
\\
&&\le 2(1+2K)n|{\bS}|\|f_0\|_{L^1}\int_0^t(-{f^n}(\tau,v))^+{\rm d}\tau.
 \eeas
By Gronwall's lemma we conclude that $(-f^n(t,v))^+= 0$ for all $(t,v)\in [0,T_n]\times \R^3.$
Thus we have proved $f^n(t,v)\ge 0$ on $[0,T_n]\times \R^3$
and so $f^n$ satisfies
 $$f^n(t,v)=f_0(v)+\int_{0}^{t}Q_{n,K}(f^n)(\tau,v){\rm d}\tau
 \qquad \forall\,(t,v)\in[0,T_n]\times {\mR}^3.$$
Next we prove that $f^n(t)$ conserves the mass on $[0,T_n]$.
 For any $\vp \in L^{\infty}(\R^3)$ and  any $0\le t\le T_n$, we have by using (\ref{identity}) that
\beas
&&\int_{\R^3}\vp(v)(Q^{+}_{n,K}(f^n)(t,v)+Q^{-}_{n,K}(f^n)(t,v))\dd v
\\
&&\le
(1+2K)n\|\vp\|_{L^{\infty}}\iiint_{\bRRS}f^n(t,v'){f^n}(t,v_*')
+f^n(t,v)f^n(t,v_*)
{\rm d}\sg{\rm d}v_*\dd v
\\
&&\le
2(1+2K)n\|\vp\|_{L^{\infty}}|\bS|(\sup_{0\le t\le T_n }\|f^n(t)\|_{L^1})^2
\le
2(1+2K)n\|\vp\|_{L^{\infty}}|\bS|(2\|f_0\|_{L^1})^2.
\eeas
This integrability allows us to compute (as usual)
\beas
&&\int_{{\bR}}\vp(v)
Q_{n,K}(f^n)(t,v){\rm d}v
\no \\
&&=\int_{{\bR}}\fr{\vp(v)+\vp(v_*)}{2}\intt_{{\bRS}}
B_n(v-v_*,\sg)\Big( ({f^n}'\wedge n)({f_{*}^n}'\wedge n)
(1+ {f^n}\wedge K+ {f_{*}^n}\wedge K)\no\\
&&\quad-({f^n}\wedge n)({f_{*}^n}\wedge n)
(1+ {f^n}'\wedge K+ {f_{*}^n}'\wedge K)
\Big){\rm d}\sg{\rm d}v_*\dd v\no\\
&&=\fr{1}{2}\inttt_{{\bRRS}}
B_n(v-v_*,\sg)
(f^n\wedge n)(f^n_{*}\wedge n)
(1+ {{f}^n}'\wedge K+ {f_{*}^n}'\wedge K)\no\\
&&\quad\times
\big(\vp(v')+\vp(v_*')-\vp(v)-\vp(v_*)\big){\rm d}\sg{\rm d}v_*{\rm d}v.
\eeas
According to Fubini's theorem it follows that
\beqa\label{test fuc}
&&
\int_{{\bR}}\vp(v)f^n(t,v){\rm d}v=\int_{{\bR}}\vp(v)f_0(v){\rm d}v
+\fr{1}{2}\int_{0}^{t}{\rm d}\tau\inttt_{{\bRRS}}
B_n(v-v_*,\sg)
({f}^n\wedge n)(f^n_{*}\wedge n)\no\\
&&\quad\times(1+ {f^n}'\wedge K+ {f_{*}^n}'\wedge K)
\big(\vp(v')+\vp(v_*')-\vp(v)-\vp(v_*)\big){\rm d}\sg{\rm d}v_*{\rm d}v,\quad t\in [0,T_n].
\eeqa
Taking $\vp(v)\equiv 1$ gives the conservation of mass of $f^n$:
$$\int_{{\bR}}f^n(t,v){\rm d}v=\int_{{\bR}}f_0(v){\rm d}v
=\|f_0\|_{L^1}<\infty, \quad \forall\, t\in[0,T_n].$$
Since $f^n$ is nonnegative and conserves the mass, it follows from the choice of $T_n$
that with the same number $T_n>0$, the function $f^n$ can be extended from
the time interval $[0, T_n]$ to $[T_n, 2T_n],[2T_n, 3T_n], ...$\,.
Therefore we obtain a function $0\le f^n\in L^{\infty}([0,\infty); L^1({\bR}))$
which conserves the mass and satisfies
 \beas
 f^n(t,v)=f_0(v)+\int_{0}^{t}Q_{n,K}(f^n)(\tau,v){\rm d}\tau\qquad \forall\,(t,v)\in [0,\infty)\times{\mR}^3.
 \eeas
To prove that $f^n$ also conserves  momentum and energy we consider truncation
$\vp(v)=\la v\ra^{s}\wedge R$ with $s\ge 1$ and $0<R<\infty$.
From Lemma \ref{LemmaF} we have
\beqa\label{vp estimate}
&&\vp(v')+\vp(v_*')-\vp(v)-\vp(v_*)
\le
\vp(v')+\vp(v_*')
\le 2^{s}(\la v\ra^{s}\wedge R)+2^{s}(\la v_*\ra^{s}\wedge R).
\eeqa
Suppose $f_0\in L_s^1(\R^3)$. Then, since $\vp\in L^{\infty}(\R^3)$ and $f^n$ conserves the mass, it follows from
(\ref{test fuc}) and (\ref{vp estimate}) that
\beas&&
\int_{{\bR}}(\la v\ra^{s}\wedge R)f^n(t,v){\rm d}v
=\int_{{\bR}}(\la v\ra^{s}\wedge R)f_0(v){\rm d}v
+\fr{1}{2}\int_{0}^{t}{\rm d}\tau\inttt_{{\bRRS}}
B_n(v-v_*,\sg)
(f^n\wedge n)(f^n_{*}\wedge n)\\
&&\quad\times(1+ {f^n}'\wedge K+ {f_{*}^n}'\wedge K)
\big(\vp(v')+\vp(v_*')-\vp(v)-\vp(v_*)\big){\rm d}\sg{\rm d}v_*{\rm d}v\\
&&\le
\|f_0\|_{L^1_s}
+2^{s} (1+2K)n|{\bS}|\int_{0}^{t}{\rm d}\tau\intt_{{\bRR}}
(\la v\ra^{s}\wedge R)(f^n\wedge n)(f^n_{*}\wedge n)
{\rm d}v_*{\rm d}v
\\
&&\le
\|f_0\|_{L^1_s}
+2^{s} (1+2K)n|{\bS}|\|f_0\|_{L^1}\int_{0}^{t}{\rm d}\tau\int_{{\bR}}
(\la v\ra^{s}\wedge R)f^n(\tau,v){\rm d}v,\quad t\ge 0.
\eeas
By Gronwall's lemma this gives
$$\int_{{\bR}}(\la v\ra^{s}\wedge R)f^n(t,v){\rm d}v
\le  \|f_0\|_{L^1_s}\exp\big(2^{s} (1+2K)n|{\bS}|\|f_0\|_{L^1}
t\big),\quad t\ge 0.$$
Letting $R\to\infty$
we conclude from Fatou's lemma that $f^n(t,\cdot)\in L^1_s({\mR})$ and
$$\|f^n(t)\|_{L^1_s}=\int_{{\bR}}\la v\ra^{s}f^n(t,v){\rm d}v
\le  \|f_0\|_{L^1_s}\exp\big(2^{s} (1+2K)n|{\bS}|\|f_0\|_{L^1}
t\big),\quad t\ge 0.$$
Next for any $s\ge 1$  and any $0<T<\infty$ we have by using (\ref{identity}) that
\beas
&&\sup_{t\in [0, T]}\int_{\R^3}\la v\ra^s \Big(Q^{+}_{n,K}(f^n)(t,v)+Q^{+}_{n,K}(f^n)(t,v)\Big)\dd v
\\
&&\le \sup_{t\in [0, T]}
2n(1+2K)\iiint_{\bRRS}2^s \la v\ra^s \la v_*\ra^sf^n(t,v){f^n}(t,v_*)
{\rm d}\sg{\rm d}v_*\dd v
\\
&&\le
2n(1+2K)2^s|\bS|\big[\|f_0\|_{L^1_s}\exp\big(2^{s} (1+2K)n|{\bS}|\|f_0\|_{L^1}
T\big)\big]^2< \infty.
\eeas
This means that the test function can be chosen as $\vp(v)=\la v \ra^s$. Thus (again recall definition
of $\|\cdot\|_{L^1_s}$)
\beqa\label{derivative}
&&\|f^n(t)\|_{L^1_s}=\|f^n_0\|_{L^1_s}
+\fr{1}{2}\int_{0}^{t}{\rm d}\tau\inttt_{{\bRRS}}
B_n(v-v_*,\sg)
(f^n\wedge n)(f^n_{*}\wedge n)\no\\
&&\quad\times(1+ {f^n}'\wedge K+ {f_{*}^n}'\wedge K)
\big(\la v'\ra^s+\la v_*'\ra^s-\la v\ra^s-\la v_*\ra^s\big){\rm d}\sg{\rm d}v_*{\rm d}v, \,\, t\in [0,\infty).
\eeqa
Take $s=2$. Since $f_0 \in L_2^1(\R^3), \la v\ra^2=1+|v|^2$, and
$\la v'\ra^2+\la v_*'\ra^2-\la v\ra^2-\la v_*\ra^2=0$, and since $f^n$ conserves the mass, it follows that
$f^n$ conserves the energy:
\beas
\int_{\R^3}|v|^2 f^n(t,v)\dd v=\int_{\R^3}|v|^2 f_0(v)\dd v,\quad t\in [0,\infty).
\eeas
Let $\vp(v)=\la {\bf e}_i, v\ra$, where ${\bf e}_1=(1,0,0), {\bf e}_2=(0,1,0), {\bf e}_3=(0,0,1)$.
Then a similar proof shows also that $f^n$ conserves the momentum.
Summarizing the above results we have proved that $\{f^{n}\}_{n=1}^{\infty}$ are conservative  mild solutions of Eq.(\ref{approximate equation 1}).
 \eepf
\vskip3mm

\subsection{Moment estimates}

In order to prove the $L^1$ relative compactness of solutions $\{f^n\}_{n=1}^{\infty}$ of Eq.(\ref{approximate equation 1}),
we need to establish uniform moment estimates of $\{f^n\}_{n=1}^{\infty}$.

\begin{proposition}\label{moment propogation}
For any $n\in \mathbb{N}$, let
$B_n=B\wedge n$ with $B$ the collision kernel satisfying $(\ref{condition of kernel})$ and let
$Q_{n,K}(\cdot)$ be
the collision operator defined in $(\ref{Qnk})$.
For any $0\le f_0\in L^1_s(\R^3)$ with $3\le s<\infty$, let $f^n$
be a conservative mild solution of Eq.$(\ref{approximate equation 1})$ corresponding to the kernel $B_n$ with the initial datum $f_0$.
Then
\beqa\label{moment estimate}
\sup_{n\ge 1}\|f^n(t)\|_{L^1_s}
\le
 \Big(\|f_0\|_{L^1_s}^{\fr{1}{s-2}}
+(1+2K)2^{\fr{s}{2}+2}|\bS|
b(\|f_0\|_{L^1_2})^{\fr{s-1}{s-2}}\fr{t}{s-2}\Big)
^{s-2},\quad t\in [0,\infty).
\eeqa
\end{proposition}

\bepf From (\ref{derivative}) and Corollary \ref{LemmaE}, we have for a.e. $t\in [0,T]$
\beas
&&
\fr{{\rm d}}{{\rm d}t}
\|f^n(t)\|_{L^1_s}
=\fr{1}{2}\inttt_{{\bRRS}}
B_n(v-v_*,\sg)
(f^n\wedge n)(f^n_*\wedge n)
(1+ {f^n}'\wedge K+ {f^n}_*'\wedge K)\\
&&\quad\times
(\la v'\ra^s+\la v_*'\ra^s-\la v\ra^s-\la v_*\ra^s){\rm d}\sg{\rm d}v_*{\rm d}v
\\
&&\le
\inttt_{{\bRRS}}
B_n(v-v_*,\sg)
(f^n\wedge n)(f^n_*\wedge n)
(1+ {f^n}'\wedge K+ {f^n}_*'\wedge K)\\
&&\quad\times 2^{s/2}(\la v\ra^{s-2}\la v_*\ra^2+
\la v\ra^2\la v_*\ra ^{s-2})
{\rm d}\sg{\rm d}v_*{\rm d}v
\\
&&
\le (1+2K)2^{\fr{s}{2}}|\bS|b
 \intt_{{\bRR}}
(\la v\ra+\la v_*\ra)(\la v\ra^{s-2}\la v_*\ra^2+
\la v\ra^2\la v_*\ra ^{s-2})f^nf^n_*
{\rm d}v_*{\rm d}v
\\
&&=
(1+2K)2^{\fr{s}{2}+1}|\bS|b
 \big(\|f_0\|_{L^1_2}\|f^n(t)\|_{L^1_{s-1}}
+\|f^n(t)\|_{L^1_{3}}
\|f^n(t)\|_{L^1_{s-2}}\big).
\eeas
Recall that for any $2<p\le s$, by writing
$p=2\fr{s-p}{s-2}+s\fr{p-2}{s-2}$ and using H\"{o}lder's inequality we have
\beas&& \|f\|_{L^1_p}
\le
(\|f\|_{L^1_2})^{\fr{s-p}{s-2}}(\|f\|_{L^1_s})^{\fr{p-2}{s-2}},\qquad  f\in L^1_s({\bR}).
\eeas
It follows that
\beas&& \fr{{\rm d}}{{\rm d}t}
\|f^n(t)\|_{L^1_s}
\le
(1+2K)2^{\fr{s}{2}+2}|\bS|b
 \big(\|f_0\|_{L^1_2}\big)^{\fr{s-1}{s-2}}\big(\|f^n(t)\|_{L^1_s}\big)^{\fr{s-3}{s-2}}
 :=A \|f^n(t)\|_{L^1_s}^{\fr{s-3}{s-2}} \quad {\rm for\,\,\, a.e.}\,\, t\in [0,T]
.\eeas
Thus denoting $\alpha=\fr{1}{s-2}$ we have with $u(t):=\|f^n(t)\|_{L^1_s}$ that
$$([u(t)]^{\alpha})'
=\alpha [u(t)]^{\alpha-1}u'(t)
\le \alpha [u(t)]^{\alpha-1}A [u(t)]^{1-\alpha}
=A\alpha\qquad  {\rm a.e.}\quad t\in [0, T].$$
Here we note that by conservation of mass we have $u(t)=\|f^n(t)\|_{L^1_s}\ge \|f_0\|_{L^1}>0$
and from (\ref{derivative}) we know that $u(t)$ is absolutely continuous on
any bounded interval.  It follows that $t\mapsto [u(t)]^{\alpha}$ is also absolutely continuous on
any bounded interval. Thus
$$[u(t)]^{\alpha}\le [u(0)]^{\alpha}+A\alpha t,\quad
u(t)\le \big([u(0)]^{\alpha}+A\alpha t\big)^{\fr{1}{\alpha}}\qquad \forall\, t\in [0, T]$$
and we conclude
$$\sup_{n\ge 1}\|f^n(t)\|_{L^1_s}
\le
 \Big(\|f_0\|_{L^1_s}^{\fr{1}{s-2}}
+(1+2K)2^{\fr{s}{2}+2}|\bS|
b(\|f_0\|_{L^1_2})^{\fr{s-1}{s-2}}\fr{t}{s-2}\Big)
^{s-2},\quad t\in [0, T].$$
\eepf
\vskip3mm

 \subsection{$Q^{+}$ iteration and $L^1$ compactness}

Our first use of multi-step iterations of $Q^{+}$ is in the proof of the following proposition
which gives $L_2^1$-stability estimates for mild solutions of Eq.$(\ref{approximate equation 1})$, Eq.$(\ref{approximate equation 2})$,  and Eq.(\ref{Equation}). One will see that
it is this  multi-step iterations of $Q^{+}$ that enable us to obtain useful $L_2^{\infty}$ estimates of
solutions.

 \begin{proposition}\label{stable}  Let $B$ the collision kernel satisfying $(\ref{condition of kernel})$.
For any fixed $n\in \mathbb{N}$, let
$B_n=B\wedge n$ and let
$Q(\cdot)$, $Q_{K}(\cdot)$, $Q_{n,K}(\cdot)$ be the collision operators defined in $(\ref{Q(f)})$, $(\ref{QK})$, $(\ref{Qnk})$ respectively.
With $0\le f_0,g_0\in L^1_3(\R^3)\cap L_3^{\infty}(\R^3)$ being initial data, let
 $f, g$ be the conservative mild solutions of Eq.$(\ref{approximate equation 2})$ and
$f=f^n, g=g^n$ the conservative mild solutions of Eq.$(\ref{approximate equation 1})$, respectively.
Also assume in both cases that $f,g$ satisfy the estimate $(\ref{moment estimate})$  with $s=3$, i.e.
\beqa\label{L31 condition}
&&\|f(t)\|_{L_3^1}\le
\|f_0\|_{L^1_3}
+(1+2K)2^{\fr{7}{2}}b |\bS|\|f_0\|_{L^1_2}^{2}t,\quad t\in [0,\infty)\no \\
&&\|g(t)\|_{L_3^1}\le
\|g_0\|_{L^1_3}
+(1+2K)2^{\fr{7}{2}}b |\bS|\|g_0\|_{L^1_2}^{2}t,\quad t\in [0,\infty).\eeqa
Then
\beqa\label{f-g}
\sup_{ t\in [0, T]}\|f(t)-g(t)\|_{L^1_2}\le C_{T,K}\|f_0-g_0\|_{L^1_2}\quad \forall\, 0<T<\infty\eeqa
where $C_{T,K}=C_{T,K}(\|f_0\|_{L^{\infty}_3},\|f_0\|_{L^1_3},
 \|g_0\|_{L^{\infty}_3},\|g_0\|_{L^1_3})<\infty$ is independent of $n$,  $[0,\infty)^4\ni (y_1,y_2,y_3, y_4)\mapsto C_{T,K}(y_1,y_2,y_3,y_4)$ is a continuous function and is
 monotone non-decreasing with respect to each $y_1,y_2,y_3,$ $ y_4\in [0,\infty)$.

 Furthermore, for any $0\le f_0,g_0 \in L^1_3(\R^3)\cap L_3^{\infty}(\R^3)$, let $0\le f, g \in L^{\infty}_{{\rm loc}}([0,\infty); L^1_3({\bR})\cap L^{\infty}({\bR}))$
be  mild solutions of Eq.$(\ref{Equation})$ corresponding to the kernel $B$ with the initial data $f_0,g_0$ respectively.
Then
\beqa\label{Theorem proposition}
\sup_{t\in [0, T]}\|f(t)-g(t)\|_{L^1_2}\le C_{T}\|f_0-g_0\|_{L^1_2} \quad \forall\, 0<T<\infty
\eeqa
where $C_{T}<\infty$ depends only on $T$ and $\|f_0\|_{L^{\infty}_3},\sup\limits_{t\in [0, T]}\|f(t)\|_{L^1_3},\sup\limits_{t\in[0,T]}\|f(t)\|_{L^{\infty}},
 \|g_0\|_{L^{\infty}_3},\sup\limits_{t\in [0, T]}\|g(t)\|_{L^1_3},$ $ \sup\limits_{t\in[0,T] }\|g(t)\|_{L^{\infty}}.$
\end{proposition}

\bepf The proof is divided into two steps.

{\bf Step 1.}
Let $Q^{+}(\cdot,\cdot)$ be defined in (\ref{Q+ bilinear}) and fix any $0<T<\infty$, and let $Z\subset \R^3$ is a null set appeared in Definition \ref{mild solutions}.  Notice first that under the assumption in the proposition we have
\beqa\label{iteration1} f(t,v)\le f_0(v)+(1+2K)\int_{0}^{t}Q^{+}(f_{\tau_1}, f_{\tau_1})(v){\rm d}\tau_1,\quad (t,v)\in [0,T]\times(\R^3\setminus Z) \eeqa
where to shorten notation we denote the solution $f_t(v)\equiv f(t,v)$.
Using the inequality (\ref{iteration1}) to $f(\tau_1,v)$ and substituting the right hand side of the inequality
into the two arguments of $Q^{+}(f_{\tau_1}, f_{\tau_1})$, and then making further iteration,  we compute for any $(t,v)\in [0,T]\times (\R^3\setminus Z)$ that
\beqa\label{fn infinity}
&&f(t,v)\le f_0(v)\quad\no \\
&&\quad +(1+2K)\int_{0}^{T}Q^{+}\Big(f_0+
(1+2K)\int_{0}^{T}Q^{+}(f_{\tau_{2}},f_{\tau_{2}})\dd {\tau_{2}},\,
f_0+
(1+2K)\int_{0}^{T}Q^{+}(f_{\tau_{3}},f_{\tau_{3}})\dd {\tau_{3}}
\Big)(v){\rm d}\tau_1
\no\\
&&\le
f_0(v)+(1+2K)\int_{0}^{T}Q^{+}(f_0,f_0)(v){\rm d}\tau_1+2(1+2K)^2\int_{0}^{T}
\int_{0}^{T}Q^{+}(Q^{+}(f_{\tau_2}, f_{\tau_2}), f_0)(v){\rm d}\tau_2{\rm d}\tau_1
\no\\
&&\quad +(1+2K)^3\int_{0}^{T}\int_{0}^{T}\int_{0}^{T}Q^{+}\big(
Q^{+}(f_{\tau_2},f_{\tau_2}), Q^{+}(f_0, f_0)\big)(v)\dd {\tau_3} \dd {\tau_2}{\rm d}\tau_1
\no\\
&&\quad +(1+2K)^4\int_{0}^{T}\int_{0}^{T}\int_{0}^{T}\int_{0}^{T}Q^{+}\big( Q^{+}(f_{\tau_2},f_{\tau_2}),
Q^{+}(f_0, Q^{+}(f_{\tau_4},f_{\tau_4}))
\big)(v)\dd {\tau_4}\dd {\tau_3} \dd {\tau_2}{\rm d}\tau_1
\no\\
&&\quad +(1+2K)^4\int_{0}^{T}\int_{0}^{T}\int_{0}^{T}\int_{0}^{T}Q^{+}\big(Q^{+}(f_{\tau_2},f_{\tau_2}), Q^{+}(f_{\tau_3}, Q^{+}(f_{\tau_4},f_{\tau_4}))
\big)(v)\dd {\tau_4}\dd {\tau_3} \dd {\tau_2}{\rm d}\tau_1.
\no\\
\eeqa
Next, using Proposition \ref{proposition2} and $|v-v_*|=|v'-v_*'|\le \la v'\ra \la v_*'\ra$ we have
\beqa\label{Q+f0f0 infty}
Q^{+}(f_0,f_0)(v)
\le b\intt_{{\bRS}}\la v'\ra \la v_*'\ra f_0(v') f_0(v_*'){\rm d}\sg{\rm d}v_*
\le
2^{\fr{5}{2}}b|\bS|\|f_0\|_{L^{\infty}_1}\|f_0\|_{L^{1}_{1}},\quad v\in \R^3.
\eeqa
In order to use Proposition \ref{most important} to obtain uniform estimates of  $\|f(t)\|_{L_2^{\infty}}$ and $\|g(t)\|_{L_2^{\infty}}$ let us define
$$(f)_s(v):=\la v\ra^s f(v)\quad {\rm with}\quad s\ge 0.$$
It is easily seen that for any
  $0\le f,g,h\in L^1_s({\mR}^3)$ and any $v\in \R^3$ we have
  \beas
  \big(Q^+(f,g)\big)_s(v)&\le& Q^+\big((f)_s,(g)_s\big)(v),\\
  \big(Q^+(f,Q^+(g,h))\big)_s(v)&\le& Q^+\big((f)_s,Q^+((g)_s,(h)_s)\big)(v)
  \eeas
where we used the fact that $\la v\ra\le \la v'\ra\la v_*'\ra$.
From these we see that the inequality (\ref{iteration1}) hence
all inequalities in (\ref{Q+f0f0 infty}) hold also for the function
$(f)_s(t,v)=\la v\ra ^sf(t,v)$.  In other words, this means that if $f$ is a solution
of the inequality (\ref{iteration1}), so is $(f)_s$ for any $s\ge 0$.
Take $s=2$.  Then combining this with
(\ref{exchange}), (\ref{Q+(f0,ft)L1})$-$(\ref{Q+(Q+(ft,Q+(fr,fr)),Q+(fs,fs))}), (\ref{moment estimate}), (\ref{L31 condition}), (\ref{fn infinity}) and (\ref{Q+f0f0 infty}) we obtain a uniform estimate:
\beqa\label{fn T Linfty}
&& \sup_{0\le t\le T}\big(\|f(t)\|_{L^{\infty}_2}+\|f(t)\|_{L^1_3}+
 \|g(t)\|_{L^{\infty}_2}+\|g(t)\|_{L^1_3}\big) \no  \\
&&\qquad  \le \wt{C}_{T,K}
 \big(\|f_0\|_{L^{\infty}_3},\|f_0\|_{L^1_3},
 \|g_0\|_{L^{\infty}_3},\|g_0\|_{L^1_3}\big)=: \wt{C}_{T,K}<\infty\eeqa
where and below $\wt{C}_{T,K}(y_1,y_2,y_3,y_4)$ denotes a continuous function
on $[0,\infty)^4$ which is
 monotone non-decreasing with respect to each variable $y_1,y_2,y_3,y_4\in [0,\infty)$.

In order for our proof of stability estimates to cover both Eq.$(\ref{approximate equation 2})$ and Eq.$(\ref{approximate equation 1})$, we denote $Q_{*}=Q_K, Q^{\pm}_{*}=Q^{\pm}_K$ and $Q_{*}=Q_{n,K}, Q^{\pm}_{*}=Q^{\pm}_{n,K}$ respectively, and let
$$\psi(t,v)={\rm sign}(f(t,v)-g(t,v)).$$
Then we have
$$
|f(t,v)-g(t,v)|
=|f_0(v)-g_0(v)|+\int_{0}^{t}\big(Q_{*}(f)(\tau,v)-Q_{*}(g)(\tau,v)\big)\psi(\tau,v){\rm d}\tau.$$
For further estimates we need to show that $Q_{*}^{\pm}(f), Q_{*}^{\pm}(g)$ belong to $L^{\infty}([0, T]; L^1_2({\mR}^3))$ so that there is no problem of integrability.  In fact we have,
for instance for $Q_{*}^{\pm}(f)$,
\beqa\label{integrability}
&&\int_{\R^3} \la v\ra ^2 \big(Q^{+}_{*}(f)(t,v)+Q^{-}_{*}(f)(t,v)\big)\dd v\no\\
&&
\le (1+2K)\iiint_{\bRRS}\la v\ra^2 B(v-v_*,\si)({f}'{f}_*'+ff_*)\dd \si \dd v_* \dd v
  \no\\
  &&\le (1+2K)b\iiint_{\bRRS}(1+|v'|^2+|v_*'|^2)|v-v_*|({f}'{f}_*'+ff_*)\dd \si \dd v_* \dd v
  \no\\
  &&\le 2(1+2K)b\iiint_{\bRRS}\la v\ra^3 \la v_*\ra^3 {f}{f}_*\dd \si \dd v_* \dd v
  =2(1+2K)b|\bS| \|f(t)\|_{L^1_3}^2
  \eeqa
  where we used (\ref{identity}).
From this integrability we have
\beas&&\|f(t)-g(t)\|_{L^1_2}
=\|f_0-g_0\|_{L^1_2}+\int_{0}^{t}{\rm d}\tau\int_{{\bR}}\la v\ra ^2\big(Q_{*}(f)(\tau,v)-Q_{*}(g)(\tau,v)\big)\psi(\tau,v){\rm d}v,\\ \\
&&\int_{{\bR}}\la v\ra ^2\big(Q_{*}(f)(\tau,v)-Q_{*}(g)(\tau,v)\big)\psi(\tau,v){\rm d}v
\\
&&\le
\inttt_{{\bRRS}}
b|v-v_*|\big((1+2K)(f+g)|f_*-g_*|+ (ff_*+g g_*)
(|f'-g'|+|f_*'-g_*'|)
\big)\la v\ra^2{\rm d}\sg{\rm  d}v_*{\rm d}v
\\
&&\le b (1+2K)|{\bS}|
\intt_{{\bRR}}
\la v\ra ^2|v-v_*|(f+g) |f_*-g_*|{\rm  d}v_*{\rm d}v
\\
&&\quad +2b \int_{{\bR}} |f-g|\intt_{{\bRS}}\la v'\ra^2|v-v_*|f'f_*'{\rm d}\sg{\rm  d}v_* {\rm d}v
+2b \int_{{\bR}}|f-g|
\intt_{{\bRS}}\la v'\ra ^2|v-v_*|g'g_*'{\rm d}\sg{\rm  d}v_*{\rm d}v.
\eeas
Further estimates: From $|v-v_*|\le \la v\ra \la v_*\ra$ we have
$$
\intt_{{\bRR}}
\la v\ra^2|v-v_*|(f+g)|f_*-g_*|{\rm  d}v_*{\rm d}v\le
\|f+g\|_{L^1_3}\|f-g\|_{L^1_1}.$$
Combining this with Proposition \ref{proposition2} ($p=2, q=0, \gamma=1$) gives
\beas&&
\int_{{\bR}}\la v\ra^2\big(Q_{*}(f)(\tau,v)-Q_{*}(g)(\tau,v)\big)\psi(\tau,v){\rm d}v
\\
&&\le b(1+2K)|{\bS}|\|f+g\|_{L^1_3}\|f-g\|_{L^1_1}\\
&&\quad +
8b  |{\bS}|\big(\|f\|_{L^{\infty}}\|f\|_{L^{1}_{3}}
+\|f\|_{L^{\infty}_2}\|g\|_{L^{1}_{1}}+
\|g\|_{L^{\infty}}\|g\|_{L^{1}_{3}}
+\|g\|_{L^{\infty}_2}\|g\|_{L^{1}_{1}}
\big)\|f-g\|_{L^1_1}
\\
&&\le \big(b(1+2K)|{\bS}|\wt{C}_{T,K}+8b|{\bS}|\wt{C}^2_{T,K}\big)
\|f(\tau)-g(\tau)\|_{L^1_1},\quad \tau\in [0, T]\eeas
where we used (\ref{fn T Linfty}).
Thus we obtain
$$
\|f(t)-g(t)\|_{L^1_2}
\le \|f_0-g_0\|_{L^1_2}
+\big(b(1+2K)|{\bS}|\wt{C}_{T,K}+8b|{\bS}|\wt{C}^2_{T,K}\big)
\int_{0}^{t}\|f(\tau)-g(\tau)\|_{L^1_2}{\rm d}\tau
$$
for all $t\in [0, T]$. We then conclude from Gronwall's lemma that
\beas
\|f(t)-g(t)\|_{L^1_2}\le C_{T,K}\|f_0-g_0\|_{L^1_2} \quad \forall\, t\in [0, T]
\eeas
where $C_{T,K}=C_{T,K}
 (\|f_0\|_{L^{\infty}_3},\|f_0\|_{L^1_3},
 \|g_0\|_{L^{\infty}_3},\|g_0\|_{L^1_3})=
 \exp\big[\big(b(1+2K)|{\bS}|\wt{C}_{K,T}+8b|{\bS}|\wt{C}^2_{K,T}\big)T\big]$. From (\ref{fn T Linfty}) we see that the function $(y_1,y_2,y_3,y_4)\mapsto C_{T,K}(y_1,y_2,y_3,y_4)$ is independent of $n$, continuous
on $[0,\infty)^4$, and
 monotone non-decreasing with respect to each variable $y_1,y_2,y_3,y_4\in [0,\infty).$
 This proves (\ref{f-g}).

{\bf Step 2.} We now prove  (\ref{Theorem proposition}). Fix any $0<T<\infty$ and let  $\overline{K}=\max\Big\{\sup\limits_{0\le t\le T}\|f(t)\|_{L^{\infty}},\sup\limits_{0\le t\le T }\|g(t)\|_{L^{\infty}}\Big\}$. Following a similar argument in {\bf Step 1} we see that $f$ satisfies  (\ref{iteration1}), (\ref{fn infinity}) with $K=\overline{K}$.
  Then it follows from the proof of (\ref{fn T Linfty}) that
  \beas
&&\sup_{0\le t\le T}\big(\|f(t)\|_{L^{\infty}_2}+\|f(t)\|_{L^1_3}+
 \|g(t)\|_{L^{\infty}_2}+\|g(t)\|_{L^1_3}\big) \no  \\
&&\le \overline{C}_{T,\overline{K}}\big(\|f_0\|_{L^{\infty}_3},\sup\limits_{0\le t\le T}\|f(t)\|_{L^1_3},\sup\limits_{0\le t\le T}\|f(t)\|_{L^{\infty}},
 \|g_0\|_{L^{\infty}_3},\sup\limits_{0\le t\le T}\|g(t)\|_{L^1_3},\sup\limits_{0\le t\le T}\|g(t)\|_{L^{\infty}}\big)=: \overline{C}_{T,\overline{K}}\qquad
  \eeas
  where $\overline{C}_{T,\overline{K}}$ depends only on $T$ and $\|f_0\|_{L^{\infty}_3},\sup\limits_{0\le t\le T}\|f(t)\|_{L^1_3},\sup\limits_{0\le t\le T}\|f(t)\|_{L^{\infty}},
 \|g_0\|_{L^{\infty}_3},\sup\limits_{0\le t\le T}\|g(t)\|_{L^1_3},$ $\sup\limits_{0\le t\le T}\|g(t)\|_{L^{\infty}}.$
  Thus, (\ref{Theorem proposition}) can be proved with the same argument in the rest part of the proof in {\bf Step 1} by replacing $K$ , $\wt{C}_{T,K}$ with $\overline{K}$, $\overline{C}_{T,\overline{K}}$ respectively.
 \eepf
\vskip3mm

As an immediate application of this proposition we obtain the following continuity estimates for
 mild solutions of Eq.$(\ref{approximate equation 1})$ and Eq.$(\ref{approximate equation 2})$.

\begin{proposition}\label{equicontinuous} Let $0\le f_0\in L^1_3(\R^3)\cap L_3^{\infty}(\R^3)$ and
let $Q_{n,K}(\cdot)$ be defined in $(\ref{Qnk})$ with $B_n=B\wedge n$ and $B$ satisfying $(\ref{condition of kernel})$.
 With the same initial datum $f_0$, let $f^n$ and $f$
be  mild solutions of Eq.$(\ref{approximate equation 1})$ and Eq.$(\ref{approximate equation 2})$ respectively.
Then
\beqa\label{stability}
&&\sup_{n\ge 1, t\in [0, T]}\|f^n(t, \cdot +h)-f^n(t)\|_{L^1_2}\le C_{T,K}\|f_0(\cdot +h)-f_0\|_{L^1_2}\quad \forall\, |h|\le 1,\,\,\forall\, T\in (0,\infty)\quad \\
&&\label{uniform continuity in t}
\sup_{n\ge 1}\|f^n(t_1)-f^n(t_2)\|_{L^1}
\le 2b (1+2K)|{\bS}\|f_0\|_{L^1}\|f_0\|_{L^1_2}
|t_1-t_2| \quad \forall\, t_1,t_2\in [0,\infty)\quad \\
&&\sup_{t\in [0, T]}\|f(t, \cdot +h)-f(t)\|_{L^1_2}\le C_{T,K}\|f_0(\cdot +h)-f_0\|_{L^1_2}\quad \forall\, |h|\le 1,\,\,\forall\, T\in (0,\infty)\quad \\
&&
\|f(t_1)-f(t_2)\|_{L^1}
\le 2b (1+2K)|{\bS}\|f_0\|_{L^1}\|f_0\|_{L^1_2}
|t_1-t_2| \quad \forall\, t_1,t_2\in [0,\infty)\quad \eeqa
where $0<C_{T,K}<\infty$  depends only on  $T,\,K,\,\|f_0\|_{L^{\infty}_3},\, \|f_0\|_{L^1_3}.$
Consequently the sequence $\{f^n(t,\cdot)\}_{n=1}^{\infty}$ is both equicontinuous in $L^1_2({\mR}^3)$
uniformly in local time and equicontinuous in $C([0,\infty); L^1({\mR}^3))$.
\end{proposition}

\bepf We need only to prove the estimates for $f^n$ since the proof for $f$ is
completely the same.
From the structure of the collision $B$ (see (\ref{about v-v*})), it is easily seen that the velocity translation $g^n(t, v):=f^n(t, v+h)$ is still a mild solution to Eq.(\ref{approximate equation 1}) with
the initial datum $g_0(v):=f_0(v+h)$. Since
for any $|h|\le 1$ and any $s\in\{2,3\}$
\beas&& \|f_0(\cdot+h)\|_{L^1_s}\le 3^{s/2}\|f_0\|_{L^1_s},\quad
\|f_0(\cdot+h)\|_{L^{\infty}_s}\le 3^{s/2}\|f_0\|_{L^{\infty}_s}
\eeas
it follows from (\ref{f-g}) that (\ref{stability}) holds true.
Next for any $h\in \R^3$ with $|h|\le 1$ we have
$|\la v\ra^2-\la v-h\ra^2|f_0(v)
\le 3|h|\la v\ra
f_0(v)$ so that
\beqa\label{go to zero}\|f_0(\cdot+h)-f_0\|_{L_2^1}
\le 3|h|\|f_0\|_{L^1_1}+\|\wt{f_0}(\cdot+h)-\wt{f_0}\|_{L^1}
\to 0\quad {\rm as}\quad h\to 0\eeqa
where $\wt{f_0}(v)=\la v \ra^2f_0(v)$. This together with (\ref{stability})
proves the uniform local in time $L^1_2({\mR}^3)$-equicontinuity of the sequence
$\{f^n(t,\cdot)\}_{n=1}^{\infty}$.

Finally for any $t_1,t_2\ge 0$, using Cauchy-Schwarz inequality and the conservation of mass and energy of $f^n$  we have
\beas&&
\int_{{\bR}}|f^n(t_1,v)-f^n(t_2,v)|{\rm d}v
\\
&&\le \int_{t_1\wedge t_2}^{t_1\vee t_2}
{\rm d}\tau\inttt_{{\bRRS}}
B_n(v-v_*,\sg) ({f^n}'\wedge n)({f^n}_*'\wedge n)
(1+ {f^n}\wedge K+ {f_{*}^n}\wedge K){\rm d}\sg{\rm d}v_*{\rm d}v
\\
&&\quad +\int_{t_1\wedge t_2}^{t_1\vee t_2}
{\rm d}\tau\inttt_{{\bRRS}}
B_n(v-v_*,\sg) ({f^n}\wedge n)({f_{*}^n}\wedge n)
(1+ {f^n}'\wedge K+ {f_{*}^n}'\wedge K){\rm d}\sg{\rm d}v_*{\rm d}v
\\
&&\le 2b (1+2K)|{\bS}|\int_{t_1\wedge t_2}^{t_1\vee t_2}
\Big(\int_{{\bR}}\la v\ra f^n(\tau,v){\rm d}v\Big)^2{\rm d}\tau
\le 2b (1+2K)|{\bS}\|f_0\|_{L^1}\|f_0\|_{L^1_2}
|t_1-t_2|.\eeas
This proves (\ref{uniform continuity in t}).
\eepf
\vskip3mm

\section{Mild solutions of  the intermediate equation}

In this section we first prove the existence of conservative mild solutions of
Eq.(\ref{approximate equation 2}). Then we will use  multi-step iterations of
the collision gain operator $Q^{+}(\cdot, \cdot)$ to prove further estimates for these mild solutions which are used in proving Theorem {\ref{main results}.
Throughout this section, the same constant $K>0$ in Eq.$(\ref{approximate equation 1})$ and Eq.(\ref{approximate equation 2}) is fixed.

\subsection{Existence of mild solutions of  the intermediate equation}

\begin{proposition}\label{mild solution f}
Let $Q_{K}(\cdot)$ be defined in $(\ref{QK})$ with  $B$ satisfying $(\ref{condition of kernel})$, and let $0\le f_0\in L^1_3(\R^3)\cap L_3^{\infty}(\R^3)$.
  Then Eq.$(\ref{approximate equation 2})$ has a conservative mild solution $f\in L^{\infty}([0,\infty); L^1_2(\R^3))$ with the initial datum $f_0$, i.e.
\begin{equation}\label{mild sl}
 f(t,v)=f_0(v)+\int_{0}^{t}Q_{K}(f)(\tau,v){\rm d}\tau,\quad (t,v)\in [0,\infty)\times(\R^3\setminus Z).
\end{equation}
Here $Z$ is a null set independent of $t$.
\end{proposition}

\bepf Our proof consists of two steps. In the first step we shall use mild solutions of Eq.(\ref{approximate equation 1}).

{\bf Step 1.}  Let $\{f^n\}_{n= 1}^{\infty}$ be a sequence of mild solutions of
Eq.(\ref{approximate equation 1}) obtained in Proposition \ref{fn} with the same initial data $f_0^n=f_0, n=1,2,3,...\,.$  By conservation of mass and energy of $\{f^n\}_{n= 1}^{\infty}$ we have $\sup\limits_{n\ge 1, t\ge 0}\|f^n(t)\|_{L^1_2}=\|f_0\|_{L^1_2}<\infty$.
Next for any fixed $t\ge 0$ we deduce from (\ref{stability}) that $\{f^n(t,\cdot)\}_{n= 1}^{\infty}$ is equicontinuous in $L^1(\R^3)$. These imply that $\{f^n(t,\cdot)\}_{n= 1}^{\infty}$ is a relatively compact set in $L^1(\bR)$
(for any fixed $t>0$).
 Using diagonal argument we can find a common subsequence $\{n_k\}_{k= 1}^{\infty}\subset {\mN}$
such that for any $r\in \mQ\cap [0,\infty)$, $\{f^{n_k}(r,\cdot)\}_{k= 1}^{\infty}$ is a Cauchy sequence in $L^1(\R^3)$.
Then for any fixed $t\ge 0$, let us consider $\{f^{n_k}(t,\cdot)\}_{k=1}^{\infty}.$ It follows from (\ref{uniform continuity in t}) that for any  $\vep >0,$ there exists an $r\in \mQ\cap [0,\infty)$, such that
$\|f^{n_k}(t)-f^{n_k}(r)\|_{L^1}\le \fr{\vep}{3}$ for all  $k\in{\mN},$  and thus
\beas
&&\|f^{n_k}(t)-f^{n_j}(t)\|_{L^1}
\le
\|f^{n_k}(t)-f^{n_k}(r)\|_{L^1}+\|f^{n_j}(t)-f^{n_j}(r)\|_{L^1}
+\|f^{n_k}(\tau)-f^{n_j}(r)\|_{L^1} \\
&&\le \fr{\vep}{3}+\fr{\vep}{3}+\|f^{n_k}(r)-f^{n_j}(r)\|_{L^1}\quad \forall\, k,j\in \N.
\eeas
Since $\{f^{n_k}(r,\cdot)\}_{k= 1}^{\infty}$ is a Cauchy sequence in $L^1(\bR)$, it follows from the arbitrariness of $\vep$ that  $\{f^{n_k}(t,\cdot)\}_{k=1}^{\infty}$ is a Cauchy sequence in $L^1(\bR)$ (for any fixed $t\ge 0$), i.e.
\beqa\label{Cauchy 1}\lim_{k\ge j\to\infty}\|f^{n_k}(t)-f^{n_j}(t)\|_{L^1}=0\qquad \forall\, t\in [0,\infty).\eeqa
Since all $f^n$ are conservative mild solutions of
Eq.(\ref{approximate equation 1}) with the same initial data $f_0,$ it follows that $\{f^{n_k}\}_{k=1}^{\infty}$ is a nonnegative bounded sequence in $L^{\infty}([0,\infty); L^1(\R^3))$. Thus, it follows from (\ref{Cauchy 1}) and Lemma \ref{lemma of L(infty)([0,T];L1)} that there exists a function $0\le \wt{f}\in L^{\infty}([0,\infty); L^1(\R^3))$ such that for any $t\in [0,\infty)$
 \beqa\label{L1 convergence}
\lim_{k\to\infty}\|f^{n_k}(t)-\wt{f}(t)\|_{L^1}=0.
\eeqa

{\bf Step 2.} We prove that the above function $\wt{f}$, after a modification on a null set,
is a conservative mild solution of Eq.(\ref{approximate equation 2}).
Let $0<R<\infty$ and  let $0< T<\infty$. Define $\vp(v)= \la v\ra^3\wedge R$ for any $v\in \R^3$.
   Since $\vp\in L^{\infty}(\R^3)$ and $f_0\in L_3^1(\R^3)$, it follows from Proposition \ref{moment propogation} and (\ref{L1 convergence}) that for any $t\in [0,T]$
  \beqa\label{small 1}
  &&\int_{\R^3} (\la v\ra^{3}\wedge R)\wt{f}(t,v) \dd v=\lim_{k\to \infty}\int_{\R^3}(\la v\ra ^3\wedge R)f^{n_k}(t,v)
\dd v\no\\
  &&\le \lim_{k\to \infty}\int_{\R^3}\la v\ra ^3 f^{n_k}(t,v)\dd v
    \le \|f_0\|_{L^1_3}
+(1+2K)2^{\fr{7}{2}}b |\bS|\|f_0\|_{L^1_2}^{2}T:=C_{f_0,T}<\infty \eeqa
 and so, by Fatou's lemma,
 \beqa\label{wt(f) L31}
 \|\wt{f}(t,v)\|_{L_3^1}
 \le C_{f_0,T}.
 \eeqa
 Next from (\ref{integrability}) we have for any $t\in [0,T]$
\beqa\label{Q+ infty}
\int_{\R^3} \la v\ra^2 \big(Q^{+}_K(\wt{f})(t,v)+Q^{-}_K(\wt{f})(t,v)\big)\dd v
  \le 2(1+2K)b|\bS| \|\wt{f}\|_{L^1_3}^2
  \le 2(1+2K)b|\bS| C_{f_0,T}^2 <\infty.
\eeqa
 This together with the fact that $0\le Q_{n,K}^{\pm}(\wt{f})\le Q_{K}^{\pm}(\wt{f})$ and Fatou's lemma implies that
\beqa \label{f=f0+int...}
&&\int_0^{T}\dd t\int_{\bR}\Big|\wt{f}(t,v)-f_0(v)
-\int_{0}^{t}Q_K(\wt{f})(\tau,v)\dd\tau \Big|\dd v
\no\\
&&\le \int_0^{T}\dd t\liminf_{k\to \infty}\int_0^t\dd \tau\Big(\int_{\bR}\big|Q_{n_k,K}(\wt{f})(\tau,v)-Q_K(\wt{f})(\tau,v)\big|\dd v
\no\\
&&\quad +\int_{\bR}\big|Q_{n_k,K}(f^{n_k})(\tau,v)-Q_{n_k,K}(\wt{f})(\tau,v)\big|\dd v \Big).
\eeqa
To estimate the first term in the right hand side of the above inequality we compute
\beqa\label{intQnk}
&&\int_{\bR}\big|Q_{n_k,K}(\wt{f})(\tau,v)-Q_K(\wt{f})(\tau,v)\big|\dd v
\no\\
&&\le
2(1+2K)\inttt_{{\bRRS}}\Big(\big|B_{n_k}(v-v_*,\sg)
(\wt{f}\wedge n_k)(\wt{f}_*\wedge n_k)-B_{n_k}(v-v_*,\sg)
\wt{f}{\wt{f}}_*\big|
\no\\
&&\quad +\big|B(v-v_*,\sg)(\wt{f}\wedge n_k)({\wt{f}}_*\wedge n_k)
-B(v-v_*,\sg){\wt{f}}{\wt{f}}_*\big|
\Big) {\rm d}\sg{\rm d}v_*\dd v
\no\\
&&\le 8b(1+2K)\inttt_{{\bRRS}}\la v\ra\la v_*\ra
\big|
(\wt{f}\wedge n_k)({\wt{f}}_*\wedge n_k)-
\wt{f}{\wt{f}}_*\big| {\rm d}\sg{\rm d}v_*\dd v
\no\\
&&\le
16b(1+2K)C_{f_0,T}^2\quad ({\rm uniformly \,\, in\,\,} \tau\in[0,T] ).
 \eeqa
Since
 $
 \la v\ra\la v_*\ra
\big|(\wt{f}\wedge n_k)({\wt{f}}_*\wedge n_k)-
{\wt{f}}{\wt{f}}_*\big|
\le
2 \la v\ra\la v_*\ra {\wt{f}}{\wt{f}}_*,
 $
 it follows from Lebesgue's dominated convergence that
 \beas
 \inttt_{{\bRRS}}\la v\ra\la v_*\ra
\big|(\wt{f}\wedge n_k)({\wt{f}}_*\wedge n_k)-
{\wt{f}}{\wt{f}}_*\big|
{\rm d}\sg{\rm d}v_*\dd v\to 0 \quad {\rm as}\,\, k\to \infty.
 \eeas
 This together with (\ref{intQnk}) implies that
 \beqa
 \int_{\bR}\big|Q_{n_k,K}(\wt{f})(\tau,v)-Q_K(\wt{f})(\tau,v)\big|\dd v \to 0 \quad {\rm as}\,\, k\to \infty.
 \eeqa
 Then from (\ref{intQnk}) and Lebesgue's dominated convergence we obtain
 \beqa
 \lim_{k\to \infty}\int_0^t\dd \tau\int_{\bR}\big|Q_{n_k,K}(\wt{f})(\tau,v)-Q_K(\wt{f})(\tau,v)\big|\dd v=0\qquad \forall\, t\in[0,\infty).
 \eeqa
 Following a similar method we have
 \begin{equation}\label{term2 int}
 \lim_{k\to \infty}\int_0^t\dd \tau\int_{\bR} \big|Q_{n_k,K}(f^{n_k})(\tau,v)-Q_{n_k,K}(\wt{f})(\tau,v)\big|\dd v=0\qquad \forall\, t\in[0,\infty).
 \end{equation}
Combining these with (\ref{f=f0+int...})  we conclude
\beas
\int_{\bR}\dd v\int_0^{T}\Big|\wt{f}(t,v)-f_0(v)-\int_{0}^{t}Q_K(\wt{f})(\tau,v)\dd \tau\Big|\dd t=0.
\eeas
Since $0<T<\infty$ is arbitrary, it follows from Fatou's lemma that
\beqa\label{zero}
\int_{\bR}\dd v\int_0^{\infty}\Big|\wt{f}(t,v)-f_0(v)
-\int_{0}^{t}Q_K(\wt{f})(\tau,v)\dd \tau\Big|\dd t=0.
\eeqa
Let $f(t,v):=|f_0(v)+\int_{0}^{t}Q_K(\wt{f})(\tau,v){\rm d}\tau|,\,\, (t,v)\in[0,\infty)\times{\R}^3$.
We see that $f$ is nonnegative and there is a null set $Z_1\subset {\R}^3$ (independent of $t$) such that
for any fixed $v\in {\mR}^3\setminus Z_1$ the
function $t\mapsto f(t,v)$ is continuous on $[0,\infty)$. This is because the functions
$(\tau,v)\mapsto Q_{K}^{\pm}(\wt{f})(\tau,v)$ belong to $L^1([0, T]\times {\mR}^3)$ for all
$0<T<\infty$, and so by Fubini's theorem there is a null  set $Z_1\subset {\R}^3$ such that
$\int_{0}^{T}|Q_K(\wt{f})(\tau,v)|{\rm d}\tau<\infty$ for all $0<T<\infty$ and all $v\in {\mR}^3\setminus Z_1$.
From (\ref{zero}) and the nonnegativity of $\wt{f}$ we see that
\beqa\label{f=wt(f)}
f(t,v)=\wt{f}(t,v) \quad{\rm for\,\,\, a.e.}\,\, (t,v)\in [0,\infty)\times \R^3.
\eeqa
 Thus
\beas
\int_0^{\infty}\dd t\int_{\R^3}\big|f(t,v)-\wt{f}(t,v)\big|\dd v=0.
\eeas
 This implies, using Fubini's theorem, that there exists a null $\wt{Z}\subset [0,\infty)$ such that
\beas
\int_{\R^3}\big|f(t,v)-\wt{f}(t,v)\big|\dd v=0\qquad \forall\, t\in [0,\infty)\setminus \wt{Z}.
\eeas
This together with  (\ref{wt(f) L31}) implies that for any $0<T<\infty$ and any $t\in [0, T]\setminus \wt{Z}$,
$\|f(t)\|_{L_3^1}=\|\wt{f}(t)\|_{L_3^1}\le C_{f_0,T}<\infty.$
Since $t\mapsto f(t,v)$ is continuous on $[0,\infty)$ for any fixed $v\in \R^3\setminus Z$, it follows from Fatou's lemma and (\ref{small 1}) that
\beqa\label{mono}
\|f(t)\|_{L_3^1}
\le
C_{f_0,T}\quad \forall\, t\in [0, T],\,\,\forall\, 0<T<\infty.
\eeqa
Thus, from (\ref{integrability}) we have
\beqa\label{Qk(f) L21}
\sup\limits_{0\le t\le T} \|Q_K^{\pm}(f)(t,\cdot)\|_{L_2^1}<\infty
\quad {\rm for\quad any}\,\, 0<T<\infty.
\eeqa
This together with (\ref{f=wt(f)}), the boundedness of the mapping $g\mapsto |g|\wedge K$,
the formula (\ref{identity}) of change variables,
and the arbitrariness of $T$  imply that
\beas
Q_K(\wt{f})(t,v)=Q_K(f)(t,v)\quad {\rm a.e.}\qquad (t,v)\in[0,\infty)\times{\mR}^3.
\eeas
Combining this with (\ref{zero}) and (\ref{f=wt(f)}) leads to
$$
f(t,v)=f_0(v)+\int_{0}^{t}Q_K(f)(\tau,v)\dd \tau \qquad {\rm a.e.}\quad (t,v)\in[0,\infty)\times{\mR}^3.
$$
Since by definition of $f$,  $f(t,v)$ is fully measurable on $[0,\infty)\times{\mR}^3$ and
 the function $t\mapsto f(t,v)$ is continuous on $ [0,\infty)$ for almost every $v\in {\mR}^3$, while using Fubini's theorem it is easily seen that the function $ f_0(v)+\int_{0}^{t}Q_K(f)(\tau,v)\dd \tau$ is fully measurable on $[0,\infty)\times{\mR}^3$ and
 $t\mapsto f_0(v)+\int_{0}^{t}Q_K(f)(\tau,v)\dd \tau$
 is continuous on $t\in [0,\infty)$ for almost every
 $v\in {\mR}^3$,
it follows from Lemma \ref{Lemma H} that
there is a null set $Z\subset {\mR}^3$ such that
\beqa\label{f is a mild solution}
f(t,v)=f_0(v)+\int_{0}^{t}Q_{K}(f)(\tau,v){\rm d}\tau \quad \forall\,(t,v)\in [0,\infty)\times (\R^3\setminus Z).
\eeqa
 To prove that $f$ is a conservative mild solution of Eq.(\ref{approximate equation 2}), we now need only to prove the conservation law of $f$. Let $\vp\in C({\R}^3)$ satisfy $|\vp(v)|\le C\la v\ra^2$ for constant $0<C<\infty$.
From (\ref{Qk(f) L21}) and (\ref{f is a mild solution}) we have
  \beas\int_{{\R}^3}\vp(v) f(t,v){\rm d}v=\int_{{\R}^3} \vp(v)f_0(v){\rm d}v
+\int_0^t {\rm d}\tau\int_{{\R}^3}\la v\ra ^2 Q_{K}(f)(\tau,v) \dd v,\quad  t\in [0, \infty)
  \eeas
 and so for a.e. $t\in [0,\infty)$
  \beas
  &&\fr{\dd}{\dd t} \int_{{\R}^3} \vp(v)f(t,v){\rm d}v
  =\fr{1}{2}\inttt_{{\bRRS}}
B(v-v_*,\sg)
ff_*
(1+ f'\wedge K+ f_*'\wedge K)\\
&&\quad\times
\big(\vp(v')+\vp(v_*')-\vp(v)-\vp(v_*)\big){\rm d}\sg{\rm d}v_*{\rm d}v.
  \eeas
  Taking $\vp(v)=1$ implies that $f$ conserves the mass. While
taking $\vp(v)=|v-{\bf a}|^2$ with any constant vector ${\bf a}\in {\R}^3$ we also have
$\vp(v')+\vp(v_*')-\vp(v)-\vp(v_*)\equiv 0$ and thus
  \beas
  \int_{\R^3}|v-{\bf a}|^2f(t,v)\dd v=\int_{\R^3}|v-{\bf a}|^2f_0(v)\dd v \quad \forall\,t\in [0,\infty).
  \eeas
By choosing ${\bf a}=0,(1,0,0),(0,1,0),(0,0,1)$ respectively and using the conservation of mass of $f$ we conclude that $f$ conserves  energy and momentum:
$$\int_{\R^3}|v|^2 f(t,v) \dd v=\int_{\R^3}|v|^2 f_0(v) \dd v,\quad \int_{\R^3}v f(t,v) \dd v=\int_{\R^3}v f_0(v) \dd v\qquad \forall\, t\in [0,\infty).$$
\eepf
\vskip3mm

  \subsection{Some estimates for collision integral of mild solutions of the intermediate equation}

Recall that in order to use mild solutions $f$ of Eq.(\ref{approximate equation 2}) to prove the global in time existence of solutions of Eq.(\ref{Equation}),
one needs only to prove that $f\le K$ for a suitable $0<K<\infty$. To do this we will establish the next two Propositions and use Duhamel's formula to deduce certain pointwise estimates for
$f$ in Proposition \ref{Duhamel}.  Then we will give further propositions which are based on multi-step iterations of the collision gain operator $Q^{+}(\cdot,\cdot)$ and the estimates obtained in Proposition \ref{Duhamel}.
For notational convenience we will use in this subsection the notation $f_t(v)\equiv f(t,v)$.

\begin{proposition}\label{L3}
Let $Q_{K}(\cdot)$ be defined in $(\ref{QK})$ with  $B$ satisfying $(\ref{condition of kernel})$. For any $0\le f_0\in L^1_3(\R^3)\cap L_3^{\infty}(\R^3)$, let $f$
be a mild solution of Eq.$(\ref{approximate equation 2})$ obtained in Proposition \ref{mild solution f} corresponding to the kernel $B$ with the initial datum $f_0$.
  Then
  \beqa\label{L3 estimate}
  \|f(t)\|_{L_3^1}
  \le \max\{1,C_1\}\|f_0\|_{L_3^1}
 \qquad \forall\, t\ge 0
  \eeqa
  where
  \beqa\label{C1 def}
  &&
  C_1=\fr{\big(128(\sqrt{2}-1)(1+2K)b
  +2(\fr{22}{5}\sqrt{2}-\fr{31}{5})a\big)\|f_0\|_{L_2^1}^2}{a (\fr{22}{5}\sqrt{2}-\fr{31}{5}) \|f_0\|_{L^1}\|f_0\|_{L_3^1}}\,.\eeqa

  \end{proposition}

  \bepf We first prove that (\ref{L3 estimate}) and
  (\ref{C1 def}) hold true for the case $f_0\in \bigcap_{s\ge 3} L_s^1(\R^3)$.
  For any $s\ge 3$ and $R>0$, define $\vp(v)= \la v\ra^s\wedge R$. Since $\vp\in L^{\infty}(\R^3)$, it follows from Proposition \ref{moment propogation} and the same proof of (\ref{mono}) that
  \beas
  \int_{\R^3} (\la v\ra^s\wedge R)f(t,v) \dd v &\le & \int_{\R^3} (\la v\ra^s\wedge R)\wt{f}(t,v)\dd v
  =
  \lim_{k\to \infty}\int_{\R^3}(\la v\ra^s\wedge R)f^{n_k}(t,v)\dd v
  \\
  &\le & \Big(\|f_0\|_{L^1_s}^{\fr{1}{s-2}}
  +(1+2K)2^{\fr{s}{2}+2}|\bS|b(\|f_0\|_{L^1_2})^{\fr{s-1}{s-2}}
  \fr{t}{s-2}\Big)
  ^{s-2}\qquad \forall\, t\ge 0
  \eeas
  where $\wt{f}$ and $f^{n_k}$ are defined in Proposition \ref{mild solution f}.
  Letting $R\to \infty$ we obtain by Fatou's lemma that
  \beqa\label{L1s}
 \|f(t)\|_{L_s^1}
  \le \Big(\|f_0\|_{L^1_s}^{\fr{1}{s-2}}
  +(1+2K)2^{\fr{s}{2}+2}|\bS|b(\|f_0\|_{L^1_2})^{\fr{s-1}{s-2}}
  \fr{t}{s-2}\Big)
  ^{s-2}.
  \eeqa
    This implies that $\sup\limits_{0\le t\le T}\|f(t)\|_{L_s^1}<\infty$ for all $s\ge 3$ and all $0<T<\infty$.
  Since $|v-v_*|\le \la v\ra \la v_*\ra$ and $\la v'\ra^{s}\le \la v\ra^s\la v_*\ra^s $ it follows that for any $t\in [0,\infty)$
  \beas
  \|Q_K^{+}(f)(t,\cdot)\|_{L_s^1}
  \le (1+2K)b|\bS| \|f(t)\|_{L^1_{s+1}}^2,
  \quad
  \|Q_K^{-}(f)(t,\cdot)\|_{L_s^1}
  \le (1+2K)b |\bS|\|f(t)\|_{L^1_1}\|f(t)\|_{L^1_{s+1}}.
  \eeas
   Thus for any $s\ge 3$ and any $0<T<\infty$ we have
   $\sup\limits_{0\le t\le T}\|Q_K^{\pm}(f)(t)\|_{L_s^1}< \infty$. Then it follows from Fubini's theorem that
  \beqa\label{integrabliity derivitive}
  \|f(t)\|_{L_s^1}=\|f_0\|_{L_s^1}+\int_0^t \dd \tau\int_{\R^3}\la v\ra^s Q_K(f)(\tau,v)\dd v \quad \forall\, t\in [0,\infty).
  \eeqa
 From (\ref{integrabliity derivitive}) with $s=3$  and Corollary \ref{LemmaE}  we have for a.e. $t\in [0,\infty),$
  \beqa \label{I}
  &&\frac{\dd}{\dd t}\|f(t)\|_{L_3^1}\no\\
  &&
  =\frac{1}{2}\iiint_{\R^3\times\R^3\times\SP^2}
  \big(\la v'\ra ^3+\la v_*'\ra ^3-\la v\ra ^3-\la v_*\ra ^3\big)
  B(v-v_*,\si)ff_{*}(1+ f'\land K+ f'_{*}\land K)\dd \si\dd v_{*} \dd v
  \nonumber \\
  &&\le
  4(\sqrt{2}-1)(1+2K)\iiint_{\R^3\times\R^3\times\SP^2} B(v-v_*,\si)\la v\ra ^2\la v_*\ra ff_*
  \dd \si \dd v_* \dd v
  \nonumber \\
  &&\ \ \ \ -\frac{1}{32}\iiint_{\R^3\times\R^3\times\SP^2} B(v-v_*,\si) ff_*(1+ f'\land K+ f'_{*}\land K)[\kappa(\theta)]^{\fr{3}{2}}\la v\ra ^3\dd \si \dd v_* \dd v
  \nonumber \\
  &&:=4(\sqrt{2}-1)(1+2K)I_1-\frac{1}{32}I_2
  \eeqa
  where $\kappa(\theta)=\min\{(1-\cos^2(\theta/2)),\,(1-\sin^2(\theta/2))\}$.
  For $I_1$ and $I_2$ we compute
  \beqa \label{I1}
  I_1&\le&
  b\inttts
  \la v\ra ^3\la v_*\ra^2 ff_*\dd \si \dd v_* \dd v
  =b |\bS|\|f_0\|_{L_2^1}\|f(t)\|_{L_3^1},\\
  I_2&\ge&\label{I2}
  \inttts B(v-v_*,\si) ff_*(1+ f'\land K+ f'_{*}\land K)[\kappa(\theta)]^{\fr{3}{2}}\la v\ra ^3\dd \si \dd v_* \dd v
  \no\\
  &\ge&
  a\int_{\bS}[\kappa(\theta)]^{\fr{3}{2}}\dd \si\intt_{\bRR}
  \fr{|v-v_*|^{\beta+1}}{1+|v-v_*|^{\beta}} ff_*\la v\ra ^3 \dd v_* \dd v
  \no\\
  &\ge&
  a C_0\intt_{\bRR}
  \Big(|v-v_*|-\fr{|v-v_*|}{1+|v-v_*|^{\beta}}\Big)
  ff_*\la v\ra ^3 \dd v_* \dd v
  \no\\
  &\ge&
  a C_0
  (\|f_0\|_{L^1}\|f(t)\|_{L_4^1}-\|f(t)\|_{L_1^1}\|f(t)\|_{L_3^1}
  -\|f_0\|_{L^1}\|f(t)\|_{L_3^1})
  \no\\
  &\ge&
  a C_0
  \Big(\|f_0\|_{L^1}\fr{\|f(t)\|_{L_3^1}^2}{\|f_0\|_{L_2^1}}-2\|f_0\|_{L_2^1}
 \|f(t)\|_{L_3^1}\Big)
  \eeqa
  where we used the conservation of mass and energy of $f$ and $C_0=\int_{\bS}[\kappa(\theta)]^{\fr{3}{2}}\dd \si=\fr{22\sqrt{2}-31}{5}\pi.$
 Combining (\ref{I}), (\ref{I1}) and (\ref{I2}) we obtain for a.e. $t\in [0,\infty),$
  \beas
  \frac{\dd}{\dd t}\|f(t)\|_{L_3^1}\le
  -c_1\|f(t)\|_{L_3^1}^2
  +c_2\|f(t)\|_{L_3^1}
  \eeas
   where
\beas&&c_1=\fr{ (\fr{22}{5}\sqrt{2}-\fr{31}{5})\pi a\|f_0\|_{L^1}}{32\|f_0\|_{L_2^1}},\quad
  c_2=\big(4(\sqrt{2}-1)(1+2K)b|\bS|
  +\frac{1}{16} \big(\fr{22}{5}\sqrt{2}-\fr{31}{5}\big)\pi a\big)\|f_0\|_{L_2^1}.\eeas
  Here we note that by conservation of mass we have $\|f(t)\|_{L^1_s}\ge \|f_0\|_{L^1}>0$
and from (\ref{L1s}) we know that $\fr{1}{\|f(t)\|_{L^1_s}}$ is absolutely continuous on
any bounded interval.  It follows that $t\mapsto \fr{\exp(c_2 t)}{\|f(t)\|_{L_s^1}}$ is also absolutely continuous on
any bounded interval. Thus
\beas
\fr{\dd}{\dd t}\Big(\fr{\exp(c_2 t)}{\|f(t)\|_{L_s^1}}\Big)\ge c_1\exp(c_2 t) \qquad  {\rm a.e.}\quad t\in [0, \infty).
\eeas
  Solving this differential inequality and noticing that $\|f(0)\|_{L_3^1}=\|f_0\|_{L_3^1}$
  we conclude
\beas&&\|f(t)\|_{L_3^1}
  \le\fr{c_2\exp(c_2t)}{c_2+c_1\|f_0\|_{L_3^1}(\exp(c_2 t)-1)}\|f_0\|_{L_3^1}
  \le \max\{1,C_1\}\|f_0\|_{L_3^1}\qquad\forall\, t\ge 0
  \eeas
  where
 $C_1=c_2/(c_1\|f_0\|_{L_3^1})$ is defined in (\ref{C1 def}).

Now for general case let $f_0$ be given in Theorem \ref{main results}. Consider approximation  $f^n_0(v)=f_0(v)e^{-\fr{|v|^2}{n}}(n\in \N)$ and let $f^n$ be the conservative mild solutions of Eq.(\ref{approximate equation 2}) obtained in Proposition {\ref{mild solution f}} with
the initial data $f_0^n$. Since $f^n_0\in \bigcap_{s\ge 3} L_s^1(\R^3)$, we have proved in above that $\|f^n(t)\|_{L^1_3}, \|f^n_0\|_{L^1_3}$ satisfy (\ref{L3 estimate}), (\ref{C1 def}). Note that in the inequality (\ref{L3 estimate}) applied for $f^n, f^n_0$ the coefficient $C_1$ depends only and continuously on $(\|f^n_0\|_{L^1},\|f^n_0\|_{L_2^1},K)$.
It follows from (\ref{mono}) and Proposition \ref{stable} that
\beas
\|f^n(t)-f(t)\|_{L_2^1}\le C_{T,K,n}\|f^n_0-f_0\|_{L^1_2}\le C_{T,K}\|f^n_0-f_0\|_{L^1_2}\to 0\quad (n\to\infty)\quad \forall\, t\in [0, T]
\eeas
where $C_{T, K, n}=C_{T,K}
 (\|f_0\|_{L^{\infty}_3},\|f_0\|_{L^1_3},
 \|f_0^n\|_{L^{\infty}_3},\|f_0^n\|_{L^1_3})\le C_{T, K}=C_{T,K}
 (\|f_0\|_{L^{\infty}_3},\|f_0\|_{L^1_3},
 \|f_0\|_{L^{\infty}_3},\|f_0\|_{L^1_3})$ which is because $0\le f^n_0\le f_0$ and
$C_{T,K}(y_1,y_2,y_3,y_4)$ is monotone non-decreasing with respect to each
$y_1,y_2,y_3,y_4\in[0,\infty)$.

Since $\|f(t)\|_{L^1_3}\le \liminf\limits_{n\to\infty}\|f^n(t)\|_{L^1_3}$ (by Fatou's lemma)
and $\|f^n_0\|_{L^1_k}\to \|f_0\|_{L^1_k}\,(n\to\infty)$ for $0\le k\le 3,$
 it follows by taking $n\to\infty$ to (\ref{L3 estimate}), (\ref{C1 def}) for $f^n$
 that $f$ also satisfies (\ref{L3 estimate}), (\ref{C1 def}).
 \eepf
  \vskip3mm

  \begin{proposition}\label{L estimate}
Let $L_{K}(\cdot)$, $Q_{K}(\cdot)$ be defined in $(\ref{Lft})$, $(\ref{QK})$ respectively with  $B$ satisfying $(\ref{condition of kernel})$. For any $0\le f_0\in L^1_3(\R^3)\cap L_3^{\infty}(\R^3)$ satisfying $\int_{{\R}^3} v f_0(v){\rm d}v=0$, let $f_t(v)\equiv f(t,v)$ with the initial datum $f_0$ be a mild solution of Eq.$(\ref{approximate equation 2})$ obtained in Proposition \ref{mild solution f} corresponding to the kernel $B$. Then we have
  \beqa\label{L(ft)(v)}
   L_{K}(f_t)(v)\ge a
   \fr{(\min\{M_0, M_2\})^{\fr{\beta+1}{2}}}
   {2^{\beta}\|f_t\|_{L_{3}^1}^{\fr{\beta-1}{2}}}\langle v\rangle
    \qquad \forall\, (t,v)\in [0,\infty)\times\R^3
  \eeqa
  where
   $
   M_0=\int_{\R^3}f_0(v) \dd v,\,
   M_2=\int_{\R^3}|v|^2f_0(v) \dd v
   $.
  \end{proposition}

  \bepf Using H\"{o}lder's inequality we have for any $v\in \R^3$
   \beas
   &&|v|^2M_0+M_2=\int_{\R^3}|v-v_*|^2f_t(v_*)\dd v_*
   \\
   &&\le
   \Big(\int_{\R^3}\fr{|v-v_*|}{1+|v-v_*|^{-\beta}}f_t(v_*)\dd v_*\Big)^{\fr{2}{\beta+1}}
   \Big(\int_{\R^3}|v-v_*|^{\fr{2\beta}{\beta-1}}(1+|v-v_*|^{-\beta})^{\fr{2}{\beta-1}}f_t(v_*)\dd v_*\Big)^{\fr{\beta-1}{\beta+1}}
   \\
   &&\le
   \Big(\int_{\R^3}\fr{|v-v_*|^{\beta+1}}{1+|v-v_*|^{\beta}}f_t(v_*)\dd v_*\Big)^{\fr{2}{\beta+1}}
   \Big(\int_{\R^3}(1+|v-v_*|^{\fr{2\beta}{\beta-1}})f_t(v_*)\dd v_*\Big)^{\fr{\beta-1}{\beta+1}}.
   \eeas
   From this inequality and the definition of $L_{K}(f_t)(v)$ we deduce that
   for any $(t,v)\in [0,\infty)\times \R^3$
   \beas
   L_{K}(f_t)(v) & = &\iint_{\R^3\times\SP^2}B(v-v_*,\si)f_t(v_*)(1+ f_t(v')\land K+ f_t(v_*')\land K)\dd \si\dd v_{*}
   \\
   &\ge & a
   \int_{\R^3}\fr{|v-v_*|^{\beta+1}}{1+|v-v_*|^{\beta}}f_t(v_*)\dd v_*
   \ge a\fr{(\int_{\R^3}|v-v_*|^2f_t(v_*)\dd v_*)^{\fr{\beta+1}{2}}}
   {\big(\int_{\R^3}(1+|v-v_*|^{\fr{2\beta}{\beta-1}})f_t(v_*)\dd v_*\big)^{\fr{\beta-1}{2}}}
   \\
   &\ge& a\fr{(\min\{M_0, M_2\}(|v|^2+1))^{\fr{\beta+1}{2}}}
   {2^{\beta}\|f_t\|_{L_{\fr{2\beta}{\beta-1}}^1}^{\fr{\beta-1}{2}}\langle v\rangle^{\beta}}
   \ge a
   \fr{(\min\{M_0, M_2\})^{\fr{\beta+1}{2}}}
   {2^{\beta}\|f_t\|_{L_{3}^1}^{\fr{\beta-1}{2}}}\langle v\rangle
\eeas
where we used the assumption $\beta \ge 3$ and the conservation of mass, momentum ($\int_{{\R}^3} v f_0(v){\rm d}v=0$) and energy of $f$. \eepf
\vskip3mm

  \begin{proposition}\label{Duhamel}
  Let $Q^{+}(\cdot,\cdot)$, $Q_{K}(\cdot)$ be defined in $(\ref{Q+ bilinear})$, $(\ref{QK})$ respectively with  $B$ satisfying $(\ref{condition of kernel})$.
  For any $0\le f_0\in L^1_3(\R^3)\cap L_3^{\infty}(\R^3)$ satisfying $\int_{{\R}^3} v f_0(v){\rm d}v=0$, let $f_t(v)\equiv f(t,v)$ be a mild solution of Eq.$(\ref{approximate equation 2})$ obtained in Proposition \ref{mild solution f} corresponding to the kernel $B$ with the initial datum $f_0$. Let
  \beas E^{t}_{\tau}(v)&=& e^{-\int_{\tau}^tL_{K}(f_{\tau_1})(v)\dd \tau_1},\quad 0\le \tau\le t<\infty,\quad v\in \R^3,
  \\
  {\wt{E}}_{\tau}^{t}(v)&=&\exp
  \Big(-a
  \fr{(\min\{M_0, M_2\})^{\fr{\beta+1}{2}}}
   {2^{\beta}{C_2}^{\fr{\beta-1}{2}}}
   (t-\tau)\la v\ra\Big),\quad 0\le \tau\le t<\infty,\quad v\in \R^3,
   \\
   \overline{E}_\tau^t&=&\exp\Big(-a
  \fr{(\min\{M_0, M_2\})^{\fr{\beta+1}{2}}}
   {2^{\beta}{C_2}^{\fr{\beta-1}{2}}}(t-\tau)\Big),
   \quad 0\le \tau\le t<\infty, \quad v\in \R^3
   \eeas
  where
  $
   M_0=\int_{\R^3}f_0(v) \dd v,\,
   M_2=\int_{\R^3}|v|^2f_0(v) \dd v,$
   and $C_2=\max\{1,C_1\}\|f_0\|_{L_3^1}$ is the right hand side of $(\ref{L3 estimate})$ in Proposition \ref{L3}. Then for all $(t,v)\in [0,\infty)\times(\R^3\setminus Z)$
$($where $Z\subset \R^3$ is a null set as mentioned in Proposition \ref{mild solution f}$)$
   we have
  \beqa\label{f(t,v)}
  &&f_t(v)
  \le
  E_0^{t}(v)f_0(v)+(1+2K)\int_0^t {\wt{E}}_{\tau}^{t}(v)Q^+(f_\tau,f_\tau)(v)\dd \tau
  \\\label{f(t,v) 2}
  &&\,\,\,\quad \quad \le
  E_0^{t}(v)f_0(v)+(1+2K)\int_0^t
  \overline{E}_{\tau}^{t}Q^+(f_\tau,f_\tau)(v)\dd \tau,
  \eeqa and
  \beqa\label{E1}
  && E^{t}_{\tau}(v)\le  \wt{E}^{t}_{\tau}(v)
  \le \overline{E}_\tau^t
  ,\quad  0\le \tau\le t<\infty, \, v\in {\R}^3.\eeqa
  In particular
\beqa\label{E2} E_0^{t}(v)
  \le\overline{E}_0^t= \exp\Big(-a
  \fr{(\min\{M_0, M_2\})^{\fr{\beta+1}{2}}}
   {2^{\beta}{C_2}^{\fr{\beta-1}{2}}} t\Big),
   \quad (t,v)\in [0,\infty)\times{\R}^3.
  \eeqa
  \end{proposition}

  \bepf Inequalities (\ref{E1}) and (\ref{E2}) follow from Propositions \ref{L3} and \ref{L estimate}.
For every $v\in {\R}^3\setminus Z$, the function
$t\mapsto L_{K}(f_t)(v)$ is bounded on $[0,\infty)$ so that $t\mapsto \int_0^t L_{K}(f_\tau)(v)\dd \tau$
is Lipschitz continuous hence $t\mapsto e^{\int_0^t L_{K}(f_\tau)(v)\dd \tau}$
is Lipschitz on every bounded interval, it follows that
the function $t\mapsto e^{\int_0^t L_{K}(f_\tau)(v)\dd \tau}f(t,v)$ is absolutely continuous on
every bounded interval. Thus it holds Duhamel's formula:
$$f_t(v)e^{\int_0^t L_{K}(f_\tau)(v)\dd \tau}=f_0(v)+\int_0^t Q_K^+(f_\tau)(v)e^{\int_0^\tau L_{K}(f_{\tau_1})(v)\dd \tau_1}\dd \tau,
  \quad (t,v)\in [0,\infty)\times(\R^3\setminus Z)$$
which is rewritten with $E_0^{t}(v), E_{\tau}^{t}(v)$  as
\beqa\label{E3}f_t(v)=E_0^{t}(v)f_0(v)+\int_0^t E_{\tau}^{t}(v) Q_K^{+}(f_\tau)(v)\dd \tau, \quad (t,v)\in [0,\infty)\times(\R^3\setminus Z).\eeqa
Then (\ref{f(t,v)}), (\ref{f(t,v) 2}) follow from (\ref{E3}) and (\ref{E1}).
\eepf
  \vskip3mm

  From (\ref{f(t,v)}) we see that a uniform pointwise estimate for $f$ can be obtained from
  an appropriate pointwise estimate for $Q^+(f,f)$.  We will establish such an estimate for $Q^+(f,f)$ in Proposition \ref{continuity}. It will be seen that a successful method to do this is to use multi-step iterations of the collision gain operator $Q^{+}(\cdot,\cdot)$, which we put into Propositions \ref{estimate 3}$-$\ref{estimate 2}.

  \begin{proposition}\label{estimate 3}
Let $Q^{+}(\cdot,\cdot)$, $Q_{K}(\cdot)$ be defined in $(\ref{Q+ bilinear})$, $(\ref{QK})$ respectively with  $B$ satisfying $(\ref{condition of kernel})$. For any $0\le f_0\in L^1_3(\R^3)\cap L_3^{\infty}(\R^3)$ satisfying $\int_{{\R}^3} v f_0(v){\rm d}v=0$, let $f_t(v)\equiv f(t,v)$ be the conservative mild solution of Eq.$(\ref{approximate equation 2})$ obtained in Proposition \ref{mild solution f} corresponding to the kernel $B$ with the initial datum $f_0$.
Then for all $t\in [0,\infty), v\in \R^3\setminus Z$
 \beqa\label{Q+f0ft}
  Q^+(f_t,f_0)(v)
  &\le &
 2^5b\pi\|f_0\|_{L^{\infty}}\|f_0\|_{L_1^1}\la v\ra
 \no\\
 &+&(1+2K)\cdot 2^{5+\fr{2}{3}}b^2\pi^{\fr{4}{3}}
   \|f_0\|_{L^1}^{\fr{8}{3}}\|f_0\|_{L_2^1}^{\fr{1}{3}}
   \fr{2^{\beta}{C_2}^{\fr{\beta-1}{2}}}
  {a(\min\{M_0, M_2\})^{\fr{\beta+1}{2}}}
 \eeqa
 where $Z\subset \R^3$ is a null set given in Proposition \ref{mild solution f},
  $
   M_0=\int_{\R^3}f_0(v) \dd v,\,
   M_2=\int_{\R^3}|v|^2f_0(v) \dd v,$ and
    $C_2=\max\{1,C_1\}\|f_0\|_{L_3^1}$ is the right hand side of $(\ref{L3 estimate})$ in Proposition \ref{L3}.

 \end{proposition}

 \bepf From (\ref{f(t,v) 2}), (\ref{exchange}), (\ref{Q1}) and (\ref{Q+(ft,Q+(fs,fs))}) we have for all $(t,v)\in [0,\infty)\times (\R^3\setminus Z)$
   \beas
 Q^+(f_t,f_0)(v)
 & \le & Q^+(f_0,f_0)(v)+(1+2K)\int_0^t \overline{E}_{\tau}^{t} Q^+(Q^+(f_\tau,f_\tau),f_0)(v)\dd \tau
   \no\\
   &\le &
   8b|\bS|\|f_0\|_{L^{\infty}}\|f_0\|_{L_1^1}\la v\ra
   +(1+2K)\int_0^t \overline{E}_{\tau}^{t} 16b^2|\mathbb S^1||\bS|^{\fr{1}{3}}\|f_\tau\|^2_{L^1}\|f_0\|^{\fr{1}{3}}_{L^1}
  \|f_0\|^{\fr{2}{3}}_{L^2}\dd \tau
   \no\\
   &\le &
   8b|\bS|\|f_0\|_{L^{\infty}}\|f_0\|_{L_1^1}\la v\ra
   +(1+2K)16b^2|\mathbb S^1||\bS|^{\fr{1}{3}}
   \|f_0\|_{L^1}^{\fr{7}{3}}\|f_0\|_{L^2}^{\fr{2}{3}}\int_0^t \overline{E}_{\tau}^{t}\dd \tau
   \no\\
   &\le &
   8b|\bS|\|f_0\|_{L^{\infty}}\|f_0\|_{L_1^1}\la v\ra
   +(1+2K)16b^2|\mathbb S^1||\bS|^{\fr{1}{3}}
   \|f_0\|_{L^1}^{\fr{8}{3}}\|f_0\|_{L^1_2}^{\fr{1}{3}}
   \fr{2^{\beta}{C_2}^{\fr{\beta-1}{2}}}
  {a(\min\{M_0, M_2\})^{\fr{\beta+1}{2}}}
   \eeas
   where we used the conservation of mass and energy of $f$ and $\|f_0\|_{L^2}\le \|f_0\|_{L^{\infty}}^{\fr{1}{2}}\|f_0\|_{L^1}^{\fr{1}{2}}$.
   \eepf
\vskip3mm

   \begin{proposition}\label{estimate 1}
 Let $Q^{+}(\cdot,\cdot)$, $Q_{K}(\cdot)$ be defined in $(\ref{Q+ bilinear})$, $(\ref{QK})$ respectively with  $B$ satisfying $(\ref{condition of kernel})$. For any $0\le f_0\in L^1_3(\R^3)\cap L_3^{\infty}(\R^3)$ satisfying $\int_{{\R}^3} v f_0(v){\rm d}v=0$, let $f_t(v)\equiv f(t,v)$ be a conservative  mild solution of Eq.$(\ref{approximate equation 2})$ obtained in Proposition \ref{mild solution f} corresponding to the kernel $B$ with the initial datum $f_0$.
  Then for all $t,\tau \in [0,\infty)$ and $v\in \R^3\setminus Z$
  \beqa\label{Q+(Q+(ft,ft),Q+(fs,fs))}
  Q^+\big(Q^+(f_t,f_t),Q^+(f_\tau,f_\tau)\big)(v)
  &\le&
  2^{8+\fr{1}{2}}b^3\pi^{\fr{7}{3}}\|f_0\|_{L^1}^{\fr{7}{3}}\|f_0\|_{L_2^1}^{\fr{4}{3}}
 \|f_0\|_{L^{\infty}}^{\fr{1}{3}}\no\\
  &+ & (1+2K)2^{12}b^4\pi^{3}
  \|f_0\|_{L^1}^{\fr{23}{6}}\|f_0\|_{L_2^1}^{\fr{7}{6}}
  \fr{2^{\beta}{C_2}^{\fr{\beta-1}{2}}}
  {a(\min\{M_0, M_2\})^{\fr{\beta+1}{2}}}\qquad
  \eeqa
  where $Z\subset \R^3$ is a null set given in Proposition \ref{mild solution f},
  $
   M_0=\int_{\R^3}f_0(v) \dd v,\,
   M_2=\int_{\R^3}|v|^2f_0(v) \dd v,$ and
    $C_2=\max\{1,C_1\}\|f_0\|_{L_3^1}$ is the right hand side of $(\ref{L3 estimate})$ in Proposition \ref{L3}.
  \end{proposition}

  \bepf According to Proposition \ref{Duhamel} we have  for any $t,\tau\in [0,\infty)$ and any $v\in \R^3\setminus Z$
  \beqa\label{Q(ftft,fsfs)}
  &&Q^+\big(Q^+(f_t,f_t),Q^+(f_\tau,f_\tau)\big)(v)
  \le Q^+\big(Q^+(f_t,f_0),Q^+(f_\tau,f_\tau)\big)(v)\no\\
  && +
 (1+2K)\int_0^t\overline{E}_{\tau_1}^{t}Q^+
 \big(Q^+(f_t,Q^+(f_{\tau_1},f_{\tau_1})),Q^+(f_\tau,f_\tau)\big)(v)\dd \tau_1.\quad\quad\quad
  \eeqa
   First we estimate the second term in the right hand side of (\ref{Q(ftft,fsfs)}). Due to  (\ref{Q+(Q+(ft,Q+(fr,fr)),Q+(fs,fs))}), H\"{o}lder's inequality and the conservation of mass and energy  of $f$ we have
 \beqa\label{Q(ft,frfr,fsfs)}
 &&
 \int_0^t \overline{E}_{\tau_1}^{t} Q^+\big(Q^+(f_t,Q^+(f_{\tau_1},f_{\tau_1})),Q^+(f_\tau,f_\tau)\big)(v)\dd \tau_1
  \no\\
  &&\le 2^{7}b^4|\mathbb S^1|^2|\bS|
  \|f_0\|_{L^1}^{\fr{23}{6}}\|f_0\|_{L_2^1}^{\fr{7}{6}}\int_0^t \overline{E}_{\tau_1}^{t} \dd \tau_1
  \no\\
  &&\le 2^{7}b^4|\mathbb S^1|^2|\bS|
  \|f_0\|_{L^1}^{\fr{23}{6}}\|f_0\|_{L_2^1}^{\fr{7}{6}}
  \fr{2^{\beta}{C_2}^{\fr{\beta-1}{2}}}
  {a(\min\{M_0, M_2\})^{\fr{\beta+1}{2}}}.
  \eeqa
  For the first term in the right hand side of (\ref{Q(ftft,fsfs)}) we deduce  from (\ref{exchange}), (\ref{Q+(f0,ft)L1})$-$(\ref{Q+(Q+(ft,Q+(fr,fr)),Q+(fs,fs))}) that
  \beqa\label{Q+(Q+(ft,f0),Q+(fs,fs))}
  &&Q^+\big(Q^+(f_t,f_0),Q^+(f_\tau,f_\tau)\big)(v)\no\\
  &&\le Q^+\big(Q^+(f_0,f_0),Q^+(f_\tau,f_\tau)\big)(v)\no\\
  && +
 (1+2K)\int_0^t\overline{E}_{\tau_1}^{t}Q^+
 \big(Q^+(f_0,Q^+(f_{\tau_1},f_{\tau_1})),Q^+(f_\tau,f_\tau)\big)(v)\dd \tau_1\no\\
  &&\le
  16b^2|\mathbb S^1||\bS|^{\fr{1}{3}}\|f_0\|^2_{L^1}
  (b|\bS|\|f_0\|_{L_2^1}^2)^{\fr{1}{3}}
  \big(2^{\fr{5}{4}}b|\bS|\|f_0\|_{L^{\infty}}^{\fr{1}{2}}
  \|f_0\|_{L^1}^{\fr{1}{2}}\|f_0\|_{L_2^1}\big)^{\fr{2}{3}}
  \no\\
  &&+(1+2K)\int_0^t\overline{E}_{\tau_1}^{t}Q^+
 \big(Q^+(f_0,Q^+(f_{\tau_1},f_{\tau_1})),Q^+(f_\tau,f_\tau)\big)(v)\dd \tau_1\no\\
 &&\le
 2^{4+\fr{5}{6}}b^3|\mathbb S^1||\bS|^{\fr{4}{3}}
 \|f_0\|_{L^1}^{\fr{7}{3}}\|f_0\|_{L_2^1}^{\fr{4}{3}}
 \|f_0\|_{L^{\infty}}^{\fr{1}{3}}\no\\
 &&+(1+2K)2^{7}b^4|\mathbb S^1|^2|\bS|
  \|f_0\|_{L^1}^{\fr{23}{6}}\|f_0\|_{L_2^1}^{\fr{7}{6}}
  \fr{2^{\beta}{C_2}^{\fr{\beta-1}{2}}}
  {a(\min\{M_0, M_2\})^{\fr{\beta+1}{2}}}
  \eeqa
  where we used the conservation of mass and energy of $f$.
Combining (\ref{Q(ftft,fsfs)}), (\ref{Q(ft,frfr,fsfs)}) and (\ref{Q+(Q+(ft,f0),Q+(fs,fs))})
 gives (\ref{Q+(Q+(ft,ft),Q+(fs,fs))}).
 \eepf
  \vskip3mm

  \begin{proposition}\label{estimate 2}
  Let $Q^{+}(\cdot,\cdot)$, $Q_{K}(\cdot)$ be defined in $(\ref{Q+ bilinear})$, $(\ref{QK})$ respectively with  $B$ satisfying $(\ref{condition of kernel})$. For any $0\le f_0\in L^1_3(\R^3)\cap L_3^{\infty}(\R^3)$ satisfying $\int_{{\R}^3} v f_0(v){\rm d}v=0$, let $f_t(v)\equiv f(t,v)$ be a mild solution of Eq.$(\ref{approximate equation 2})$ obtained in Proposition \ref{mild solution f} corresponding to the kernel $B$ with the initial datum $f_0$.
  Then for any  $t,\tau \in [0,\infty)$ and $v\in \R^3\setminus Z$
  \beqa\label{2}
  &&Q^+(f_t,Q^+(f_\tau,f_\tau))(v)
  \no\\
  &&\le
  2^{5+\fr{2}{3}}b^2\pi^{\fr{4}{3}}
  \|f_0\|_{L^1}^{\fr{8}{3}}\|f_0\|_{L^{\infty}}^{\fr{1}{3}}
  +(1+2K)\fr{2^{\beta}{C_2}^{\fr{\beta-1}{2}}}
  {a(\min\{M_0, M_2\})^{\fr{\beta+1}{2}}}
  \Bigg(
  2^{8+\fr{1}{2}}b^3\pi^{\fr{7}{3}}\|f_0\|_{L^1}^{\fr{7}{3}}\|f_0\|_{L_2^1}^{\fr{4}{3}}
 \|f_0\|_{L^{\infty}}^{\fr{1}{3}}\no\\
  &&\quad + (1+2K)2^{12}b^4\pi^{3}
  \|f_0\|_{L^1}^{\fr{23}{6}}\|f_0\|_{L_2^1}^{\fr{7}{6}}
  \fr{2^{\beta}{C_2}^{\fr{\beta-1}{2}}}
  {a(\min\{M_0, M_2\})^{\fr{\beta+1}{2}}}
  \Bigg)
  \eeqa
 where $Z\subset \R^3$ is a null set given in Proposition \ref{mild solution f},
  $
   M_0=\int_{\R^3}f_0(v) \dd v,\,
   M_2=\int_{\R^3}|v|^2f_0(v) \dd v,$ and
    $C_2=\max\{1,C_1\}\|f_0\|_{L_3^1}$ is the right hand side of $(\ref{L3 estimate})$ in Proposition \ref{L3}.
  \end{proposition}

  \bepf From (\ref{exchange}), (\ref{Q+(ft,Q+(fs,fs))}), (\ref{f(t,v) 2}) and (\ref{Q+(Q+(ft,ft),Q+(fs,fs))}) we have for any $t,\tau \in [0,\infty)$ and any $v\in \R^3\setminus Z$
  \beas
  &&Q^+(f_t,Q^+(f_\tau,f_\tau))(v)
  \no\\
  &&\le Q^+\Big(E_0^{t}f_0+(1+2K)\int_0^t \overline{E}_{\tau_1}^{t}Q^+(f_{\tau_1},f_{\tau_1})\dd \tau_1,Q^+(f_\tau,f_\tau)\Big)(v)
  \no\\
  &&\le Q^+(f_0,Q^+(f_\tau,f_\tau))(v)
  +(1+2K)\int_0^t \overline{E}_{\tau_1}^{t}Q^+\big(Q^+(f_{\tau_1},f_{\tau_1}),Q^+(f_\tau,f_\tau)\big)\dd \tau_1
  \no\\
  &&\le
  16b^2|\mathbb S^1||\bS|^{\fr{1}{3}}
  \|f_0\|_{L^1}^{\fr{8}{3}}\|f_0\|_{L^{\infty}}^{\fr{1}{3}}
  +(1+2K)\fr{2^{\beta}{C_2}^{\fr{\beta-1}{2}}}
  {a(\min\{M_0, M_2\})^{\fr{\beta+1}{2}}}
  \Bigg(
  2^{4+\fr{5}{6}}b^3|\mathbb S^1||\bS|^{\fr{4}{3}}
 \|f_0\|_{L^1}^{\fr{7}{3}}\|f_0\|_{L_2^1}^{\fr{4}{3}}
 \|f_0\|_{L^{\infty}}^{\fr{1}{3}}\no\\
  &&\quad + (1+2K)2^{8}b^4|\mathbb S^1|^2|\bS|
  \|f_0\|_{L^1}^{\fr{23}{6}}\|f_0\|_{L_2^1}^{\fr{7}{6}}
  \fr{2^{\beta}{C_2}^{\fr{\beta-1}{2}}}
  {a(\min\{M_0, M_2\})^{\fr{\beta+1}{2}}}
  \Bigg)
  \eeas
  where we used the conservation of mass and energy of $f$.
This gives (\ref{2}).
\eepf
\vskip3mm

   \begin{proposition}\label{continuity}
 Let $Q^{+}(\cdot,\cdot)$, $Q_{K}(\cdot)$ be defined in $(\ref{Q+ bilinear})$, $(\ref{QK})$ respectively with  $B$ satisfying $(\ref{condition of kernel})$. For any $0\le f_0\in L^1_3(\R^3)\cap L_3^{\infty}(\R^3)$ satisfying $\int_{{\R}^3} v f_0(v){\rm d}v=0$, let $f_t(v)\equiv f(t,v)$ be a mild solution of Eq.$(\ref{approximate equation 2})$ obtained in Proposition \ref{mild solution f} corresponding to the kernel $B$ with the initial datum $f_0$.
Then for any $(t,v)\in [0,\infty)\times(\R^3\setminus Z)$ we have
   \beqa\label{Q+(ft,ft)}
   Q^{+}(f_t,f_t)(v)
   \le A(f_0)\la v\ra+P(f_0)\eeqa
   where
   \beas
   && A(f_0)=2^5b\pi\|f_0\|_{L^{\infty}}\|f_0\|_{L_1^1},\\
   &&
   P(f_0)=(1+2K) 2^{6+\fr{1}{3}}b^2\pi^{\fr{4}{3}}
   \|f_0\|_{L^1}^{\fr{8}{3}}\|f_0\|_{L^{\infty}}^{\fr{1}{3}}
   \fr{2^{\beta}{C_2}^{\fr{\beta-1}{2}}}
  {a(\min\{M_0, M_2\})^{\fr{\beta+1}{2}}}
  \\
   &&\quad +\fr{2^{\beta}{C_2}^{\fr{\beta-1}{2}}}
  {a(\min\{M_0, M_2\})^{\fr{\beta+1}{2}}}\Bigg\{(1+2K)
   2^{6+\fr{1}{2}}b^2\pi^{\fr{4}{3}}
   \|f_0\|_{L^1}^{\fr{8}{3}}\|f_0\|_{L^{\infty}}^{\fr{1}{3}}
  +(1+2K)^2
  \fr{2^{\beta}{C_2}^{\fr{\beta-1}{2}}}
  {a(\min\{M_0, M_2\})^{\fr{\beta+1}{2}}}\no\\
  &&\quad\times\bigg[
  2^{8+\fr{1}{2}}b^3\pi^{\fr{7}{3}}\|f_0\|_{L^1}^{\fr{7}{3}}\|f_0\|_{L_2^1}^{\fr{4}{3}}
  \|f_0\|_{L^{\infty}}^{\fr{1}{3}} 
  +(1+2K)2^{12}b^4\pi^{3}
  \|f_0\|_{L^1}^{\fr{23}{6}}\|f_0\|_{L_2^1}^{\fr{7}{6}}
  \fr{2^{\beta}{C_2}^{\fr{\beta-1}{2}}}
  {a(\min\{M_0, M_2\})^{\fr{\beta+1}{2}}}
  \bigg]\Bigg\}
  \eeas
 where $Z\subset \R^3$ is a null set given in Proposition \ref{mild solution f},
  $
   M_0=\int_{\R^3}f_0(v) \dd v,\,
   M_2=\int_{\R^3}|v|^2f_0(v) \dd v,$ and
    $C_2=\max\{1,C_1\}\|f_0\|_{L_3^1}$ is the right hand side of $(\ref{L3 estimate})$ in Proposition \ref{L3}.
  \end{proposition}

   \bepf This is just a calculation:
  according to Propositions \ref{Duhamel}$-$\ref{estimate 2},
   we have for all $(t,v)\in [0,\infty)\times(\R^3\setminus Z)$
\beqa\label{Q+(ft,ft)}
   &&Q^{+}(f_t,f_t)(v)\le Q^{+}\Big( E_0^{t}f_0+(1+2K)\int_{0}^{t}\overline{E}_{\tau}^{t}Q^{+}(f_\tau,f_\tau)\dd \tau, \,f_t \Big)(v)
   \no\\
   &&\le Q^{+}(f_0,f_t)(v)+(1+2K)\int_{0}^{t}\overline{E}_{\tau}^{t}
   Q^{+}(Q^{+}(f_\tau,f_\tau),f_t)(v)\dd \tau
   \no\\
   &&\le
   8b|\bS|\|f_0\|_{L^{\infty}}\|f_0\|_{L_1^1}\la v\ra
   +(1+2K)16b^2|\mathbb S^1||\bS|^{\fr{1}{3}}
   \|f_0\|_{L^1}^{\fr{8}{3}}\|f_0\|_{L^{\infty}}^{\fr{1}{3}}
   \fr{2^{\beta}{C_2}^{\fr{\beta-1}{2}}}
  {a(\min\{M_0, M_2\})^{\fr{\beta+1}{2}}}
   \no\\
   &&\quad +\int_{0}^{t}\overline{E}_{\tau}^{t}\dd \tau\Bigg\{(1+2K)
   16b^2|\mathbb S^1||\bS|^{\fr{1}{3}}
   \|f_0\|_{L^1}^{\fr{8}{3}}\|f_0\|_{L^{\infty}}^{\fr{1}{3}}
  +(1+2K)^2
  \fr{2^{\beta}{C_2}^{\fr{\beta-1}{2}}}
  {a(\min\{M_0, M_2\})^{\fr{\beta+1}{2}}}\no\\
  &&\times\bigg[
  2^{4+\fr{5}{6}}b^3|\mathbb S^1||\bS|^{\fr{4}{3}}
 \|f_0\|_{L^1}^{\fr{7}{3}}\|f_0\|_{L_2^1}^{\fr{4}{3}}
 \|f_0\|_{L^{\infty}}^{\fr{1}{3}}
 + (1+2K)2^{8}b^4|\mathbb S^1|^2|\bS|
  \|f_0\|_{L^1}^{\fr{23}{6}}\|f_0\|_{L_2^1}^{\fr{7}{6}}
  \fr{2^{\beta}{C_2}^{\fr{\beta-1}{2}}}
  {a(\min\{M_0, M_2\})^{\fr{\beta+1}{2}}}
  \bigg]\Bigg\}
  \no\\
  &&\le
  8b|\bS|\|f_0\|_{L^{\infty}}\|f_0\|_{L_1^1}\la v\ra
   +(1+2K)16b^2|\mathbb S^1||\bS|^{\fr{1}{3}}
   \|f_0\|_{L^1}^{\fr{8}{3}}\|f_0\|_{L^{\infty}}^{\fr{1}{3}}
   \fr{2^{\beta}{C_2}^{\fr{\beta-1}{2}}}
  {a(\min\{M_0, M_2\})^{\fr{\beta+1}{2}}}
   \no\\
   &&\quad +\fr{2^{\beta}{C_2}^{\fr{\beta-1}{2}}}
  {a(\min\{M_0, M_2\})^{\fr{\beta+1}{2}}}\Bigg\{(1+2K)
   16b^2|\mathbb S^1||\bS|^{\fr{1}{3}}
   \|f_0\|_{L^1}^{\fr{8}{3}}\|f_0\|_{L^{\infty}}^{\fr{1}{3}}
  +(1+2K)^2
  \fr{2^{\beta}{C_2}^{\fr{\beta-1}{2}}}
  {a(\min\{M_0, M_2\})^{\fr{\beta+1}{2}}}
  \no\\
  &&\quad\times\bigg[
  2^{4+\fr{5}{6}}b^3|\mathbb S^1||\bS|^{\fr{4}{3}}
 \|f_0\|_{L^1}^{\fr{7}{3}}\|f_0\|_{L_2^1}^{\fr{4}{3}}
 \|f_0\|_{L^{\infty}}^{\fr{1}{3}}
  + (1+2K)2^{8}b^4|\mathbb S^1|^2|\bS|
  \|f_0\|_{L^1}^{\fr{23}{6}}\|f_0\|_{L_2^1}^{\fr{7}{6}}
  \fr{2^{\beta}{C_2}^{\fr{\beta-1}{2}}}
  {a(\min\{M_0, M_2\})^{\fr{\beta+1}{2}}}
  \bigg]\Bigg\}\no\\
  &&=A(f_0)\la v\ra +P(f_0).\no\eeqa
  \eepf
\vskip3mm

\section{Proof of Theorem \ref{main results}}

Having made sufficient preparations in previous sections we can now turn to
the

{\bf Proof of Theorem \ref{main results}.} To be clear we divide the proof into four steps.

 {\bf Step 1}.  We prove that the mild solution $f$ of Eq.(\ref{approximate equation 2}) obtained in
  Proposition {\ref{mild solution f} is a mild solution of Eq.(\ref{Equation}) by proving $f\le K$ with a suitable constant $K>\|f_0\|_{L^{\infty}}$.  Let $Z\subset \R^3$ be the null set appeared in Proposition \ref{mild solution f} and
  let
  $$C_0=\Big(\fr{22}{5}\sqrt{2}-\fr{31}{5}\Big)\pi,\quad C_2=\{1,C_1\}\|f_0\|_{L_3^1},\quad  C_{3}=\fr{2^{\beta}{C_2}^{\fr{\beta-1}{2}}}
  {a(\min\{M_0, M_2\})^{\fr{\beta+1}{2}}}$$
  where $C_1$ is given in (\ref{C1 def}) in Proposition \ref{L3}.
   Using Propositions \ref{Duhamel} and \ref{continuity} we compute for all $(t,v)\in [0,\infty)\times(\R^3\setminus Z)$
  \beqa\label{estimate}
  f(t,v)
  &\le&
  f_0(v)+(1+2K)\int_0^t \wt{E}^{t}_{\tau}(v)A(f_0)\la v\ra + \overline{E}_{\tau}^{t}P(f_0)\dd \tau
  \no\\
  &\le&
  f_0(v)
  +(1+2K)C_{3}
  \Bigg\{8b|\bS|\|f_0\|_{L^{\infty}}\|f_0\|_{L_1^1}
  +(1+2K)16b^2|\mathbb S^1||\bS|^{\fr{1}{3}}
   \|f_0\|_{L^1}^{\fr{8}{3}}\|f_0\|_{L^{\infty}}^{\fr{1}{3}}
   C_{3}
   \no\\
   &&+C_3\bigg[(1+2K)
   16b^2|\mathbb S^1||\bS|^{\fr{1}{3}}
   \|f_0\|_{L^1}^{\fr{8}{3}}\|f_0\|_{L^{\infty}}^{\fr{1}{3}}
  +(1+2K)^2C_{3}\no\\
  &&\times\bigg(
  2^{4+\fr{5}{6}}b^3|\mathbb S^1||\bS|^{\fr{4}{3}}
 \|f_0\|_{L^1}^{\fr{7}{3}}\|f_0\|_{L_2^1}^{\fr{4}{3}}
 \|f_0\|_{L^{\infty}}^{\fr{1}{3}}
   + (1+2K)2^{8}b^4|\mathbb S^1|^2|\bS|
  \|f_0\|_{L^1}^{\fr{23}{6}}\|f_0\|_{L_2^1}^{\fr{7}{6}}C_{3}
  \bigg)\bigg]\Bigg\}
  \no\\
  &\le& f_0(v)+ 2^8|\mathbb S^1|^2|\bS|(1+2K)^4 C_4
 \eeqa
where
  \beqa\label{CC4}&&C_4= b C_3(M_0+M_2)\|f_0\|_{L^{\infty}}
  + (b C_3)^2 {M_0}^{\fr{8}{3}}\|f_0\|_{L^{\infty}}^{\fr{1}{3}}
  +(b C_3)^3{M_0}^{\fr{7}{3}}(M_0+M_2)^{\fr{4}{3}}
  \|f_0\|_{L^{\infty}}^{\fr{1}{3}}
  \no\\
  &&
  \qquad\,\,\,
  +(b C_3)^4
  {M_0}^{\fr{23}{6}}(M_0+M_2)^{\fr{7}{6}}.\eeqa
  From (\ref{estimate}) we see that in order to prove $f(t,v)\le K$ for all $(t,v)\in [0,\infty)\times (\R^3\setminus Z)$, it suffices to prove that
  $\|f_0\|_{L^{\infty}}+2^8|\mathbb S^1|^{2}|\bS|(1+2K)^4 C_4\le K$, i.e. to prove that
  \beqa\label{to prove}
  &&   C_4
  \le \fr{K-\|f_0\|_{L^{\infty}}}
  {2^8 |\mathbb S^1|^{2}|\bS| (1+2K)^4}.
  \eeqa
 Denoting
 $\rho_0:=M_2/M_0$,
 we will prove (\ref{to prove}) by discussing the following four cases.

{\bf Case 1:} $ C_1 \ge 1$ and $\rho_0\ge 1$.
 In this case we have
  \beas
  bC_{3}
  \le
  C(a,b)\rho_0^{\beta}\cdot \fr{1}{M_2}
  \eeas
  where
  \beas
  C(a,b)=\fr{2^{2\beta-1}(2^6(1+2K)|\bS|)^{\fr{\beta-1}{2}}}
  {C_0^{\fr{\beta-1}{2}}}\Big(\fr{b}{a}\Big)^{\fr{\beta+1}{2}}
  \eeas
  and we have used $0<C_0<1, 0<a\le b<\infty$.
 By definition of $C_4$ in (\ref{CC4}), $\rho_0\ge 1$ and $C(a,b)> 1$,  this gives
 \beqa\label{C1}
  && C_4+\fr{\|f_0\|_{L^{\infty}}}
  {2^8 |\mathbb S^1|^{2}|\bS| (1+2K)^4}\le
  C(a,b)\rho_0^{\beta-1}
   (1+\rho_0)\|f_0\|_{L^{\infty}}
   +
 C(a,b)^2\rho_0^{2(\beta-1)} M_0^{\fr{2}{3}}
  \|f_0\|_{L^{\infty}}^{\fr{1}{3}}\no\\
  &&\quad\,\,+
  +C(a,b)^3\rho_0^{3(\beta-1)}
  (1+\rho_0)^{\fr{4}{3}}
  M_0^{\fr{2}{3}} \|f_0\|_{L^{\infty}}^{\fr{1}{3}}
  +
  C(a,b)^4\rho_0^{4(\beta-1)}(1+\rho_0)M_0+\|f_0\|_{L^\infty}
  \no\\
  &&\quad\,\,
  \le
  4\rho_0^{4\beta-3}C(a,b)^4
  \big(
   M_0+\|f_0\|_{L^{\infty}}
  \big)
 \eeqa
 where we used the inequality $x^{p}y^{1-p}\le px+(1-p)y$ for all $x,y>0$ and $0\le p\le 1$.
Combining  this with (\ref{C1}) we see that in order to prove (\ref{to prove})  it suffices to prove that
 \beqa\label{CASE1}
 &&\rho_0^{4\beta-3} (M_0+ \|f_0\|_{L^{\infty}}
  )
  \le
  \fr{C_0^{2(\beta-1)}K}
  {2^{20\beta-6}
  |\mathbb S^1|^{2}|\bS|^{2\beta-1}
  (1+2K)^{2(\beta+1)}}\Big(\fr{a}{b}\Big)^{2\beta+2}.
 \eeqa

 {\bf Case 2:}
 $ C_1 \ge 1$ and
 $\rho_0< 1$.
For this case we have
 \beas
 b C_{3}
 \le C(a,b)\fr{1}{{\rho_0}^{\beta}}\fr{1}{M_0}.
 \eeas
 An argument similar to the one used in {\bf Case1} shows that
 \beqa\label{C2}
  &&C_4+\fr{\|f_0\|_{L^{\infty}}}
  {2^8 |\mathbb S^1|^{2}|\bS| (1+2K)^4}\le
   4C(a,b)^4\fr{1}{{\rho_0}^{4\beta}}(
   M_0+ \|f_0\|_{L^{\infty}}) \quad\quad
 \eeqa
Following the same argument as in {\bf Case 1} we see that in this case in order to prove (\ref{to prove})  it suffices to prove that
 \beqa\label{CASE2}
 &&
  \fr{1}{{\rho_0}^{4\beta}}(
   M_0+ \|f_0\|_{L^{\infty}})
  \le
  \fr{C_0^{2(\beta-1)}K}
  {2^{20\beta-6}
  |\mathbb S^1|^{2}|\bS|^{2\beta-1}
  (1+2K)^{2(\beta+1)}}\Big(\fr{a}{b}\Big)^{2\beta+2}.
 \eeqa

 {\bf Case 3:}
 $ C_1< 1 $ and $\rho_0\ge1$.
 In this case we have
 \beas
 bC_{3}\le
 2^{\beta}
 \Big(\fr{\|f_0\|_{L_3^1}}{M_0}\Big)^{\fr{\beta-1}{2}}\fr{b}{a}\cdot
 \fr{1}{M_0}
 :=\wt{C}(a,b, M_0,\|f_0\|_{L_3^1})\fr{1}{M_0},
 \eeas
 Inserting this into the expression of $C_4$ and noticing that
 $\wt{C}(a,b, M_0,\|f_0\|_{L_3^1})\ge 1$  we compute
 \beqa\label{C3}
 && C_4+\fr{\|f_0\|_{L^{\infty}}}
  {2^8 |\mathbb S^1|^{2}|\bS| (1+2K)^4}\le
   4\rho_0\wt{C}(a,b, M_0,\|f_0\|_{L_3^1})^4(
   M_0+ \|f_0\|_{L^{\infty}}).
  \eeqa
Thus, to prove (\ref{to prove}), it suffices to prove that
  \beqa\label{MM}   M_0+ \|f_0\|_{L^{\infty}}
   \le
   \fr{K}
  { 2^{10} |\mathbb S^1|^{2}|\bS|\rho_0\wt{C}(a,b,M_0,\|f_0\|_{L_3^1})^4(1+2K)^4
   }.
  \eeqa
 While from the definition of $\wt{C}(a,b, M_0,\|f_0\|_{L_3^1})$ we see that in order to prove (\ref{MM})
 it suffices to prove that
  \beqa\label{CASE3}
  (M_0+ \|f_0\|_{L^{\infty}})
 \Big(\fr{\|f_0\|_{L_3^1}}{M_0}\Big)^{2\beta-1}
   \le
   \fr{K}
  {2^{4\beta+10}|\mathbb S^1|^{2}|\bS|(1+2K)^4}\Big(\fr{a}{b}\Big)^4.
  \eeqa

  {\bf Case 4:}
 $ C_1< 1 $ and $\rho_0< 1$.
 For this case we have
 \beas
 bC_{3}\le
 2^{\beta}
 \Big(\fr{\|f_0\|_{L_3^1}}{M_2}\Big)^{\fr{\beta-1}{2}}\fr{b}{a}\cdot
 \fr{1}{M_2}
 =:\overline{C}(a,b,M_2,\|f_0\|_{L_3^1})\fr{1}{M_2}.
 \eeas
 An argument similar to the one used in {\bf Case3} shows that
 \beqa\label{C4}
 &&
 C_4+\fr{\|f_0\|_{L^{\infty}}}
  {2^8 |\mathbb S^1|^{2}|\bS| (1+2K)^4}\le
   4\fr{1}{{\rho_0}^4}\overline{C}(a,b,M_2,\|f_0\|_{L_3^1})^4
  (
   M_0+ \|f_0\|_{L^{\infty}}
  ).
  \eeqa
  Following the same argument as in {\bf Case 3}, we see that in order to prove (\ref{to prove}) it suffices to prove
  \beqa\label{CASE4}
  (M_0+ \|f_0\|_{L^{\infty}})
 \Big(\fr{\|f_0\|_{L_3^1}}{M_2}\Big)^{2(\beta+1)}
   \le
   \fr{K}
  {2^{4\beta+10}|\mathbb S^1|^{2}|\bS|(1+2K)^4}\Big(\fr{a}{b}\Big)^4.
  \eeqa

 Now we summarize these four cases.
  From (\ref{CASE1}), (\ref{CASE2}), (\ref{CASE3}) and (\ref{CASE4}) we see that in order to prove (\ref{to prove}) it suffices to prove
 \beas
 (\|f_0\|_{L^1}+ \|f_0\|_{L^{\infty}})
  \Big(\fr{\|f_0\|_{L_3^1}}{\min\{M_0,M_2\}}\Big)^{4\beta}
  \le
  \fr{C_0^{2(\beta-1)}K}
  {2^{20\beta-6}
  |\mathbb S^1|^{2}|\bS|^{2\beta-1}
  (1+2K)^{2(\beta+1)}}\Big(\fr{a}{b}\Big)^{2\beta+2}.
 \eeas
 Maximizing the right hand side of it with respect to $K$ gives that a good choice of $K$ is $K=\fr{1}{4\beta+2}$. Thus to prove proved (\ref{to prove}) with $K=\fr{1}{4\beta+2}$, we need only to prove that
  \beas
 (\|f_0\|_{L^1}+ \|f_0\|_{L^{\infty}})
  \Big(\fr{\|f_0\|_{L_3^1}}{\min\{M_0,M_2\}}\Big)^{4\beta}
  \le \fr{1}
  {2^{35\beta-11}}
  \fr{(4\beta+2)^{2\beta+1}}
  {(4\beta+4)^{2\beta+2}}\Big(\fr{a}{b}\Big)^{2(\beta+1)},
  \eeas
but this is just the given condition (\ref{condition}).
Thus we have proved that
$f(t,v)\le K=\fr{1}{4\beta+2}$ for all $(t,v)\in [0,\infty)\times (\R^3\setminus Z)$.
Then by elementary calculation we have $Q_K(f)(t,v)=Q(f)(t,v)$ for ${\rm a.e}$ $(t,v)\in [0,\infty)\times \R^3$. This implies that $f(t,v)=f_0(v)+\int_0^t Q(f)(\tau,v)\dd \tau$ for ${\rm a.e.}$ $(t,v)\in [0,\infty)\times \R^3$. Since for almost every  $v\in \Og,\,
t\mapsto f(t,v),\, t\mapsto \int_0^t Q(f)(\tau,v)\dd \tau$ are  continuous on $[0,\infty)$, it follows from Lemma \ref{Lemma H} that there exists a null set (still denote it as $Z$) $Z\subset \R^3$, such that
for all $(t,v)\in [0,\infty)\times(\R^3\setminus Z)$, $f(t,v)=f_0(v)+\int_0^t Q(f)(\tau,v)\dd \tau$, i.e.
    \beqa\label{1.4 mild solution}
    f(t,v)=f_0(v)+\int_0^t\dd \tau \iint\limits_{\R^3\times\SP^2}B(v-v_*,\si)\big[f'f'_{*}(1+ f+ f_{*})-ff_{*}(1+ f'+ f'_{*})\big]\dd \si\dd v_{*}.
    \eeqa
This proves that $f$ is a mild solution of Eq.({\ref{Equation}).
Moreover combining (\ref{C1}), (\ref{C2}), (\ref{C3}) and (\ref{C4}) leads to
    \beqa\label{about infty}
    &&f(t,v)
    \le
   2^{35\beta}
    \Big(\fr{b}{a}\Big)^{2(\beta+1)}
    \fr{(2\beta+2)^{2\beta-2}}{(2\beta+3)^{2\beta-2}}
    \Big(\fr{\|f_0\|_{L_3^1}}{\min\{M_0,M_2\}}\Big)^{4\beta}
   (\|f_0\|_{L^1}+ \|f_0\|_{L^{\infty}})=:C_{L^{\infty}}(f_0)\qquad
    \eeqa
     for all $(t,v)\in [0,\infty)\times(\R^3\setminus Z)$.

{\bf Step 2.} We will use Lemma \ref{lemma9} to prove the high temperature condition (\ref{Kinetic high change}).
Define
$
\phi(r):=\fr{1}{4\pi \|f_0\|_{L^{\infty}}}\int_{\SP^2}f_0(r\si)\dd \si
$
for $r\in [0,\infty).$
Then $0\le \phi(r)\le 1$ for all $r\in [0,\infty)$ and
\beas
\fr{M_0}{4\pi \|f_0\|_{L^{\infty}}}=\int_0^{\infty}r^2 \phi(r)\dd r,
\quad
\fr{M_2}{4\pi \|f_0\|_{L^{\infty}}}=\int_0^{\infty}r^4 \phi(r)\dd r.
\eeas
  According to the inequality (\ref{Tc estimate}), it follows from (\ref{rate of Kinetic temperature}) and (\ref{condition}) that
\beas
\fr{\overline{T}}{\overline{T}_{c}}
\ge
\fr{2\pi[\zeta(3/2)]^{\fr{5}{3}}}{3\zeta(5/2)}\fr{3^{\fr{5}{3}}}{5(4\pi)^{\fr{2}{3}}
\|f_0\|_{L^{\infty}}^{\fr{2}{3}}}
\ge
2^{\fr{70\beta}{3}-\fr{19}{3}}
\fr{(4\beta+4)^{\fr{4\beta+4}{3}}}{(4\beta+2)^{\fr{4\beta+2}{3}}}
\Big(\fr{b}{a}\Big)^{\fr{4\beta+4}{3}}
>>1,
\eeas
which proves (\ref{Kinetic high change}).

{\bf Step 3.} We prove that, after a modification on a $v$-null set, the bounded mild solution $f$ of Eq.(\ref{approximate equation 2}) obtained in {\bf Step 1} is a solution of Eq.(\ref{Equation}).
To do this we first derive a similar version of Duhamel's formula.
For every  $v\in {\R}^3\setminus Z$ (with $mes(Z)=0$) and every $t\ge 0$
 \beas
 &&f(t,v)=f_0(v)+\int_0^t \dd \tau\iint_{\bRS}B(v-v_*,\si)
 \big[f'f_*'(1+f_*)+ff'f_*'-ff_*(1+f'+f_*')\big]\dd \si \dd v_*
 \\
 &&\qquad\quad=f_0(v)+\int_0^t \wt{Q}^{+}(f)(\tau,v)\dd \tau-\int_0^tf(\tau,v)\wt{L}(f)(\tau,v)\dd \tau
 \eeas
 where
 \beqa\label{wtQ}&& \wt{Q}^{+}(f)(t,v)=\iint_{\bRS}B(v-v_*,\si)
 f'f_*'(1+f_*)\dd \si \dd v_*, \\
 \label{wtL}
 &&\wt{L}(f)(t,v)=\iint_{\bRS}B(v-v_*,\si)\big[f_*(1+f'+f_*')-f'f_*'\big]\dd \si \dd v_*.
 \eeqa
 Thus for every $v\in {\R}^3\setminus Z$ and for almost every $t\in [0,\infty)$
\beas
\fr{\p f(t,v)}{\p t}=\wt{Q}^{+}(f)(t,v)-f(t,v)\wt{L}(f)(t,v).
\eeas
As before,  for every $v\in {\R}^3\setminus Z$ the function
$t\mapsto \wt{L}(f)(t,v)$ is bounded and the function $t\mapsto \int_0^t\wt{L}(f)(\tau,v)\dd \tau$
is Lipschitz continuous hence $t\mapsto e^{\int_0^t\wt{L}(f)(\tau,v)\dd \tau}$
is Lipschitz on every bounded interval. It follows that
the function $t\mapsto e^{\int_0^t\wt{L}(f)(\tau,v)\dd \tau}f(t,v)$ is absolutely continuous on
every bounded interval and thus it holds  Duhamel's formula:
\beqa\label{Duhamel1}
f(t,v)=f_0(v)e^{-\int_0^t\wt{L}(f)(\tau,v)\dd \tau}+\int_0^te^{-\int_\tau^t\wt{L}(f)(\tau_1,v)\dd \tau_1}
\wt{Q}^{+}(f)(\tau,v)\dd \tau
\eeqa
for all $(t,v)\in [0,\infty)\times(\R^3\setminus Z)$.

To prove the continuity of the solution
 $f$ we also need the following proposition concerning the continuity of the collision integrals
 (\ref{wtQ}), (\ref{wtL}):

 \begin{proposition}\label{continuity about t,v}
 Let $B$ be the collision kernel satisfying $(\ref{condition of kernel})$,  $Q$ the collision operators defined in $(\ref{Q(f)})$. Let $0\le f_0\in L^1_3(\R^3)\cap L_3^{\infty}(\R^3)\cap C({\R}^3)$ satisfy $\int_{{\R}^3} v f_0(v){\rm d}v=0$ and  $(\ref{condition})$, and
 let $f$ with the initial datum $f_0$ be a mild solution of Eq.$(\ref{Equation})$ obtained in
 {\bf Step 1} of the proof of Theorem 1.1. Then
 $(t,v)\mapsto \wt{Q}^{+}(f)(t,v),\, (t,v)\mapsto \wt{L}(f)(t,v)$
are both continuous on $[0,\infty)\times {\R}^3$.
\end{proposition}
\vskip2mm

 A proof of this proposition will be given later. Right now we continue the proof of Theorem \ref{main results}. Let us define
\beqa\label{Duhamel2}
g(t,v)=f_0(v)e^{-\int_0^t\wt{L}(f)(\tau,v)\dd \tau}+\int_0^te^{-\int_\tau^t\wt{L}(f)(\tau_1,v)\dd \tau_1}\wt{Q}^{+}(f)(\tau,v)\dd \tau,\quad (t,v)\in[0,\infty)\times{\mR}^3.
\eeqa
We see from  Proposition \ref{continuity about t,v} that $0\le g\in C([0,\infty)\times\R^3)$.
Comparing (\ref{Duhamel2}) with (\ref{Duhamel1})
 we have $g(t,v)=f(t,v)$ for any $(t,v)\in [0,\infty)\times(\R^3\setminus Z)$ and $\sup\limits_{0\le t\le T}\|Q^{\pm}(g)(t)\|_{L^1}<\infty$, $\sup\limits_{0\le t\le T}\|Q^{\pm}(f)(t)\|_{L^1}<\infty$ for any $T>0$. Then we conclude with simple calculation that $Q(g)(t,v)=Q(f)(t,v)$ for a.e. $(t,v)\in [0,\infty)\times \R^3$.
This implies that
$
g(t,v)=f_0(v)+\int_{0}^{t}Q(g)(\tau,v){\rm d}\tau
$
for a.e. $(t,v)\in [0,\infty)\times \R^3$.
Since $\sup\limits_{0\le t\le T}|Q^{\pm}(g)(t,v)|<\infty$ for any $T\in [0,\infty)$ and any $v\in \R^3$,
it follows that
$t\mapsto \int_{0}^{t}Q(g)(\tau,v){\rm d}\tau$ is continuous on $[0,\infty)$. Combining this with $g\in C([0,\infty)\times \R^3)$ and Lemma \ref{Lemma H} implies that there exists a null set $Z_0\subset \R^3$ such that for all $(t,v)\in [0,\infty)\times(\R^3\setminus Z_0)$,  $g(t,v)=f_0(v)+\int_0^t Q(g)(\tau,v)\dd \tau$. From Proposition \ref{continuity about t,v} and $g\in C([0,\infty)\times \R^3)$ we see that $Q(g)=\wt{Q}^{+}(g)-g\wt{L}(g)\in C([0,\infty)\times \R^3)$.
This together with $g\in C([0,\infty)\times \R^3)$ and the fact that $\sup\limits_{0\le t\le T}|Q^{\pm}(g)(t,v)|<\infty$ for any $T\in [0,\infty)$ imply that
\beas
g(t,v)=f_0(v)+\int_{0}^{t}Q(g)(\tau,v){\rm d}\tau \quad \forall (t,v)\in [0,\infty)\times \R^3.
\eeas
Next from (\ref{about infty}) and
$\sup\limits_{t\in [0,\infty)}\|g(t)\|_{L_3^1}=\sup\limits_{t\in [0,\infty)}\|f(t)\|_{L_3^1}<\infty$,
we can make a similar calculation as done in the proof of Proposition \ref{continuity about t,v} to obtain that for every fixed $v\in \R^3$, the function $t\mapsto Q(g)(t,v)$ is continuous on $[0,\infty)$. It follows that
\beas
\fr{\p}{\p t}g(t,v)=Q(g)(t,v)\qquad \forall\,(t,v)\in [0,\infty)\times{\R}^3.
\eeas
Then by the continuity of $g$ and (\ref{about infty}) it is easily seen that $g$ satisfies the $L^{\infty}$ estimate (\ref{estimate infty}).
 Thus from Definition \ref{definition of solution} we see that $g$
 is a  solution of Eq.(\ref{Equation}) with initial datum $0\le g(0,\cdot)=f_0\in L^1_3({\bR})\cap L_{2}^{\infty}({\bR}) \cap C(\R^3)$.
To keep the same notation of solution as in the theorem, we now rewrite $g$ as $f$
and thus we have proved the global in time existence of a classical solution $0\le f\in L^{\infty}([0,\infty); L^1_3({\bR})\cap L^{\infty}({\bR}) \cap C(\R^3))$ of Eq.(\ref{Equation}) with the initial datum $0\le f_0\in L^1_3({\bR})\cap L^{\infty}_3({\bR}) \cap C(\R^3)$.

{\bf Step 4.} We prove that the classical solution $f$ obtained in {\bf Step 3} is unique
in $L^{\infty}([0,\infty); L^1_3({\bR})\cap L^{\infty}({\bR}) \cap C(\R^3))$.
But this follows easily from (\ref{Theorem proposition}).

To finish the proof of Theorem \ref{main results}, we now need only to finish the

{\bf Proof of Proposition {\ref{continuity about t,v}}.}
 In the following we denote $C_{*,*,...}$
to be any finite and positive constants that depend only on their arguments $*,*,...$, and they may have
different values in different lines.\, We divide the proof into four steps.

{\bf Step 1.} Let $C_{L^\infty}=C_{L^\infty}(f_0)$ be defined in (\ref{about infty}). Then, $\sup\limits_{t\le 0}\|f(t)\|_{L^{\infty}}\le C_{L^\infty}$.
Therefore
\beqa\label{Icase0}
&&
\|f(s)-f(t)\|_{L^1_2}\le \int_{s}^{t}{\rm d}\tau
\int_{{\R}^3}\big|Q(f)(\tau,v)\big|(1+|v|^2){\rm d}v
\no\\
&&\le  2(1+2C_{L^\infty})\int_{s}^{t}{\rm d}\tau
\int_{{\R}^3\times {\R}^3}|v-v_*|f(\tau,v)f(\tau,v_*)(1+|v|^2+|v_*|^2){\rm d}v{\rm d}v_*
\no\\
&&\le 2(1+2C_{L^\infty})\int_{s}^{t}\|f(\tau)\|_{L^1_3}^2{\rm d}\tau
\le  C_{f_0}|t-s|\qquad \forall\, 0\le s<t<\infty\eeqa
where we used $\sup\limits_{t\ge 0}\|f(t)\|_{L^1_3}\le \max\{1,C_1\}\|f_0\|_{L_3^1}$ (see
Proposition \ref{L3}).

{\bf Step 2.} Fix any $1\le R<\infty$.
Let $v\in \R^3$ satisfy $|v|<R $. For any $0<\delta <\fr{1}{2}$ and any $t,s\in [0,\infty)$, we compute
\beqa\label{Icase1}
&&\big|\wt{Q}^{+}(f)(t,v)-\wt{Q}^{+}(f)(s,v)\big|
\no\\
&&\le
(1+C_{L^{\infty}})\iint_{\bRS}B(v-v_*,\si)\Big(f(s, v_*')|f(s, v')-f(t, v')|+ f(t,v')|f(t, v_*')-f(s, v_*')|\no\\
&&\quad +
f(t, v')f(t, v_*')|f(s, v_*)-f(t,v_*)|\Big)\dd \si \dd v_*
\no\\
&&\le
b(1+C_{L^{\infty}}) \iint_{\bRS}|v-v_*|\Big(f(s, v_*')|f(s,v')-f(t,v')|+ f(t,v')|
f(t,v_*')-f(s,v_*')|\no\\
&& \quad +
f(t,v')f(t,v_*')|f(s,v_*)-f(t,v_*)|\Big)\dd \si \dd v_*
\no\\
&&\le
b(1+C_{L^{\infty}})\Big(2C_{L^{\infty}}\dt^{-4}4\pi\la v\ra \|f(s)-f(t)\|_{L^1_1}
+2C_0 C_{L^{\infty}}\la v\ra (\|f_0\|_{L^1}\|f_0\|_{L^1_2})^{1/2}\dt \no\\
&&\quad +C_{L^{\infty}}^2 4\pi \la v\ra \|f(s)-f(t)\|_{L^1_1}
\Big)
\no\\
&&\le C_{f_0, R}\Big(\dt^{-4}\|f(s)-f(t)\|_{L^1_2}+ \dt +\|f(s)-f(t)\|_{L^1_2}
\Big)
\eeqa
where we have used Proposition \ref{proposition infty} to obtain
\beas
&&\iint_{\bRS}|v-v_*|f(s, v_*')|f(s, v')-f(t, v')|\dd \si \dd v_*
=\iint_{\bRS}|v-v_*|f(s,v')|f(s, v_*')-f(t, v_*')|\dd \si \dd v_*
\\
&&=\iint_{{\R}^3\times{\mathbb S}^2}1_{\{\sin(\theta/2)\ge \dt\}}|v-v_*||f(s,v')-f(t,v')|
f(s, v_*'){\rm d}\si{\rm d}v_*\no\\
&&\quad +\iint_{{\R}^3\times{\mathbb S}^2}1_{\{\sin(\theta/2)< \dt\}}|v-v_*||f(s,v')-f(t,v')|
f(s, v_*'){\rm d}\si{\rm d}v_*
\\
&&\le C_{L^{\infty}}\iint_{{\R}^3\times{\mathbb S}^2}1_{\{\sin(\theta/2)\ge \dt\}}|v-v_*|
|f(s,v')-f(t,v')|{\rm d}\si{\rm d}v_*
\no\\
&&\quad +C_{L^{\infty}}\iint_{{\R}^3\times{\mathbb S}^2}1_{\{\sin(\theta/2)< \dt\}}|v-v_*|
f(s, v_*'){\rm d}\si{\rm d}v_*
\\
&&=C_{L^{\infty}}\iint_{{\R}^3\times{\mathbb S}^2}\fr{1}{\sin^3(\theta/2)}1_{\{\sin(\theta/2)\ge \dt\}}\fr{|v-v_*|}{\sin(\theta/2)}
|f(s,v_*)-f(t,v_*)|{\rm d}\si{\rm d}v_*
\\
&&\quad +C_{L^{\infty}}\iint_{{\R}^3\times{\mathbb S}^2}\fr{1}{\cos^3(\theta/2)}1_{\{\sin(\theta/2)< \dt\}}|v-v_*|
f_*(s){\rm d}\si{\rm d}v_*
\\
&&=C_{L^{\infty}}\dt^{-4}\iint_{{\R}^3\times{\mathbb S}^2}|v-v_*|
|f(s,v_*)-f(t,v_*)|{\rm d}\si{\rm d}v_*
+C_0 C_{L^{\infty}} \iint_{{\R}^3\times{\mathbb S}^2}1_{\{\sin(\theta/2)< \dt\}}|v-v_*|
f(s,v_*){\rm d}\si{\rm d}v_*
\\
&&\le C_{L^{\infty}}\dt^{-4}4\pi\la v\ra \|f(s)-f(t)\|_{L^1_1}
+C_0 C_{L^{\infty}}\la v\ra (\|f_0\|_{L^1}\|f_0\|_{L^1_2})^{1/2}\dt
\eeas
and
\beas
&&\iint_{\bRS}|v-v_*|f(t, v')f(t, v_*')|f(s, v_*)-f(t, v_*)|\dd \si \dd v_*
\le C_{L^{\infty}}^2 4\pi \la v\ra \|f(s)-f(t)\|_{L^1_1}.
\eeas
Minimizing the right hand side of (\ref{Icase1}) with respect to $\delta >0$ gives
$$
\sup_{|v|<R}\big|\wt{Q}^{+}(f)(t,v)-\wt{Q}^{+}(f)(s,v)\big|
\le C_{f_0, R}(\|f(s)-f(t)\|_{L^1_2})^{1/5}\qquad \forall\, s,t\in [0,\infty).$$
Following the similar argument we also have
\beqa\label{Icase2}
\sup_{|v|<R}\big|\wt{L}(f)(t,v)-\wt{L}(f)(s,v)\big|
\le C_{f_0, R}(\|f(s)-f(t)\|_{L^1_2})^{1/5}\qquad \forall\, s,t\in [0,\infty).\eeqa
From (\ref{Icase0}), (\ref{Icase1}) and (\ref{Icase2}) we can see that for any
$v\in \R^3$ with $|v|<R $, the functions  $t\mapsto\wt{Q}^{+}(f)(t,v)$ and $t\mapsto\wt{L}(f)(t,v)$ are
both uniformly continuous on $[0,\infty)$.

{\bf Step 3.} Fix any $0<T<\infty, t\in [0,T]$. For any $h\in \R^3$ with $|h|<1$,
since $B$ satisfies (\ref{about v-v*}), it is immediately verified that the velocity-translation $g(t,v):=f(t,v+h)$ is a mild solution to Eq.(\ref{Equation}) with
the initial datum $g_0(v)=f_0(v+h)$.
For $0<\delta <\fr{1}{2}$ and $|v|<R$, using Proposition \ref{proposition infty} and (\ref{Theorem proposition}) and doing a similar calculation as in {\bf Step 2} we have
\beas&&
\sup_{t\in [0,T], |v|<R}\big|\wt{Q}^{+}(f)(t,v+h)-\wt{Q}^{+}(f)(t,v)\big|=\sup_{t\in [0,T], |v|<R}\big|\wt{Q}^{+}(g)(t,v)-\wt{Q}^{+}(f)(t,v)\big|\\
&&
\le b(1+C_{L^{\infty}})\big(C_{f_0,R,T}\dt^{-4} \|f_0(\cdot+h)-f_0\|_{L^1_2}
+C_{f_0,R}\dt+C_{f_0, R, T}\|f_0(\cdot+h)-f_0\|_{L^1_2}
\big).\eeas
Minimizing the right-hand with respect to $\delta>0$ gives
\beqa\label{Icase3}
\sup_{t\in [0,T], |v|<R}\big|\wt{Q}^{+}(f)(t,v+h)-\wt{Q}^{+}(f)(t,v)\big|
\le C_{f_0, R, T}(\|f_0(\cdot+h)-f_0\|_{L^1_2})^{1/5}\qquad \forall\, |h|<1.
\eeqa
Following a similar argument we also have
\beqa\label{Icase4}
\sup_{t\in [0,T], |v|<R}\big|\wt{L}(f)(t,v+h)-\wt{L}(f)(t,v)\big|
\le C_{f_0, R, T}(\|f_0(\cdot+h)-f_0\|_{L^1_2})^{1/5}\qquad \forall\, |h|<1.
\eeqa
  From (\ref{go to zero}), (\ref{Icase3}) and (\ref{Icase4})
we see that for any  fixed  $T\in (0,\infty)$ and any $t\in[0, T]$,
the functions $v\mapsto \wt{Q}^{+}(f)(t,v)$ and $v\mapsto \wt{L}(f)(t,v)$ are
both uniformly continuous in $|v|<R$.

{\bf Step 4.} Finally we prove that $(t,v)\mapsto \wt{Q}^{+}(f)(t,v)$ is continuous on $[0,\infty)\times \R^3.$
Fix any $(t,v)\in [0,\infty)\times \R^3$. We have
\beqa\label{COT}
\big|{Q}^{+}(f)(s,u)-{Q}^{+}(f)(t,v)\big|
\le \big|{Q}^{+}(f)(s,u)-{Q}^{+}(f)(s,v)\big|
+\big|{Q}^{+}(f)(s,v)-{Q}^{+}(f)(t,v)\big|.
\eeqa
From {\bf Step 2} and {\bf Step 3} (taking for instance $R=1+|v|$) we see that
for $\forall \vep>0$ there exists
$0<\dt=\dt_{t,v}<1$ such that if $|s-t|<\dt$ (with $s\ge 0$) and $|u-v|< \dt$,
then $\big|\wt{Q}^{+}(f)(s,v)-\wt{Q}^{+}(f)(t,v)\big|<\vep/2$ and
$\big|\wt{Q}^{+}(f)(s,u)-\wt{Q}^{+}(f)(s,v)\big|<\vep/2$. From (\ref{COT}) we
conclude that $\wt{Q}^{+}(f)(\cdot,\cdot)$ is continuous at $(t,v)$.

This finishes the proof of Proposition {\ref{continuity about t,v}} and thus the proof of Theorem \ref{main results} is completed.$\hfill\Box$
\vskip4mm

{\bf Acknowledgment}. This work was supported by National Natural Science
Foundation of China Grant No. 11171173.

 \end{document}